\documentclass[a4paper, 8pt]{article}

\usepackage{thesis}
\usepackage{stmaryrd}

\usepackage[margin=2.5cm]{geometry}

\DeclareFontFamily{U}{mathx}{\hyphenchar\font45}
\DeclareFontShape{U}{mathx}{m}{n}{<-> mathx10}{}
\DeclareSymbolFont{mathx}{U}{mathx}{m}{n}
\DeclareMathAccent{\widebar}{0}{mathx}{"73}
\newcommand{\proet}{\mathrm{pro\acute{e}t}}
\newcommand{\et}{\mathrm{\acute{e}t}}
\newcommand{\fet}{\mathrm{f\acute{e}t}}

\numberwithin{equation}{section}

\usepackage{amsthm}
\theoremstyle{definition}
\newtheorem{defin}{Definition}[subsection]
\newtheorem{rem}[defin]{Remark}
\newtheorem{ex}[defin]{Example}
\newtheorem{constr}[defin]{Construction}
\newtheorem{notat}[defin]{Notation}
\newtheorem{obs}[defin]{Observation}
\newtheorem{fact}[defin]{Fact}

\newtheorem*{conj}{Conjecture}

\theoremstyle{plain}
\newtheorem{lem}[defin]{Lemma}
\newtheorem{prop}[defin]{Proposition}
\newtheorem{thm}[defin]{Theorem}
\newtheorem{cor}[defin]{Corollary}

\theoremstyle{remark}
\newtheorem{claim}[defin]{Claim}

\usepackage{url}
\usepackage{hyperref}
\usepackage{cite}
\usepackage{adjustbox}
\renewcommand{\maketitle}{%
	\begin{center}
		{\Large\bfseries Borel's computation of the ranks of $K$-theory revisited\par}
		\vspace{0.5em}
		{\normalsize Jan Ole Hesse\par}
	\end{center}
}

\begin{document}
	 \maketitle
	\begin{abstract}
		We reprove Borel's computation of the rank of the $K$-groups of rings of algebraic integers $\mathcal{O}_F$ for certain number fields, e.g., $F = \QQ(\mu_{l^k})$ for $l$ an odd regular prime number. The idea is to recover a simple formula for $l$-adic étale cohomology groups of $\mathcal{O}_F$ in terms of Iwasawa theoretic objects without using Borel's result as an input.
	\end{abstract}
	\section{Introduction}
	For a ring $A$ denote by $K_*(A)$ Quillen's higher $K$-groups defined as the homotopy groups of the $K$-theory spectrum $K(A)$. In this article we give a new proof using classical techniques of
\begin{thm}
	Set $F = \QQ(\mu_{l^k})$ for $l$ a regular odd prime and let $k \in \NN$. Denote by $r_1(F)$ and $r_2(F)$ the number of real and complex places of $F$, respectively. For $i \geq 2$, we have
	\begin{equation}\label{I-2}
		\dim_\QQ K_i(\CalO_F)_\QQ = \begin{cases}
			r_1(F) + r_2(F) &\text{, if } i \equiv 1 \mod 4 \\
			r_2(F) &\text{, if } i \equiv 3 \mod 4 \\
			0 &\text{, else}
		\end{cases}
	\end{equation}
\end{thm}
In \cite{Bor74} Borel gives a differential geometric proof for \eqref{I-2} for all number fields $F$. Our argument is more algebraic, as it proceeds via an étale cohomology computation.

\subsubsection{Reducing to étale cohomology computation assuming finite generation}
Quillen proved that the groups $K_*(\mathcal{O}_F)$ are degree-wise finitely generated (see \cite{Qui72} and \cite{Mur} for a summary of the proof). Assuming this finite generation result, a standard $\mathrm{lim}^1$-sequence computation shows that for any odd prime $l$:
\begin{equation*}
	\dim_\QQ K_*(\CalO_F)_\QQ = \dim_{\QQ_l} \pi_*(K(\mathcal{O}_F)^\wedge_l)_\QQ
\end{equation*}
Set $R := \mathcal{O}_F[1/l]$. By Quillen's localization Theorem we have a fiber sequence of spectra
\begin{equation}\label{I-5}
	\bigoplus\limits_{\mathfrak{p} \mid l} K(\mathcal{O}_F/\mathfrak{p})^\wedge_l \to K(\CalO_{F})^\wedge_l \to K(R)^\wedge_l
\end{equation}
Since $\mathcal{O}_F/\mathfrak{p}$ is a finite field, $\pi_*(K(\mathcal{O}_F/\mathfrak{p})^\wedge_l) = 0$ for $* \geq 1$ by Quillen's computation of the $K$-theory of finite fields. Thus the long exact sequence on homotopy groups associated to \eqref{I-5} shows
\begin{equation*}
	\dim_{\QQ_l}\pi_*(K(\mathcal{O}_F)^\wedge_l)_\QQ = \dim_{\QQ_l} \pi_*(K(R)^\wedge_l)_\QQ
\end{equation*}
Now recall the following theorem:
\begin{thm}[{\cite[Thm.7.13]{CM21}}]
	Let $X$ be a qcqs scheme of finite Krull dimension such that
	\begin{enumerate}
		\item $l \in \mathcal{O}_X^\times$.
		\item the Galois cohomological dimension of the residue fields is bounded by some constant $C$.
	\end{enumerate}
	Denote by $K^\et: X^\op_\et \to \mathrm{Sp}$ the étale sheafification of the $K$-theory presheaf studied in \cite{CM21}. Under these assumptions
	\begin{equation*}
		K(X)\lcomp \to K^\et(X)\lcomp
	\end{equation*}
	induces an equivalence on homotopy groups in degrees $\geq \max\{C-2, 0\}$.
\end{thm}
The étale cohomological dimension of finite fields is $1$, since their Galois group is topologically cyclic. Away from $2$, the étale cohomological dimension of global fields is $2$ \cite[Cor. I.4.18]{Mil06}. So the above theorem tells us that
\begin{equation*}
	\pi_*(K(R)\lcomp) \xrightarrow{\simeq} \pi_*(K^\et(R)\lcomp)
\end{equation*}
is an equivalence in degrees $* \geq 1$. We remind the reader of the following result:
\begin{thm}\label{I-7}
	Let $X$ be a qcqs scheme with a uniform bound on the virtual $l$-Galois cohomological dimension of its residue fields with $l \in \CalO_X(X)^\times$. There is a spectral sequence
	\begin{equation*}
		E^2_{p, q} = \RH^{-p}_\et(X, \ZZ_l(q/2)) \Rightarrow \pi_{p+q}(K^\et(X)\lcomp)
	\end{equation*}
	in homological grading with the convention that $\ZZ_l(q/2) = 0$ if $2 \nmid q$ or $q < 0$.
\end{thm}
The above is a the combination of
\begin{enumerate}
	\item \cite[Cor.6.12]{CM21}, which computes the étale local homotopy sheaves of $K^\et/l^k$ to be $\ZZ/l^k(q/2)$ in degree $q$.
	\item By the main result of \cite{CM21} under the above assumptions on $X$ the étale local Postnikov tower on $K^\et/l^k$ is convergent.
	\item $\lim_k R\Gamma_\et(R, \ZZ/l^k(q/2)) = R\Gamma_\et(R, \ZZ_l(q/2))$ by definition.
\end{enumerate}
This shows that the spectrum $(K^\et(R))^\wedge_l$ admits a convergent filtration with associated graded $R\Gamma_\et(R, \ZZ_l(q/2))$ giving rise to the spectral sequence \ref{I-7}.

In the case $X = \mathrm{Spec}(R)$, the $E^2$ page is so sparse that it yields the following formulas
\begin{align}\label{I-6}
	\dim_\QQ K_{2i-1}(\mathcal{O}_F)_\QQ = \dim_{\QQ_l} \mathrm{H}_\et^1(R, \ZZ_l(i))_\QQ & &\dim_\QQ K_{2i-2}(\mathcal{O}_F)_\QQ &= \dim_{\QQ_l} \mathrm{H}_\et^2(R, \ZZ_l(i))_\QQ
\end{align}
for $i \geq 2$.
\subsubsection{Content of this paper}
	The main inspiration for this project was Dwyer and Mitchell's paper \cite{DM98} and the formulas for étale cohomology groups in terms of Iwasawa-theoretic objects in \cite[Prop.4.15+Thm.1.7]{DM98}. The main goal of this paper is to extract the same formula without assuming Borel's result. We achieve this by working derivedly and not just on the level of cohomology groups.
	\begin{enumerate}
		\item In Sections \ref{A1} and \ref{A3} we introduce the basics of Iwasawa theory and tools to compute derived global sections of $\ZZ_l(i)$ such as proétale descent. The approach presented here uses condensed mathematics as the formal basis.
		\item In Section \ref{C1} we extract a formula for $R\Gamma_\proet(R, \ZZ_l(i))$ in terms of the derived $\mathrm{Hom}$ of certain Iwasawa-modules $M_\infty$ and $A_\infty$. The argument follows that of \cite{DM98} - descending along the $l$-cyclotomic extension.
		\item In Section \ref{C2} we compute the ranks of the Iwasawa-theoretic object $M_\infty$ over a certain completed group ring $\Lambda_F$ using class field theory. These ranks turn out to be the main contribution to the ranks predicted by Borel's result via the formula found in section \ref{C1}.
		\item In Section \ref{C3} we prove a simple Iwasawa-theoretic criterion for when Formula \eqref{I-2} is true in degree $i$ for each $i \geq 2$. Up until this point the only non-algebraic input is Quillen's proof of finite generation. Afterwards, we collect cases in which these conditions admit a proof via Iwasawa-theory. The easy cases include $F = \QQ(\mu_{l^k})$ for $l$ an odd regular prime. The more complicated cases, e.g., arbitrary totally real fields, can be tackled using hard theorems in Iwasawa theory, like the main conjecture of Iwasawa theory.
	\end{enumerate}
	The appendix consists of proofs for statements, for which we could not find a reasonably close reference.
	\newline\newline
	\textit{Acknowledgments}: First and foremost, the author would like to thank his master's advisor Markus Land for taking a lot of his time for discussions and answering questions on the contents of this former master's thesis. The author is very grateful for such a great learning experience.
	
	The author also wants to thank Emma Brink for explaining how condensed mathematics helps formalize continuous group actions, and making suggestions for how to prove Lemma \ref{Cond-6} in Section \ref{A1}.
	
	The author would also like to thank Ko Aoki for expressing his interest in this project, discussions related to it, and encouraging him to produce a readable account of what the author originally did in his master's thesis.
	
	For helpful discussions related to the contents of Appendix \ref{P}, the author thanks Semen Slobodaniuk.
	
	Many thanks also go to Jonathan Clivio, Anton Lin and Philippe Vollmuth for proofreading different parts of this paper.
	
	For advice on submitting this preprint, the author wants to thank his supervisor Ryomei Iwasa.
	
	The author finished this project while he was supported by the European Union’s Horizon 2020 research and innovation programme under the European Research Council (ERC) grant agreement No. 101164148 and by the Danish National Research Foundation through the Copenhagen Centre for Geometry and Topology (DNRF151).
	\newpage
	\setcounter{tocdepth}{2}
	\tableofcontents
	\section{Setup}
		\subsection{Notation}\label{A0}
			For an abelian group $A \in \Ab$, we write, for $k, l \in \ZZ$ $A[k] \coloneqq \ker(A \xrightarrow{k} A)$ and $A[l^\infty] \coloneqq \cup_{r \in \NN} A[l^r] = l-\text{power torsion}$. We also write $A_\QQ$ for $A \otimes \QQ$ sometimes.
We assume basic knowledge of algebraic geometry at the level of \cite[Ch.I+II]{Har77b}.
\newline\newline
\noindent\textbf{Number-theoretic setup and notation}

As usual in number theory, we have to introduce a substantial amount of notation.

Fix a number field $F$ and let $\CalO_F$ be its integral closure of $\ZZ$ in $F$. Then we have the following picture of Galois extensions:
\[\begin{tikzcd}
	& & F_\infty \coloneqq F(\mu_{l^\infty}) \arrow[dl, dash, "\Gamma"]\\
	& F_0 \coloneqq F(\mu_l) \arrow[dl, dash, "\Delta"] & \\
	F_{-1} \coloneqq F \arrow[uurr, dash, bend left, "\Gamma_F"] & &
\end{tikzcd}\]
Via the canonical injective \footnote{Because $F_\infty$ is the composite field of $F$ and $\QQ(\mu_{l^\infty})$} continuous restriction map
\begin{equation*}
	c: \Gal(F_\infty/F) \to \Gal(\QQ(\mu_{l^\infty})/\QQ) \simeq \Aut(\mu_{l^\infty}) \simeq \ZZ_l^\times
\end{equation*}
we may think of $\Gamma_F$ as a subgroup of $\ZZ_l^\times$. We call $c$ the cyclotomic character associated to $F$. This map identifies $\Delta$ as a subgroup of $\mu_{l-1} \subset \ZZ_l^\times$ and $\Gamma$ as a finite-index open subgroup of the $1$-units $1+l\ZZ_l$ in $\ZZ_l^\times \simeq \mu_{l-1} \times (1+l\ZZ_l)$. Now all finite index/open subgroups of $\ZZ_l^\times$ are a product of a subgroup $\subset \mu_{l-1}$ and a subgroup of the $1$-units $1+l\ZZ_l$. This way one obtains the splitting $\Gamma_F \simeq \Delta \times \Gamma$.

\begin{defin}
	We call the canonical inclusion $\omega: \Delta \to \mu(\ZZ_l) \subset \ZZ_l^\times$ the Teichmüller character. $\omega$ generates the character group $\mathrm{Map}_\Ab(\Delta, \QQ/\ZZ)$ \cite[Satz 2.7]{Neu06}.
\end{defin}

We also pick a prime $\mathfrak{p} \nmid l$, which stays inert throughout the tower $F_\infty/F$. Such a prime exists by Lemma \ref{F-3}. It has the property that the residue fields $\FF_n \simeq R_n/\mathfrak{p}R_n$ give rise to a field extension $\FF_\infty/\FF$, having the same Galois theory/group as $F_\infty/F$.
\newline\newline
\textbf{The $\Gamma_F$-/Iwasawa modules of interest}\newline
For $n \in \NN \cup \{\infty\}$:
\begin{align*}
	F_n &\coloneqq F_\infty^{\Gamma^{l^n}} \\
	R_n &\coloneqq \text{integral closure of } R \text{ in } F_n = \CalO_{F_n}[1/l] \\
	S_n &\coloneqq \text{ \# of primes in } \CalO_{F_n} \text{ above } l = \#\{\Spec(\CalO_{F_n}) - \Spec(R_n)\} \\
	A_n &\coloneqq \Pic(\CalO_{F_n})[l^\infty] \\
	A'_n &\coloneqq \Pic(R_n)[l^\infty] \\
	M_n &\coloneqq \text{Galois group of the maximal abelian $l$-profinite extension } L_n \text{ over } F_n \text{ ramified only over $l$}
\end{align*}
All the above sets and groups naturally promote to $\Gal(F_n/F)$-sets or modules, respectively. For $M_n$ the action comes from the short exact sequence
\begin{equation*}
	0 \to M_n \to \mathrm{Gal}(L_n/F) \to \mathrm{Gal}(F_n/F) \to 0
\end{equation*}
by lifting elements $\gamma \in \mathrm{Gal}(F_n/F)$ along the surjection and acting by conjugation on the normal subgroup $M_n$. Since $M_n$ is abelian this action does not depend on the choice of lift.

Since all of the above modules are $l$-profinite they become $\ZZ_l[\Gal(F_n/F)]$-modules. Moreover, they assemble into $\NN^\op$-indexed systems which naturally become modules over the ring
\begin{equation*}
	... \to \ZZ[\Gal(F_n/F)] \to \ZZ[\Gal(F_{n-1}/F)] \to ...
\end{equation*}
in $\Fun(\NN^\op, \Ab)$ with the pointwise tensor product. Define
\begin{align*}
	\Lambda_{F_n} &\coloneqq \ZZ_l[[\Gal(F_\infty/F_n)]] \coloneqq\lim_{m \in \NN^\op, m \geq n} \ZZ_l[\Gal(F_m/F_n)] \;\;\; \Lambda := \Lambda_{F_0}
\end{align*}
and endow it with the profinite topology.
Hence, the limits
\begin{align*}
	\ZZ_l[S_\infty] \coloneqq \lim\limits_{n \in \NN} \ZZ_l[S_n]
	& & A_\infty \coloneqq \lim\limits_{n \in \NN^\op} A_n
	& & A_\infty' \coloneqq \lim\limits_{n \in \NN^\op} A_n'
	& & M_\infty \coloneqq \lim\limits_{n \in \NN^\op} M_n
\end{align*}
become continuous $\Lambda_F$-modules. As limits over profinite abelian groups they are even underlyingly compact.

All of the modules $M_\infty, A_\infty, A'_\infty$ also arise as Galois groups of abelian extensions $T/F_\infty$. E.g., $M_\infty \simeq \Gal(L_\infty/F_\infty)$, where $L_\infty = \colim_n L_n$.
\newline\newline
\textbf{Eigenspaces with respect to the Teichm\"uller character}\newline
Consider the orthogonal idempotents
\begin{equation*}
	\epsilon_i \coloneqq \abs{{\Delta}}^{-1}\sum\limits_{g \in {\Delta}} \omega^{-i}(g)g\in \Lambda_F
\end{equation*}
where we used that $\mathrm{gcd}(\abs{\Delta}, l) = 1$, and that $\ZZ_l$ contains all $\abs{\Delta}$-roots of unity. Clearly, $\epsilon_i$ only depends on the residue class $k \mod \abs{\Delta}$. For $N$ a $\Lambda_F$-module $\epsilon_i N = N^{\omega^i} \coloneqq \{x \in N \vert \forall \delta \in \Delta: \delta.x = \omega^i(\delta)x\}$ is the $\omega^i$-eigenspace. Applied to $\Lambda_F$ this gives a splitting 
\begin{equation*}
	\Lambda_F \xrightarrow{\simeq} \prod_\Delta \Lambda \simeq \ZZ_l[\Delta] \otimes_{\ZZ_l} \Lambda
\end{equation*}
as $\Lambda$-algebras and therefore on module categories 
\begin{equation*}
	\mathrm{Mod}_{\Lambda_F} \simeq \mathrm{Mod}_\Lambda^\Delta
\end{equation*}
The tensor product over $\Lambda_F$ corresponds to Day convolution tensor product on the right. Equivalently, for $N, M$ $\Lambda_F$-modules we have an equivalence of $\Lambda$-modules:
\begin{equation}\label{S-2}
	\epsilon_n(N \otimes_{\ZZ_l} M) \simeq \bigoplus_{i + j = n} \epsilon_iN \otimes_{\ZZ_l} \epsilon_jM
\end{equation} 
\newline\newline
Any continuous character $\chi: \Gamma_F \to \ZZ_l^\times$ gives rise to a $\chi$-twist of $\ZZ_l$ which we denote by $\ZZ_l(\chi)$, making $\ZZ_l$ into a cyclic $\Lambda_F$-module.
\begin{defin}
	Denote by $\ZZ_l(i)$ $\ZZ_l$ with the $\Gamma_F \subset \ZZ_l^\times$-action given by $g.x \coloneqq c(g)^i \cdot x$. For a $\Lambda_F$-module $N$ we write $N(i)$ for $N \otimes_{\ZZ_l} \ZZ_l(i)$ with the diagonal $\Gamma_F$-action.
\end{defin}
Of course, we have the following formulas for computing the $\Delta$-eigenspaces
\begin{equation}\label{S-1}
	\epsilon_i \ZZ_l(j) \simeq \begin{cases}
		\ZZ_l(j) &\text{, if } i \equiv j \mod \abs{\Delta}\\
		0 &\text{, else}
	\end{cases}
\end{equation}
\textbf{Categorical notation}

By a category, we mean an $\infty$-categories as initially developed in \cite{Lur09} and \cite{Lur17}. Classical categories with hom-sets instead of homotopy types as mapping objects we refer to specifically as $1$-categories. We write $\Spc$ for the category of anima/homotopy types \cite[Def.1.2.16.1]{Lur09}. For $\mathcal{D}$ a category and $x, y \in \mathcal{D}$ objects we denote by $\mathrm{Map}(x, y)$ or $\mathrm{Map}_\mathcal{D}(x, y)$ depending on whether there is possible confusion or not. For $\mathcal{A}$ an abelian category, denote by $\mathcal{D}(\mathcal{A})$ its derived category \cite[Def.1.3.5.8]{Lur17}. If $\mathcal{D}$ is stable \cite[Def.1.1.1.9]{Lur17} we write $\mathrm{map}_\mathcal{D}(x, y)$ for the mapping spectrum (the unique lift of $\mathrm{Map}(x, \_)$ under the equivalence \cite[Prop.1.4.2.22]{Lur17}). We denote by $\_[k]$ the $k$-fold suspension/shift functor. In any stable category $\CalC$ with $\NN^\op$-indexed limits and finite colimit for $l$ a prime and $X \in \CalC$ we write
\begin{equation*}
	X\lcomp \coloneqq \lim_{k \in \NN^\op} \cofib(X \xrightarrow{l^k} X)
\end{equation*}
If $\mathcal{D}$ is enriched in a symmetric monoidal category $\mathcal{V}$ we write $y^x$ or $\underline{\mathrm{map}}(x, y) \in \mathcal{V}$ for the mapping object if it exists \cite[Def.4.2.1.28]{Lur17}. If we do not specify $\mathcal{V}$ then $\mathcal{D} = \mathcal{V}$ implicitly. We write $\mathrm{Pr}^L$ for the symmetric monoidal category of presentable categories as defined in \cite[Def.5.5.3.1]{Lur09} and the symmetric monoidal structure in \cite[Prop.4.8.1.15]{Lur17} and colimit preserving maps defined in and by $\mathrm{Pr}^R$ its opposite. For $\mathcal{C}$ a monoidal category denote by $\mathrm{Alg}(\mathcal{C})$ its category of associative algebras following \cite[Notat.4.1.1.8]{Lur17} and if $\mathcal{C}$ is symmetric monoidal denote by $\mathrm{CAlg}(\mathcal{C})$ the category of commutative algebras \cite[Def. 2.1.3.1]{Lur17}. For both $A \in \mathrm{Alg}(\mathcal{C})$ and $A \in \mathrm{CAlg}(\mathcal{C})$ denote by $\mathrm{Mod}_A(\mathcal{C})$ its category of left modules of the underlying associative algebra ($\mathrm{LMod}_A(\mathcal{C})$ in the notation of \cite[Def.4.2.1.13]{Lur17}).

		\subsection{Homological properties of the Iwasawa algebra}\label{A1}
			\subsubsection{Classical theory of finitely generated $\Lambda$-modules}
Iwasawa theory can be described as the study of modules such as $A_\infty$ and $M_\infty$ over the so-called Iwasawa algebra $\Lambda$ arising from invariants along a Galois extension with Galois group isomorphic to $\ZZ_l$. The main example is the cyclotomic extension $F_\infty/F_0$. In this subsection we summarize the necessary results we need in Section \ref{C1}.
\begin{rem}\label{1-8}
	We quickly review the results in \cite[Ch.13]{Was97} that are important for this section:
	\begin{enumerate}
		\item $\Lambda \simeq \ZZ_l[[T]]$ sending a topological generator $\gamma \in \Gamma$ to $T+1$. Hence, $\Lambda$ is a complete local regular ring with maximal ideal $\mathfrak{m} = (T, l)$ of global dimension $2$. Moreover, its height $1$ primes are $(l)$ and $(P)$ for $P$ an irreducible distinguished polynomial, i.e., a monic polynomial in $T$ such that $P \equiv T^{\deg(P)} \mod l$. The key input for proving this is the Weierstraß preparation Lemma:
		\begin{prop}
			Every $g \in \ZZ_l[[T]]$ can be uniquely written as
			\begin{equation*}
				g = UPl^\nu
			\end{equation*}
			with $\nu \in \NN$, $U \in \ZZ_l[[T]]^\times$ and $P \in \ZZ_l[T]$ a distinguished polynomial.
		\end{prop}
		Two further consequences are the following:
		\begin{enumerate}
			\item\label{1-8-1} Let $P$ be a distinguished polynomial. Then $\ZZ_l[T]/(P) \to \ZZ_l[[T]]/(P)$ is an isomorphism. As a consequence, for $P, Q$ distinguished polynomials, $P \mid Q$ in $\Lambda$ if and only if $P \mid Q$ in $\ZZ_l[T]$. In particular, $(P) = (Q)$ implies that $P$ and $Q$ differ by an element of $\ZZ_l^\times$.
			\item For $f, g \in \Lambda\setminus \Lambda^\times$ $\Lambda/(f, g)$ is finite if and only if $f$ and $g$ are coprime \footnote{Being coprime is a sensible condition, since $\ZZ_l[[T]]$ is a UFD by Auslander-Buchsbaum}. In particular $\Lambda/(f)$ for $f \in \Lambda$ is either $0$ or infinite.
		\end{enumerate}
		
		\item Let $M, N$ be finitely generated. We call a $\Lambda$-module map $M \to N$ a quasi-isomorphism, if it has finite kernel and cokernel. In this case we say $M$ is quasi-isomorphic to $N$. Note that a priori quasi-isomorphy is not a symmetric relation. It turns out that it is not symmetric unless we are dealing with torsion $\Lambda$-modules.
		\item Any finitely generated $\Lambda$-module $M$ admits a quasi-isomorphism into a module of the form
		\begin{equation*}
			\Lambda^r \oplus \bigoplus_i \Lambda/(P_i^{\lambda_i}) \oplus \bigoplus_j \Lambda/(l^{\mu_j})
		\end{equation*}
		with $P_i$ irreducible distinguished polynomials and $\lambda_i, \nu_j \in \NN$. Then we define
		\begin{align*}
			\mu(M) \coloneqq \sum\limits_j \mu_j & & \lambda(M) \coloneqq \sum\limits_i \lambda_i\deg(P_i) & & \mathrm{char}(M) \coloneqq (\prod_i P_i^{\lambda_i})
		\end{align*}
		These are usually called the $\mu$- and $\lambda$-invariant and the characteristic ideal of $M$, respectively. All of these are invariants of $M$ preserved by quasi-isomorphisms. Furthermore, they are multiplicative in short exact sequences of $\Lambda$-torsion modules.
		We call modules of the form $\Lambda$ or $\Lambda/(P^k)$, for $P$ a irreducible distinguished polynomial, or $\Lambda/(l^\nu)$ \textit{elementary modules}.
	\end{enumerate}
\end{rem}
\begin{ex}
	The modules $A_\infty, A'_\infty$ are torsion $\Lambda$-modules by \cite[Lemma 13.17]{Was97}. $M_\infty$ is not $\Lambda$-torsion but still finitely generated by \cite[Ch.XI, Sect.3]{NSW08}. Hence, the above theory applies to them.
\end{ex}
\begin{notat}
	Set $\omega_n \coloneq \gamma^{l^n}-1 = (T+1)^{l^n}-1$. In this notation $\Lambda/\omega_n \xrightarrow{\simeq} \ZZ_l[\ZZ/l^n]$.
\end{notat}
\begin{rem}
	Iwasawa used the structure theorem above for $A_\infty$ to provide a formula for $\abs{A_n}$ in terms of the invariants $\lambda$ and $\mu$: There exists $\nu \in \ZZ$ such that for $n >> 0$ \cite[Thm.13.13]{Was97} we have for $\abs{A_n} = l^{e_n}$
	\begin{equation*}
		e_n = \lambda(A_\infty)n + \mu(A_\infty)l^n + \nu
	\end{equation*}
\end{rem}
Usually, computing the $\mu$ and $\lambda$-invariants from this growth description is not practical as computing class groups is hard. However, we will use that the $\Lambda$-ranks of an Iwasawa module can be read off from the $\ZZ_l$-ranks of the "finite stages" $N/\omega_nN$:
\begin{lem}\label{2-7}
	Given a finitely generated $\Lambda$-module $N$. Then
	\begin{equation*}
		\rank_{\ZZ_l} N/\omega_nN = l^n \rank_\Lambda N + C_n
	\end{equation*}
	where $C_n$ is uniformly bounded in $n$.
	\begin{proof}
		Clearly, $\rank_{\Lambda} N = \rank_{\Lambda \otimes \QQ} N \otimes \QQ$ and $\rank_{\ZZ_l} N/\omega_nN = \rank_{\QQ_l} N \otimes \QQ/\omega_n(N \otimes \QQ)$. Since $\_ \otimes_\ZZ \QQ$ inverts quasi-isomorphisms we are reduced to proving the lemma for elementary modules. For $N$ a elementary torsion module, we see that $\rank_{\QQ_l} N \otimes \QQ/\omega_n(N \otimes \QQ) \leq \lambda(N)$ as a consequence of Remark \ref{1-8} item \ref{1-8-1}. For $N = \Lambda$ we obtain the desired result as well: Let $C_{l^n}$ be the cyclic group of order $l^n$. Then
		\begin{equation*}
			\rank_{\ZZ_l} \Lambda/\omega_n\Lambda \simeq \rank_{\ZZ_l} \ZZ_l[C_{l^n}] = l^n
		\end{equation*}
	\end{proof}
\end{lem}
\begin{obs}\label{H-5}
	Commutative algebra for finitely generated modules $M$ over $\Lambda_\QQ := \Lambda \otimes \QQ$ becomes even easier:
	\begin{enumerate}
		\item $M$ is isomorphic - not quasi-isomorphic - to a sum of elementary modules.
		\item $\mathrm{Ext}^1_{\Lambda_\QQ}(M, \Lambda)$ is abstractly isomorphic to the torsion of $M$ and $\mathrm{Ext}^0_{\Lambda_\QQ}(M, \Lambda)$ is free of rank the $\Lambda_\QQ$-rank of $M$: As $M$ decomposes into a sum of elementary modules it suffices to treat only elementary modules. For $M = \Lambda$ the statement is clear. For $M = \Lambda_\QQ/(P^k)$ for $P$ a distinguished polynomial this follows from the obvious projective resolution
		\begin{equation*}
			\Lambda_\QQ \xrightarrow{P^k} \Lambda_\QQ \to M
		\end{equation*}
		\item Assume $M = N \otimes \QQ$ for some finitely generated $\Lambda$-module $N$ then since $\mathrm{Spec}(\Lambda \otimes \QQ) \subset \mathrm{Spec}(\Lambda)$ contains all height one primes of $\Lambda$ except $l$ we can recover the characteristic polynomial of $N$ from that of $M$.
	\end{enumerate}
\end{obs}

\subsubsection{A well-behaved derived category of $\Lambda$-modules}
In Section \ref{C1} we derive a formula for étale cohomology of $R$ in terms of $\mathrm{Ext}$-groups of certain finitely generated modules over $\Lambda$ and $\Lambda_F$. The technique is pro-Galois descent along the pro-Galois cover $R_0 \to R_\infty$ and the finite étale cover $R \to R_0$ of order coprime to $l$. Formally, the content of the last sentence amounts to endowing the system $\{R\Gamma_\proet(R_n, \ZZ_l(i))\}_{n \in \NN^\op}$ with extra structure encoding the continuous action on cohomology groups. Condensed math \cite{CS26} allows us to formalize this structure as follows: 
\begin{enumerate}
	\item First, we define a derived category of modules over any profinite group $G$ $\mathcal{D}_\square(G)$ in \ref{Cond-17} which sees the topology of $G$. Discrete $G$-modules will embed into this category. This is basically an expansion of \cite[Sect.4.3]{BS14}.
	\item Second, we specialize to the case of $\Lambda$ and define $\mathcal{D}_\square(\Lambda)$. $\mathcal{D}_\square(\Lambda)$ embeds into $\mathcal{D}_\square(\ZZ_l)^{B\ZZ}$; a derived category of locally compact abelian groups over the topological ring $\ZZ_l$ studied by Clausen-Scholze \cite{CS26} with the additional datum of an action by the topological generator of $\Gamma$. Discrete $\Gamma$-representations and $\mathrm{Perf}(\Lambda)$ the category of perfect $\Lambda$-modules embed into this category (cf. \ref{Cond-18} resp. \ref{Cond-3}). The main example for this is étale cohomology with $\QQ_l/\ZZ_l$ and $\ZZ_l$-coefficients along a pro-Galois extension with Galois group $\Gamma$.
\end{enumerate}
\begin{notat}
	For $\mathcal{C}$ a Grothendieck site we denote by $\mathrm{Shv}(\mathcal{C})$ the hypercompletion \cite[Sect.6.5.3]{Lur09} of its $\infty$-topos of $\Spc$-valued sheaves \cite[Def.6.2.2.6]{Lur09}. $\mathcal{D} \in \mathrm{Pr}^L$ we denote by $\mathrm{PSh}(\mathcal{C}; \mathcal{D}) := \mathrm{Fun}(\mathcal{C}^\op, \mathcal{D})$ the category of $\mathcal{D}$-valued presheaves and $\mathrm{Shv}(\mathcal{C}; \mathcal{D}) := \mathrm{Shv}(\mathcal{C}) \otimes_{\mathrm{Pr}^L} \mathcal{D}$, which can be identified with a Bousfield localization of $\mathrm{PSh}(\mathcal{C}; \mathcal{D})$ of objects satisfying certain descent conditions \cite[Prop.1.3.1.7]{Lur18}. In \cite[pg.4]{Jar87} Jardine also constructs a model category equivalent to this construction (cf.\cite[Introduction of 6.5.4]{Lur09}). In the case $\mathcal{D} = \mathcal{D}(\ZZ)$ this identifies with the left completed derived category of the abelian category of $\mathrm{Ab}$-valued sheaves by \cite[Thm.2.1.2.2]{Lur18} applied to the case $\mathcal{O}$ being the sheafification of the constant sheaf $\ZZ$ in degree $0$. Thus, it is a relatively familiar object containing $\mathrm{Shv}(\mathcal{C}; \mathrm{Ab})$ in the case $\mathcal{D} = \mathcal{D}(\ZZ)$ as a full subcategory. We write $\ZZ[\_]: \mathcal{C} \to \mathrm{Shv}(\mathcal{C}; \mathcal{D}(\ZZ))$ for the $\ZZ$-linearized Yoneda functor.
\end{notat}
First we expand \cite[Sect.4.3]{BS14} to the extent we require in the following sections.
\begin{defin}
	Let $G$ be a profinite group. By this we mean a group which is inverse limit over its finite group quotients. Define the site $BG_\et$ as the category of finite discrete $G$-sets with coverings jointly surjective maps \footnote{In case $G$ is topologically finitely generated such as $G = \Gamma$ discreteness is implied by finiteness by a theorem of Nikolov and Segal \cite{NS06}}.
	Define the proétale topology on $\mathrm{Mod}_{G}(\mathrm{Pro}(\mathrm{Fin}))$ by coverings being detected by the forgetful functor $\mathrm{Mod}_{G}(\mathrm{Pro}(\mathrm{Fin})) \to \mathrm{Pro}(\mathrm{Fin})$ where the target is endowed with the joint surjection topology (cf. \cite[Lect.1]{CS26}). Denote the corresponding site by $BG_\proet$. In the case of $G = 1$ we write $BG_\proet = *_\proet$.
\end{defin}
\begin{lem}\label{Cond-24}
	Given a set $X$, denote by $(G \times X)^\wedge$ the profinite completion \footnote{By which we mean the left adjoint to the functor $\mathrm{Pro}(\mathrm{Fin}) \to \Set$ induced by the inclusion $\mathrm{Fin} \to \Set$} of the product. Equivalently, this is the product of $X^\wedge$ and $G$ in profinite sets.
	Then $\{(G \times X)^\wedge\}_{X \in \Set}$ induces a conservative family of points, i.e., restriction along which induces a geometric morphism with left exact left adjoint
	\begin{equation*}
		((G \times X)^\wedge)^*: \mathrm{Shv}(BG_\proet, \Set) \to \Set: ((G \times X)^\wedge)_*
	\end{equation*}
	which jointly detect equivalences \footnote{cf. \cite[Tome 1, Exp.IV, Def.6.1]{SGA4} for the definition of points of a topos}.
	
	As standard consequences, we obtain:
	\begin{enumerate}
		\item The geometric morphism $\mathrm{Shv}(BG_\proet) \to \Spc$ induced by restriction to the value at $(G \times X)^\wedge$ for all $X$ detects hypercomplete objects.
		\item We have a diagram of categories:
		\[\begin{tikzcd}
			\mathrm{PSh}(BG_\proet; \mathrm{An}) \arrow[r, "L_\proet"] \arrow[rrr, "\prod_{X \in \Set}((G \times X)^\wedge)^*", bend left=10] & \mathrm{Shv}^\text{un}(BG_\proet; \mathrm{An}) \arrow[r, "L_\text{hyp}"] & \mathrm{Shv}(BG_\proet; \mathrm{An}) \arrow[r, dashed] & \prod_{X \in \Set} \mathrm{An}
		\end{tikzcd}\]
		,i.e., the value at the points $(G \times X)^\wedge$ does not change under hypersheafification. This is still true once we tensor over any $\mathcal{D} \in \mathrm{CAlg}(\mathrm{Pr}^L)$. Moreover, the last map is conservative.
	\end{enumerate}
	\begin{proof}
		First we prove that restriction to the value at $((G \times X)^\wedge)^*$ induces a point of set valued sheaf topos: By \cite[Exp. IV, §6, Ex.6.5.2]{SGA4} it suffices to identify $(G \times X)^\wedge \in BG_\proet$ as a pro-object such that $\mathrm{Map}_{BG_\proet}((G \times X)^\wedge, \_)$ sends covering sieves to surjections. But $(G \times X)^\wedge$ is a representable pro-object of $BG_\proet$ and the second condition is satisfied by the underlying joint surjection topology:
		Given an underlying surjection $Y \to (G \times X)^\wedge$ \footnote{It suffices to treat the universal case: For $Y \to Z$ any other underlying surjection lifting any morphism $(G \times X)^\wedge \to Z$ along this map is equivalent to finding a section of the pulled back covering $Y \times_{Z} ((G \times X)^\wedge) \to (G \times X)^\wedge$.} of profinite $G$-sets, we may lift $X \xrightarrow{1 \times \mathrm{id}_X} G \times X \to (G \times X)^\wedge$ to obtain a $G$-equivariant factorization $G \times X \to Y \to (G \times X)^\wedge$ of the unit of profinite completion. Thus upon profinite completion we obtain a section.
		
		Secondly, we prove the factorization: First observe that the functor
		\begin{equation*}
			\mathrm{PSh}(BG_\proet) \xrightarrow{\prod_{X \in \Set} ((G \times X)^\wedge)^*} \prod_{X \in \Set} \mathrm{An}
		\end{equation*}
		inverts covering sieves by what we proved in the first part of this proof. Thus we get a factorization through the non-hypercomplete category of sheaves $\mathrm{Shv}^\text{un}(BG_\proet)$. Since the target is hypercomplete this geometric morphism further factorizes through the hypercompletion by \cite[Prop.6.5.2.13]{Lur09}.
		
		Lastly, we prove that the left adjoint $\mathrm{Shv}(BG_\proet) \to \prod_{X \in \Set} \mathrm{An}$ is conservative: The set of morphisms which are mapped to equivalences under this functor is a saturated class and coincides with the $\infty$-connective morphisms by \cite[pg.48]{Jar87}. So we are done.
	\end{proof}
\end{lem}
We need to introduce some condensed math as first developed in \cite{CS26}:
\begin{defin}
	Define the symmetric monoidal category of condensed anima as $\mathrm{Cond}\Spc := \mathrm{Shv}(*_\proet)$. For an arbitrary category $\mathcal{C} \in \mathrm{Pr}^L$ define $\mathrm{Cond}\mathcal{C} := \mathrm{Cond}\Spc \otimes_{\text{Pr}^L} \mathcal{C}$. In the case $\mathcal{C} = \mathcal{D}(\ZZ)$ we also write $\mathcal{D}(\mathrm{Cond})$ following tradition.
\end{defin}
\begin{notat}
	We identify a profinite set $S$ with its image in $\mathrm{Cond}\Spc$. For arbitrary $\mathcal{C} \in \mathrm{Pr}^L$ denote by $\underline{\cdot}: \mathcal{C} \to \mathrm{Cond}\mathcal{C}$ the constant sheaf functor and $\_(*)$ its right adjoint - its global sections. Denote by $\ZZ[\_]: \mathrm{Cond}\Spc \to \mathcal{D}(\mathrm{Cond})$ the left adjoint coming from tensoring $\ZZ[\_]: \Spc \to \mathcal{D}(\ZZ)$ with $\mathrm{Cond}\Spc$.
\end{notat}
\begin{ex}\label{Cond-21}
	We have $\ZZ[S](*) = \ZZ[S]$ for $S$ any profinite set by the fact that the value on $*$ is preserved by hypersheafification by Lemma \ref{Cond-24}.
\end{ex}
\begin{fact}\label{Cond-20}
	Let $\mathcal{C}$ be a symmetric monoidal category and let $A \in \mathrm{CAlg}(\mathrm{Alg}(\mathcal{C})^\op)^\op$ a cocommutative bialgebra. Then the category $\mathrm{Mod}_A(\mathcal{C})$ inherits a symmetric monoidal structure constructed in \cite{Ra26} having the following properties, which we review in the appendix with an added discussion about Hopf algebras:
	\begin{enumerate}
		\item Given $A \to B$ a map of cocommutative bialgebras, then restriction of scalars $\mathrm{Mod}_B(\mathcal{C}) \to \mathrm{Mod}_A(\mathcal{C})$ is symmetric monoidal (see \ref{H-3}). In particular, $\mathrm{Mod}_A(\mathcal{C}) \to \mathcal{C}$ is symmetric monoidal.
		\item The monoidal structure is compatible with base change, i.e., for $f: \mathcal{C} \to \mathcal{D}$ in $\mathrm{CAlg}(\mathrm{Pr}^L)$ the comparison map $\mathrm{Mod}_A(\mathcal{C}) \otimes_\mathcal{C} \mathcal{D} \xrightarrow{\simeq} \mathrm{Mod}_{f(A)}$ identifies the two a priori different monoidal structures (see \ref{H-3}).
		\item In case $\mathcal{C}$ is endowed with the cartesian monoidal structure this symmetric monoidal structure on $\mathrm{Mod}_A(\mathcal{C})$ identifies with the cartesian one (see \ref{H-11}).
		\item $A$ is called a Hopf algebra if $\mathrm{Mod}_A(\mathcal{C}) \to \mathcal{C}$ preserves internal mapping objects in the sense of \ref{Cond-15} - this is classically called closed monoidal. This is the case if and only if $A \in h\mathcal{C}$ as a bialgebra (note that $\mathcal{C} \to h\mathcal{C}$ is a symmetric monoidal) is a Hopf algebra in the classical sense (cf. \cite[5.2]{BLV11} for a definition of Hopf algebra object in a symmetric monoidal 1-category).
	\end{enumerate}
\end{fact}
A proof for the following statement can be found in the appendix at \ref{P-2}.
\begin{lem}\label{Cond-6}
	Let $G$ be a profinite group viewed as a group object in $\mathrm{Cond}\Spc$. Then the functor $BG_\proet \subset \Mod_G(\mathrm{Cond}\Set) \subset \Mod_G(\mathrm{Cond}\Spc)$ induces an equivalence
	\begin{equation*}
		\mathrm{Shv}(BG_\proet) \to \Mod_{G}(\mathrm{Cond}\Spc)
	\end{equation*}
	
	Thus applying $\_ \otimes_{\mathrm{Pr}^L} \mathcal{D}$ for any $\mathcal{D} \in \mathrm{CAlg}(\mathrm{Pr}^L)$ yields a symmetric monoidal equivalence:
	\begin{align*}
		\mathrm{Shv}(BG_\proet; \mathcal{D}) \to \Mod_{f(G)}(\mathrm{Cond}\mathcal{D})
	\end{align*}
	where $f$ denotes the canonical symmetric monoidal functor $\mathrm{Cond}\Spc \to \mathrm{Cond}\mathcal{D}$.
	Here we define the left hand side as in \cite[Def. 1.3.1.4]{Lur18} and endow the right hand side with the monoidal structure from Fact \ref{Cond-20}.
\end{lem}
\begin{rem}
	Define the $(\_)^{hG}$ continuous $G$-fixed points the be the functor corepresented by the cartesian unit $*^\text{triv} \in BG_\proet$. For $F \in \mathrm{Shv}(BG_\proet, \mathcal{D})$ for some presentable category $\mathcal{D}$ the underlying homotopy $F^{hG}(*)$ type is by definition just obtained by evaluating the corresponding sheaf on the left hand side at $*^\text{triv} \in BG_\proet$.
\end{rem}
\begin{defin}
	Define the solid derived category $\mathcal{D}_\square(\ZZ)$ as the localization at all morphisms $\ZZ[\lim_i X_i] \to \lim_i \ZZ[X_i] =: \ZZ[\lim_i X_i]^\square$ of $\mathcal{D}(\mathrm{Cond})$ \cite[Def.5.1]{CS26}. Denote by $(\_)^\square$ the associated localization functor and call it solidification. By \cite[Thm.5.8(ii)]{CS26} this is not a notational overload, since $\ZZ[\lim_i X_i] \to \lim_i \ZZ[X_i]$ identifies the target as the solidification of the source.
\end{defin}
\begin{fact}
	We need the following facts about $\mathcal{D}_\square(\ZZ)$:
	\begin{enumerate}
		\item It inherits a symmetric monoidal structure from $\mathcal{D}(\mathrm{Cond})$ making solidification symmetric monoidal \cite[Thm.6.2(ii)]{CS26}.
		\item Surprisingly, the inclusion $\mathcal{D}_\square(\ZZ) \subset \mathcal{D}(\mathrm{Cond})$ preserves colimits as well \cite[Thm.5.8]{CS26} \footnote{In \cite{CS26} Clausen-Scholze construct $\mathcal{D}(\mathrm{Cond})$ and $\mathcal{D}_\square(\ZZ)$ as derived categories of certain Grothendieck abelian categories $\mathrm{Solid} \subset \mathrm{Cond}$ \cite[Thm.2.2]{CS26}. The inclusion is closed under filtered colimits and $X$ lies in $\mathcal{D}_\square(\ZZ) \subset \mathcal{D}(\mathrm{Cond})$ if and only if its homology sheaves $H_j(X)$ lie in $\mathrm{Solid}$ both by \cite[Thm.5.8]{CS26}. Since $\mathrm{Cond}$ is Grothendieck abelian $H_j: \mathcal{D}(\mathrm{Cond}) \to \mathrm{Cond}$ preserve filtered colimits by \cite[Prop.1.3.5.21]{Lur17}. Thus $\mathcal{D}_\square(\ZZ) \subset \mathcal{D}(\mathrm{Cond})$ is closed under all colimits.}.
	\end{enumerate}
\end{fact}
\begin{ex}
	The functor $\underline{\cdot}: \mathcal{D}(\ZZ) \to \mathcal{D}(\mathrm{Cond})$ factors through $\mathcal{D}_\square(\ZZ)$: Since $\mathcal{D}(\ZZ)$ is generated under colimits and cofiber sequences by $\ZZ$, it suffices to show that $\underline{\ZZ}$ is solid. But $\underline{\ZZ} = \ZZ[*]$ is solid by explicit formula of solidification on representables.
\end{ex}
\begin{ex}\label{Cond-28}
	Define $\ZZ_l := \lim_{k \in \NN^\op}\underline{\ZZ/l^k} \in \mathcal{D}(\mathrm{Cond})$. As solid objects are closed under limits, $\ZZ_l \in \mathcal{D}_\square(\ZZ)$. Since $\underline{\ZZ/l^k}(*) = \ZZ/l^k$ the global section are again $\ZZ_l$. Denote by
	\begin{equation*}
		\mathcal{D}_\square(\ZZ_l) := \mathrm{Mod}_{\ZZ_l}(\mathcal{D}_\square(\ZZ))
	\end{equation*}
	Note that by \cite[Ex.6.4]{CS26} $\ZZ_l$ is an idempotent algebra in $\mathcal{D}_\square(\ZZ)$. So any $\ZZ/l^k$-module inherits a canonical $\ZZ_l$-module structure and therefore any $l$-power torsion module as well.
\end{ex}
\begin{ex}\label{Cond-3}
	We have the following sequence of symmetric monoidal left adjoints (with respect to the symmetric monoidal structures coming from commutative algebra structures)
	\begin{equation*}
		\mathcal{D}(\mathrm{Cond}) \to \mathrm{Mod}_{\ZZ_l}(\mathcal{D}(\mathrm{Cond})) \xrightarrow{(\_)^\square} \mathcal{D}_\square(\ZZ_l)
	\end{equation*}
	we claim that $\ZZ[\Gamma]$ gets mapped to an object $\Lambda$ whose global sections are $\Lambda$, justifying the notation:
	
	Thus we need to compute global sections of $(\ZZ_l \otimes \ZZ[\Gamma])^\square$ - the proof idea is from \cite[Ex.6.4]{CS26}:
	First observe that for $\ZZ[[U]] := \ZZ[\hat{\NN}]^\square$ where $\hat{\NN}$ denotes the profinite completion of $\NN$.
	\[\begin{tikzcd}
		\mathrm{lim}_k \underline{\ZZ[U]/U^k} \arrow[r, "U-l"] \arrow[d, "\simeq", no head] & \mathrm{lim}_k \underline{\ZZ[U]/U^k} \arrow[r] \arrow[d, "\simeq", no head] & \mathrm{lim}_k \underline{\ZZ/l^k} \arrow[d, "\simeq", no head] \\
		\ZZ[[U]] \arrow[r, "U-l"] & \ZZ[[U]] \arrow[r] & \ZZ_l
	\end{tikzcd}\]
	By a similar argument we have a cofiber sequence $\ZZ[[\hat{\NN} \times \Gamma]] \xrightarrow{U-l} \ZZ[[\hat{\NN} \times \Gamma]] \to \ZZ_l[[\Gamma]]$. But since $\ZZ[[\hat{\NN} \times \Gamma]] \simeq (\ZZ[\hat{\NN}] \otimes \ZZ[\Gamma])^\square$ we arrive at
	\begin{equation*}
		\Lambda \simeq (\ZZ_l \otimes \ZZ[\Gamma])^\square(*)
	\end{equation*}
\end{ex}
\begin{defin}\label{Cond-17}
	Denote $\Lambda := \ZZ_l[\Gamma]^\square$. Define
	\begin{equation*}
		\mathcal{D}_\square(\Lambda) := \mathrm{Mod}_{\Lambda}(\mathcal{D}_\square(\ZZ)) = \mathrm{Mod}_{\ZZ[\Gamma]}(\mathcal{D}(\mathrm{Cond})) \otimes_{\mathcal{D}(\mathrm{Cond})} \mathcal{D}_\square(\ZZ)
	\end{equation*}
	For a general profinite group $G$ we analogously define
	\begin{equation*}
		\mathcal{D}_\square(G) := \mathrm{Mod}_{\ZZ[G]^\square}(\mathcal{D}_\square(\ZZ))
	\end{equation*}
	\textit{Addendum}: So the lax monoidal functor $\_(*): \mathcal{D}_\square(\ZZ) \to \mathcal{D}(\ZZ)$ induces a map $\mathcal{D}_\square(\Lambda) = \mathrm{Mod}_{\Lambda}(\mathcal{D}_\square(\ZZ)) \to \mathrm{Mod}_\Lambda(\mathcal{D}(\ZZ)) \simeq \mathcal{D}(\Lambda)$, which induces an equivalence on endomorphisms of the unit. Thus the category of perfect complexes over $\Lambda$ $\mathrm{Perf}(\Lambda)$ embeds into $\mathcal{D}_\square(\Lambda)$ via the left adjoint $\underline{\cdot}$.
\end{defin}
\begin{defin}\label{Cond-9}
	A sheaf $F \in \mathrm{Shv}(BG_\proet; \mathcal{D}(\ZZ))$ is said to have classical homology sheaves if 
	\begin{equation*}
		\colim_{\mathcal{I}^\op} \mathrm{H}^*(G/H_i; H_p^\proet(F)) \xrightarrow{\simeq} \mathrm{H}^*(\lim_{\mathcal{I}^\op} X_i; H_p^\proet(F))
	\end{equation*}
	for all $\mathcal{I}$-indexed systems of open subgroups with $\mathcal{I}$ cofiltered, $* \in \ZZ$ and
	\begin{equation*} 
		\mathrm{H}_p^\proet: \mathrm{Shv}(BG_\proet;\mathcal{D}(\ZZ)) \simeq \mathcal{D}(BG_\proet;\ZZ) \to \mathrm{Shv}(BG_\proet; \mathrm{Ab})
	\end{equation*}
	denotes the homology sheaf functor.
	\textit{Addendum}: The condition on homology sheaves is equivalent to
	\begin{equation*}
		\colim_{i \in \mathcal{I}^\op} R\Gamma(G/H_i, \tau_{\leq p}^\proet F) \xrightarrow{\simeq} R\Gamma(G/H, \tau_{\leq p}^\proet F)
	\end{equation*}
	for $p \in \ZZ$ where $\tau_{\leq n}^\proet$ denotes the homological truncation coming from the description of $\mathrm{Shv}(BG_\proet; \mathcal{D}(\ZZ))$ as derived category\footnote{
		This is a standard Postnikov tower argument: The condition on cohomology groups is equivalent to
		\begin{equation*}
			\colim_{i \in \mathcal{I}^\op} R\Gamma(G/H_i, H_{\leq p}^\proet F) \to R\Gamma(G/H, H_{\leq p}^\proet F)
		\end{equation*}
		being an equivalence. Using $\tau_{\geq q-1}\tau_{\leq p}(F)/\tau_{\geq q}\tau_{\leq n}(F) \simeq H_{p-1}^\proet(F)[p-1]$ we get the statement for $\tau_{\geq q}\tau_{\leq p}(F)$ $q \leq p$ replacing $\mathrm{H}_p^\proet$. Thus it remains to prove that $R\Gamma(G/H_i; \_)$ commutes with the filtered colimit over $\{\tau_{\geq q}\tau_{\leq p}(F)\}_q$. But one observes that the resulting system stabilizes in homotopical degree $t$ for $q < t$ showing the desired result.
	}
\end{defin}
\begin{fact}\label{H-1}
	Consider the morphism of sites $\nu: BG_\et \to BG_\proet$. It induces a geometric morphism $\nu^*: \mathrm{Shv}(BG_\et; \mathcal{D}(\ZZ)) \leftrightharpoons \mathrm{Shv}(BG_\proet; \mathcal{D}(\ZZ)): \nu_*$ with $\nu^*$ identifying the left hand side as those objects with classical homotopy sheaves (cf. \cite[Prop.5.3.2]{BS14} for the proof for the proét site of a scheme; the proof that classicality on $\mathrm{H}^0$ implies the condition in \ref{Cond-9} on all cohomology groups is given in \cite[Cor.5.1.6]{BS14}). Since for the computation later on, we will only encounter groups that are absolute Galois groups of some field, we can just refer to the Bhatt-Scholze result.
\end{fact}
\begin{cor}\label{Cond-18}
	$\nu^*$ factors through the full subcategory $\mathcal{D}_\square(G) \subset \Mod_{\ZZ[G]}(\mathcal{D}(\mathrm{Cond}))$.
	\begin{proof}
		Denote by $\Mod_{\ZZ[G]}(\mathcal{D}(\mathrm{Cond})) \to \mathrm{Cond}$ the forgetful functor. It suffices to show that $u \circ \nu^*$ factors through the full subcategory $\mathcal{D}_\square(\ZZ)$, where $u$ is the forgetful functor $\Mod_{\ZZ[G]}(\mathcal{D}(\mathrm{Cond})) \to \mathcal{D}(\mathrm{Cond})$. Since solid objects are closed under colimits the claim boils down to proving that $u\nu^*\ZZ[X] = \ZZ[uX]$ for $X$ a finite $G$-set is solid. But that is clearly the case by explicit formula of solidification on representables.
		
		The claim about fully faithfulness follows from the more general fact: Given a symmetric monoidal localization $f:\mathcal{C} \to \mathcal{D}$ and $A \in \mathrm{CAlg}(\mathcal{C})$. Then by naturality of the base change formula the map $\mathrm{Mod}_A(\mathcal{C}) \to \mathrm{Mod}_{f(A)}(\mathcal{C})$ identifies with the canonical map $\mathrm{Mod}_A(\mathcal{C}) \otimes_\mathcal{C} \mathcal{C} \to \mathrm{Mod}_A(\mathcal{C}) \otimes_\mathcal{C} \mathcal{D}$ \cite[Thm.4.8.4.6]{Lur17}. But base change of a Bousfield localization = localization at a class of morphisms in $\mathrm{Pr}^L$ is a Bousfield localization by universal property of the Lurie tensor product \cite[Prop.4.8.1.15]{Lur17}, hence, its right adjoint is fully faithful \cite[Prop.5.5.4.15]{Lur09}.
	\end{proof}
\end{cor}
\begin{ex}
	If $G = \Gamma_F \simeq \Delta \times \Gamma$ we observe that $\ZZ[G]^\square = \underline{\ZZ[\Delta]} \otimes \ZZ[\Gamma]^\square$ and thus
	\begin{equation}
		\mathcal{D}_\square(\Gamma_F) \simeq (\mathcal{D}_\square(\Gamma))^{B\Delta}
	\end{equation}
	After tensoring with $\mathcal{D}_\square(\ZZ_l)$ this expression splits $\prod_{\mathrm{Hom}_\mathrm{Ab}(\Delta, \ZZ_l^\times)} \mathcal{D}_\square(\Lambda)$ because $\Lambda_F$ splits.
\end{ex}
Lastly, we want to prove that an object in $\mathcal{D}_\square(\Lambda)$ is determined by a relatively small datum - just the action of a topological generator on an object in $\mathcal{D}_\square(\ZZ_l)$.
\begin{lem}\label{Cond-29}
	Pick any topological generator $\gamma \in \Gamma$. Then the map specifying the topological generator $\gamma: \ZZ_l[\ZZ] \xrightarrow{\gamma} \Lambda$ induces a fully faithful closed symmetric monoidal functor upon restriction of scalars:
	\begin{equation*}
		\gamma^*: \mathcal{D}_\square(\Lambda) \subset \mathcal{D}_\square(\ZZ_l)^{B\ZZ}
	\end{equation*}
	\textit{Bonus}: Given $M \in \mathcal{D}_\square(\Lambda)$ then
	\begin{equation*}
		M^{hG} \simeq M^{h\ZZ}
	\end{equation*}
	So fixed points can be computed by a fiber sequence.
	\begin{proof}
		We observe two things:
		\begin{enumerate}
			\item The outer and right most composite in the following picture are closed symmetric monoidal, i.e., they preserve the internal hom in the sense of \ref{Cond-15}, with respect to the monoidal structure coming from the Hopf algebra structures:
			\[\begin{tikzcd}
				\mathcal{D}_\square(\Lambda) \simeq \mathrm{Mod}_{\Lambda}(\mathcal{D}_\square(\ZZ)) \xrightarrow{\gamma^*} \mathrm{Mod}_{\ZZ_l[\ZZ]^\square}(\mathcal{D}_\square(\ZZ)) \to \mathcal{D}_\square(\ZZ_l)
			\end{tikzcd}\]
			Since the last map is furthermore conservative, this implies that the leftmost composite is also closed monoidal.
			\item Denote by $\gamma_!$ extension of scalars. Then
			\begin{equation*}
				\gamma_!\ZZ_l^\text{triv} = \gamma_!\mathrm{cofib}(\ZZ_l[\ZZ]^\square \xrightarrow{\gamma -1} \ZZ_l[\ZZ]^\square) = \mathrm{cofib}(\Lambda \xrightarrow{\gamma-1} \Lambda) \simeq \ZZ_l^\text{triv}
			\end{equation*}
			For the last equivalence we used that the cofiber can be taken in the essential image of $\underline{\cdot}^\square: \mathrm{Perf}(\Lambda) \to \mathcal{D}_\square(\Lambda)$, where we can use the short exact sequence $\Lambda \xrightarrow{\gamma-1} \Lambda \to \ZZ_l^\text{triv}$.
			
			From this the \textit{Bonus} already follows: Continuous $G$-fixed points and $\ZZ$-fixed points are both defined as the functors corepresented by $\ZZ_l^\text{triv}$. Since the left adjoint $\gamma_!$ preserves this object, the statement follows.
		\end{enumerate}
		For $x, y \in \mathcal{D}_\square(\Lambda)$ we have the following square
		\begin{equation}\label{Cond-26}
		\begin{tikzcd}
			\mathrm{map}_{\mathcal{D}_\square(\Lambda)}(x, y) \arrow[d, "\gamma^*"] \arrow[r, "\simeq"] & \mathrm{map}_{\mathcal{D}_\square(\Lambda)}(\ZZ_l^\text{triv}, \underline{\mathrm{map}}(x, y)) \arrow[d, "\simeq"] \\
			\mathrm{map}_{\mathcal{D}_\square(\ZZ_l)^{B\ZZ}}(\gamma^*x, \gamma^*y) \arrow[r, "\simeq"] & \mathrm{map}_{\mathcal{D}_\square(\ZZ_l)^{B\ZZ}}(\ZZ_l^\text{triv}, \underline{\mathrm{map}}(\gamma^*x, \gamma^*y))
		\end{tikzcd}
		\end{equation}
		\begin{enumerate}
			\item The vertical equivalences come from the internal hom-adjunction
			\item The right vertical equivalence comes from $\gamma^*\underline{\mathrm{map}}(x, y) \to \underline{\mathrm{map}}(\gamma^*x, \gamma^*y)$ being an equivalence by the first observation and $\gamma_!\ZZ_l^\text{triv} \simeq \ZZ_l^\text{triv}$.
		\end{enumerate}
		It remains to prove that the diagram admits a homotopy, which we check in Appendix \ref{H-15}.
	\end{proof}
\end{lem}
		\subsection{Étale cohomological input}\label{A3}
			In this section:
\begin{enumerate}
	\item We quickly review necessary facts about the étale fundamental group and the proétale topology from \cite{BS14}. We use them to formalize pro-Galois descent.
	\item We collect vanishing and finiteness results for $l$-adic étale cohomology of $R$ and $R_\infty$ making sure that we only use results which do not use Borel.
	\item Compute some specific cohomology groups that will enter the computation later on.
\end{enumerate}
\subsubsection{The pro-étale topology}
We recall the necessary definitions and results in \cite{BS14}. Throughout, $X$ is assumed to be a normal, integral scheme with generic point $\eta$:
\begin{defin}
	Define $X_\proet$ as the category of flat $X$-schemes $Y$ such that the diagonal is flat as well - so-called weakly étale. Endow $X_\proet$ with the fpqc topology making it into a site.
\end{defin}
\begin{defin}
	Denote by $X_\fet \subset X_\et \subset X_\proet$ the full subcategory of finite étale $X$-schemes. Let $x$ a geometric point. Call the restricted pullback functor $x^*\vert_{X_\fet}: X_\fet \to \mathrm{Fin}$ as well. Denote by $\pi_1^\et(X, x) := \mathrm{Aut}(x^*)$. This group is topologized by considering the embedding
	\begin{equation*}
		\mathrm{Aut}(x^*) \to \prod_{Y \in X_\proet} \mathrm{Aut}_{\mathrm{Fin}}(x^*Y)
	\end{equation*}
	where the target has the product topology and $\mathrm{Aut}_\text{Fin}(x^*(Y))$ has the discrete topology \cite[\href{https://stacks.math.columbia.edu/tag/0BMR}{Tag 0BMR}]{stacks-project}. Since the target of this embedding is a product of finite groups, this implies that $\pi_1^\et(X, \bar{\eta})$ is a profinite group as well, i.e., limit of its quotients by normal open subgroups.
\end{defin}
\begin{fact}\label{Et-20}
	Recall that $\pi_1^\et(X, \bar{\eta})$ has the following properties:
	\begin{enumerate}
		\item The functor $\bar{\eta}^*: X_\fet \to \mathrm{Fin}$ factors through $B\pi_1^\et(X, \bar{\eta})_\et$ as an equivalence by \cite[Lemma 7.4.1]{BS14}.
		\item For $X = k$ a field the Galois group canonically identifies with its absolute Galois group and both notions coincide. This, for example, recovers Galois theory of fields \cite[\href{https://stacks.math.columbia.edu/tag/03QR}{Tag 03QR}]{stacks-project}.
		\item $\pi_1^\et(\_)$ is covariantly functorial in schemes pointed by a geometric point: Let $(X, x) \to (Y, y)$ be a map of schemes pointed by geometric points, i.e., morphisms $x \to X$ and $y \to Y$ from the spectrum from a separably closed field. By assumption the diagram
		\[\begin{tikzcd}
			Y_\fet \arrow[r, "f^*"] \arrow[rr, bend left, "y^*"] & X_\fet \arrow[r, "x^*"] & \mathrm{Fin}
		\end{tikzcd}\]
		commutes and therefore we obtain a map of groups $\mathrm{Aut}(x^*) \to \mathrm{Aut}(y^*)$.
		
		\item\label{Et-20-1} We can apply this functoriality to $\widebar{\kappa(\eta)} \to X$, where $\kappa(\eta)$ denotes the function field. By the normality assumption the canonical map
		\begin{equation*}
			\mathrm{Gal}(\widebar{\kappa(\eta)}/\kappa(\eta)) \simeq \pi_1(\kappa(\eta), \eta) \to \pi_1(X, \bar{\eta})
		\end{equation*}
		identifies the target as the Galois group of $K/\kappa(\eta)$ of the maximal unramified extension \cite[\href{https://stacks.math.columbia.edu/tag/0BQM}{Tag 0BQM}]{stacks-project}.
	\end{enumerate}
\end{fact}
\begin{ex}
	By Fact \ref{Et-20} bullet point \ref{Et-20-1} we have:
	\begin{equation*}
		M_\infty = (\pi_1^\et(X, \bar{\eta})/\overline{[\pi_1^\et(X, \bar{\eta}), \pi_1^\et(X, \bar{\eta})]})^\wedge_l
	\end{equation*}
	Here $\pi_1^\et(X, \bar{\eta})/\overline{[\pi_1^\et(X, \bar{\eta}), \pi_1^\et(X, \bar{\eta})]}$ is the abelianization in profinite groups, i.e., the quotient by the closure of the commutator.
\end{ex}
From the viewpoint of the proétale site, infinite Galois theory is also more natural:
\begin{constr}
	The fiber functor induces an equivalence
	\begin{equation*}
		X_\fet \xrightarrow{\bar{\eta}^*} \mathrm{Mod}_{\pi_1^\et(X, \bar{\eta})}(\mathrm{Fin}) \simeq (B\pi_1^\et(X, \bar{\eta}))_\et
	\end{equation*}
	Taking pro-categories on both sides we arrive at
	\begin{equation}\label{Et-2}
	\begin{tikzcd}
		\mathrm{Pro}(X_\fet) \arrow[d, "i"] \arrow[r, "\mathrm{Pro}(x^*)"] & \mathrm{Pro}(B\pi_1^\et(X, \bar{\eta})) \arrow[d, sloped, "\simeq"] \\
		X_\proet & B\pi_1^\et(X, \bar{\eta})_\proet
	\end{tikzcd}
	\end{equation}
	where for the existence of $i$ we used that $X_\proet$ admits limits over cofiltered diagrams with affine transition maps \cite[Lemma 4.2.2]{BS14} and for the equivalence on the right-hand side we use \ref{F-9}. We claim that it is also fully faithful:
	Since $X$ is qcqs we may write $X$ as a finite colimit along open immersions of affine schemes. Since all objects in \eqref{F-9} satisfy Zariski excision \footnote{This just follows from weakly étale and finite étale being properties which can be checked Zariski locally on the base}, we may therefore assume that $X = \mathrm{Spec}(A)$ is affine. This reduces the question to
	\begin{equation*}
		\mathrm{Ind}(\mathrm{CAlg}_\fet(A)) \subset \mathrm{CAlg}(A)
	\end{equation*}
	being fully faithful where $\mathrm{CAlg}_\fet(A)$ denotes the category of finite étale algebras over $A$: Indeed, the proétale site over $A$ is the full subcategory of schemes over $A$ whose structure morphism is weakly étale. Since $\mathrm{CAlg}_\fet(A) \subset \mathrm{CAlg}(A) \subset \mathrm{Sch}_{A/}^\op$ is fully faithful and lands by definition of étale inside compact objects=finitely presented algebras the claim follows from \cite[Prop.5.3.5.11]{Lur09} (characterization for when a morphism from an Ind-category is fully faithful).
\end{constr}
Now we can define a pro-Galois cover verbatim to the non-pro case:
\begin{defin}\label{Et-31}
	Let $G$ be a profinite group. We call $Y \to X \in \mathrm{Pro}(X_\fet) \subset X_\proet$ a $G$-Galois covering if under the equivalence $\mathrm{Pro}(X_\fet) \simeq (B\pi_1^\et(X, \bar{\eta}))_\et$ it corresponds to a $\pi_1^\et(X, \bar{\eta})$-set arising from a group quotient $\pi: \pi_1^\proet(X, \bar{\eta}) \to G$. Note that $\pi$ is part of the datum.
\end{defin}
\begin{ex}
	$R_\infty/R$ is a $\Gamma_F$-Galois covering in the sense of Definition \ref{Et-31}.
\end{ex}
\subsubsection{Examples of solid $G$-modules via étale global sections}
\begin{cor}\label{Et-11}
	Let $G$ be a profinite group. Moreover, let $f: W \to X$ be a $G$-Galois covering and $r: BG_\proet \to B\pi^\et_1(X, \eta)_\proet \simeq \mathrm{Pro}(X_\fet) \subset X_\proet$ the left exact functor of sites specifying that covering. Restricting yields a continuous functor $BG_\et \to B\pi_1^\et(X, \eta)_\et \subset X_\et$. We have the following induced commutative diagram of categories
	\[\begin{tikzcd}
		\mathrm{Shv}(BG_\et; \mathcal{D}(\ZZ)) \arrow[r, "\nu^*"] & \mathrm{Shv}(BG_\proet; \mathcal{D}(\ZZ)) \\
		\mathrm{Shv}(X_\et; \mathcal{D}(\ZZ)) \arrow[r, "\nu^*"] \arrow[u, "r_*"] & \mathrm{Shv}(X_\proet; \mathcal{D}(\ZZ)) \arrow[u, "r_*"]
	\end{tikzcd}\]
	witnessed by the Beck-Chevalley transfer is commutative, when restricted to the full subcategories of bounded above objects. Again by taking the tensor product in $\mathrm{Pr}^L$ we can get the statement for arbitrary presentable categories as coefficients.
	\begin{proof}
		We just apply the proof in \cite[Lemma 5.4.3]{BS14} to our situation. The proof in our case even simplifies: We first show that $r_* \circ \nu^*$ factors through classical sheaves, when restricted to sheaves that have bounded above homology sheaves: We need to check that for a cofiltered diagram of open subgroups $\{H_i\}_{i \in \mathcal{I}}$ of $G$ $\colim_{\mathcal{I}^\op} r_*\nu^*F(G/H_i) \simeq r_*\nu^*F(\lim_{\mathcal{I}} G/H_i)$. But the left hand side identifies with
		\begin{equation*}
			(\nu^*F)(r(\lim_{\mathcal{I}} G/H_i)) \simeq (\nu^*F)(\lim_{\mathcal{I}} r(G/H_i)) \simeq \colim_{\mathcal{I}^\op}(\nu^*F)(r(G/H_i)) \simeq \colim_{\mathcal{I}^\op}r_*(\nu^*F)(G/H_i)
		\end{equation*}
		Here we used
		\begin{enumerate}
			\item for the middle equivalence that $r$ takes any $G/H_i \to *$ to a finite map and therefore the diagram $r(G/H_i)$ in $i$ is a cofiltered diagram with affine transition maps.
			\item the fact that $\nu^*F$ has classical homology sheaves \cite[Prop.5.3.2]{BS14}, i.e., any Postnikov truncation from above has sends cofiltered limits along affine transition maps to filtered colimits.
		\end{enumerate}
		Thus we conclude, when restricting to bounded above objects:
		\begin{equation*}
			r_*\nu^* \simeq \nu^*\nu_*r_*\nu^* \simeq \nu^*r_*\nu_*\nu^* \simeq \nu^*r_*
		\end{equation*}
		where for the first equivalence we used that $r_*\nu^*$ factors through the essential image of $\nu^*$.
	\end{proof}
\end{cor}
\begin{rem}\label{Et-13}
	This shows that $R\Gamma_\et(X, \_)$ of a classical sheaf automatically inherits the structure of a solid $G$-module. If $G = \Gamma$ and $F$ is $l$-power-torsion, then by \ref{Cond-28} $R\Gamma_\et(X, F)$ promotes to a solid $\Lambda$-module. In our case $l$ will always be invertible on $X$ and therefore this remark applies to $F = \ZZ/l^k(i) := \mu_{l^k}^{\otimes_{\ZZ/l^k} i}$ for $k \in \NN \cup \{\infty\}$.
\end{rem}
Recall that all the different formalisms for defining Tate twists of $\ZZ_l$ - $l$-adic étale cohomology agree
\begin{defin}\label{Et-14}
	The following ways of defining $l$-adic étale cohomology agree:
	\begin{enumerate}
		\item Ekedahl's formalism \cite{Ek90} of $l$-adic complexes and associated cohomology theory
		\item Cisinski-Deglise's approach of just $l$-completing the derived category of étale sheaves on scheme $X$, this was shown to produce the same cohomology theory as Ekedahl's definition \cite{Ek90} in \cite[7.2.21 + 4.5.2]{CD16}. More explicitly
		\begin{equation}\label{Et-4}
			R\Gamma(X, \ZZ_l(i)) = \lim\limits_{k \in \NN^\op}R\Gamma(X, \mu_{l^k}^{\otimes_{\ZZ/l^k} i})
		\end{equation}
		taken in the derived category $\mathcal{D}(\ZZ)$
		\item Ekedahl's formalism for $\ZZ_l$-coefficients yields the same cohomology as the proétale sheaf of abelian groups $\ZZ_l(i) := \ZZ_l(1)^{\otimes_{\ZZ_l} i}$ where $\ZZ_l(1)$ is defined as $\lim\limits_{k \in \NN^\op} \mu_{l^k}$ taken in proétale sheaves \cite{BS14}.
	\end{enumerate}
	Thus, we may use whichever definition of $l$-adic étale cohomology is the most convenient:
	Cisinski-Deglise's definition is more useful for computations whereas Bhatt-Scholze's definition is more practical for theoretical statements.
\end{defin}
\begin{rem}
	In the situation of \ref{Et-11} with the further assumption that $l$ is invertible on $X$ we get that global sections of $\ZZ_l(i)$ also factor through solid $G$-modules, as
	\begin{equation*}
		r_*(\ZZ_l(i)) \simeq \mathrm{lim}_{k \in \NN^\op} r_*(\ZZ/l^k)
	\end{equation*}
	and solid $G$-modules are closed under limits. In case $G = \Gamma$ this gives $r_*(\ZZ_l(i))$ a canonical solid $\Lambda$-module structure.
\end{rem}
\subsubsection{Étale cohomology computations}
\begin{fact}\label{Et-8}
	Recall the following facts about the étale cohomology of the locally constant étale sheaves $\ZZ/l^k(i) := \mu_{l^k}^{\otimes_{\ZZ/l^k} i}$ \footnote{Often people omit the base over which one tensors, the advantage of writing it like that is that derived tensor product and the underived one coincide. If we would tensor over $\ZZ$ we would not get the correct result!}:
	\begin{enumerate}
		\item The cohomology groups $\mathrm{H}^*(R_n, \ZZ/l^k(i))$ - by which we mean étale cohomology groups, no other Grothendieck topology will appear in this text - are finite by \cite[Thm.2.13]{Mil06}, since the sheaf is $l^k$-torsion and $l^k$ is invertible on $\mathrm{Spec}(R_n)$ for $n \in \NN$.
		\item $\mathrm{H}^p(R_n, \ZZ/l^k(i)) = 0$ for $p \geq 3$:
		This is a consequence of the Poitou-Tate Theorem (see \cite[Ch.5.3 Prop.11]{Ser97} for a proof):
		Let $S$ be a set of places of a global field $K$ containing all archimedean places and $G_S$ the Galois group of the maximal extension which only ramifies over $S$. Let $M$ be a finite $G_S$-module whose order is a unit $\mathcal{O}_{K, S}$. Then we have an equivalence of group cohomology groups for $p \geq 3$:
		\begin{equation*}
			\mathrm{H}^p(G_S, M) \xrightarrow{\simeq} \bigoplus_{v \text{ real place}}\mathrm{H}^p(D_v, M_v)
		\end{equation*}
		where $D_v$ denotes the decomposition group at the real place $v$, which is isomorphic to $C_2$. So if $M$ is of order coprime to $2$ the right-hand side vanishes. Combining this with the following fact we arrive at the desired result:
		\begin{equation*}
			\mathrm{H}^p(\mathcal{O}_{K, S}, \ZZ/l^k(i)) \simeq \mathrm{H}^p(G_S, \ZZ/l^k(i))
		\end{equation*}
		for all $p$ if $l$ is invertible in $\mathcal{O}_{K, S}$ \cite[Prop.2.9]{Mil06}.
	\end{enumerate}
\end{fact}
\begin{constr}\label{Et-30}
	Let $X \to Y$ be a finite étale map. Because of the étale assumption $\Delta: X \to X \times_Y X$ is clopen. Denote the complement by $X^c$. Then define the geometric evaluation map/Poincare pairing $\mathrm{gev}_X$:
	\[\begin{tikzcd}
		\ZZ[X] \otimes \ZZ[X] \arrow[r, "\simeq"] & \ZZ[X \times_Y X] \arrow[r] & \ZZ[X \times_Y X/X^c] & \ZZ[X] \arrow[l, "\Delta", "\simeq"'] \arrow[r, "f"] & \ZZ
	\end{tikzcd}\]
	where we use Zariski descent in the middle. Denote the adjoint $\ZZ[X] \to \ZZ[X]^\vee$ by $\mathrm{comp}_X$.
	
	\begin{claim}
		$(\ZZ[X], \ZZ[X], \mathrm{gev}_X)$ is can be extended to a dualizing datum (see \cite[Def.4.6.1.1]{Lur17}), which is a property by \cite[Lemma 4.6.1.6]{Lur17}.
		\begin{proof}
			Since this construction lifts along $r^*: \Mod_{\ZZ[\pi_1^\et(Y, \bar{\eta})]^\square}(\mathcal{D}_\square(\ZZ)) \to \mathrm{Shv}(Y_\proet; \mathcal{D}(\ZZ))$ we may check dualizability there. But dualizable objects in the right-hand side are detected underlyingly by \ref{H-10}. But underlyingly, the pairing $\mathrm{gev}_X$ becomes the pairing witnessing self-duality of a finite dimensional free $\ZZ$-module with specified basis.
		\end{proof}
	\end{claim}
	Thus we can define a transformation
	\begin{equation}\label{Et-12}
		\begin{tikzcd}[column sep=large]
			f_*f^*(\_) & \_ \otimes \ZZ[X]^\vee \arrow[l, "\simeq", "\mathrm{proj}"'] & \_ \otimes \ZZ[X] \arrow[l, "\simeq"', "\mathrm{comp}_X"] \arrow[r, "f"] & \mathrm{id}
		\end{tikzcd}
	\end{equation}
	where $\mathrm{proj}$ is the transformation witnessing the projection formula for finite étale morphisms.
	
	In the appendix \ref{P-13} we show that compatibility with basechange and it being the fold map for $f$ a split covering already uniquely determines these trace transformations and that they witness $f_* \dashv f^*$.
	
	Because of this canonicity it is easy to check that this recovers all other possible definitions of the trace along finite étale morphisms and that it recovers the usual traces for units, Picard and Brauer groups on global sections.
\end{constr}
\begin{fact}\label{Et-7}
	Recall the following computational facts for passing from $R_n$ to $R_\infty$ and $\ZZ/l^k(i)$ to $\QQ_l/\ZZ_l(i)$ for $i \in \ZZ$ in cohomology.
	\begin{enumerate}
		\item\label{Et-7-1} By classicality of the sheaves $\ZZ/l^k(i)$ for $k \in \NN \cup \{\infty\}$ the cohomology groups of $R_\infty$ satisfy the colimit formula
		\begin{equation*}
			\mathrm{H}^*(R_\infty, \ZZ/l^k(i)) \simeq \colim_{n \in \NN} \mathrm{H}^*(R_n, \ZZ/l^k(i))
		\end{equation*}
		\item\label{Et-7-2} Fix $n \in \NN$. Denote by $\hat{\mathrm{H}}^*(\mathcal{U}, F)$ Cech cohomology groups \cite[\href{https://stacks.math.columbia.edu/tag/03AK}{Tag 03AK}]{stacks-project} associated to a covering family $\mathcal{U} = \{U_i \to \mathrm{Spec}(R_n)\}_i$. Then by the Artin Comparison Theorem for the affine $\mathrm{Spec}(R_n)$ \cite[Cor.4.2]{Ar71} there is a covering family $\mathcal{U}$ for which Cech cohomology and étale sheaf cohomology agree for all sheaves $F$. Therefore we have the formula:
		\begin{equation*}
			\mathrm{H}^*(R_n, \QQ_l/\ZZ_l(i)) \simeq \hat{\mathrm{H}}^*(\mathcal{U}, \QQ/\ZZ_l(i)) \simeq \colim_{k \in \NN}\hat{\mathrm{H}}^*(\mathcal{U}, \ZZ/l^k(i)) \simeq \colim_{k \in \NN}\mathrm{H}^*(R_n, \ZZ/l^k(i))
		\end{equation*}
		\item\label{Et-7-3} There is a $\Gamma_F$-equivariant equivalence
		\begin{equation}\label{Et-3}
			\RH^*(S_\infty, \ZZ/l^k(i)) \simeq \RH^*(S_\infty, \ZZ/l^k)(i)
		\end{equation}
		for $S = R, \FF$ and $k \in \NN \cup \{\infty\}$.
		This follows from combining the following facts:
		\begin{enumerate}
			\item The sheaf $\ZZ/l^k(i) \simeq \ZZ/l^k$ when restricted to the étale site of $S_\infty$.
			\item $\pi_0R\Gamma_\proet(R_\infty, \ZZ/l^k(i)) \simeq \ZZ/l^k(i)$ $\Gamma_F$-equivariantly.
			\item To identify the action observe that for $k$ finite that we have the natural transformation
			\begin{equation*}
				\pi_{-*}R\Gamma_\proet(\_, \ZZ/l^k(i)) \otimes_{\ZZ/l^k} \pi_0R\Gamma_\proet(\_, \ZZ/l^k(i)) \Rightarrow \pi_{-*}R\Gamma_\proet(\_, \ZZ/l^k(i)): R_\proet \to \mathrm{gr}\mathrm{Ab}
			\end{equation*}
			natural transformation, which evaluates to an equivalence on $R_\infty$. This induces the desired identification in \eqref{Et-3}, which is $\Gamma_F$-equivariant by acting on the input. By bullet points \ref{Et-7-1} and \ref{Et-7-2}, taking colimits in $k$ yields the result for $k = \infty$ as well.
		\end{enumerate}
	\end{enumerate}
\end{fact}
\begin{lem}\label{5-2}
	Let $n \in \NN$. Then $(R_n^\times)\lcomp \simeq \RH^1(R_n; \ZZ_l(1))$
	\[\begin{tikzcd}
		0 \arrow[r] & \Pic(R_n) \otimes \ZZ_l \arrow[r] & \RH^2(R_n, \ZZ_l(1)) \arrow[r] & B_n \arrow[r] & 0
	\end{tikzcd}\]
	where $B_n \coloneqq \Hom(\QQ_l/\ZZ_l, \mathrm{Br}(R_n))$\footnote{By a paper of Ofer Gabber \cite[Ch.II, Thm.1]{Gab80} the Azumaya-algebra Brauer group and the torsion in $\RH^2(R_n, \GG_m)$ coincide. As by Cor. 2.6 in \href{https://www.jmilne.org/math/Books/ECpup4.pdf}{Milne's updated version of his book} \cite{Mil80} the localization $R_n \to F_n$ induces an injection on both Brauer groups and $\RH^2(F_n; \GG_m)$ is torsion, we need not make any distinction between the two definitions.}. These identifications are compatible with the trace map on cohomology and the norms maps on units, class group and Brauer group.
	\begin{proof}
		This is obtained simply by passing to inverse limits in classical computations: Consider the fiber sequence in $\mathrm{Shv}(R_n; \mathcal{D}(\ZZ))$
		\begin{equation}\label{Et-6}
			\ZZ/l^k(1) = \mu_{l^\nu} \to \GG_m \xrightarrow{l^\nu} \GG_m
		\end{equation}
		The naturality statement for the trace follows from the fact that this fiber sequence is preserved by any pullback and by naturality of the trace as defined in \ref{Et-30}.
		
		Its associated long exact sequence and the classical descriptions
		\begin{equation*}
			\RH^1(R_n; \GG_m) \simeq \Pic(X) \;\;\; \RH^2(R_n; \GG_m) \simeq \text{Br}(R_n)
		\end{equation*}
		yield short exact sequences
		\[\begin{tikzcd}
			0 \arrow[r] & (R_n^\times)/l^k \arrow[r] & \RH^1(R_n, \ZZ/l^k(1)) \arrow[r] & \Pic(R_n)[l^k] \arrow[r] & 0 \\
			0 \arrow[r] & \Pic(R_n)/l^k \arrow[r] & \RH^2(R_n, \ZZ/l^k(1)) \arrow[r] & \text{Br}(R_n)[l^k] \arrow[r] & 0
		\end{tikzcd}\]
		Now we observe the following
		\begin{align*}
			\text{Br}(R_n)[l^k] &\simeq \Hom_\Ab(\ZZ/l^k, \text{Br}(R_n)) &
			{\lim_k}^1 R_n^\times/l^k &\simeq 0 \text{ due to finite generation of } R_n^\times \\
			\lim_k \Pic(R_n)[l^k] &\simeq 0 \text{ due to finiteness} &
			{\lim_k}^1 \Pic(R_n)/l^k &\simeq 0 \text{ due to finiteness}
		\end{align*}
		For the bottom left equivalence we observe that the maps in the $\NN^\op$-indexed diagram are $\Pic(R_n)[l^{k + 1}] \xrightarrow{l} \Pic(R_n)[l^k]$ and $\mathrm{Pic}(R_n)$ is finite. Since we are computing homotopy groups of the expression \eqref{Et-4} we can use Milnor's $\mathrm{lim}^1$-sequence to compute $l$-adic étale cohomology groups. But, since we proved that all the $\mathrm{lim}^1$-terms vanish we arrive at the desired result.
	\end{proof}
\end{lem}
The following is proven in the appendix Lemma \ref{F-21}.
\begin{fact}\label{Et-32}
	 Let $\Gamma$ be a profinite group and $G$ a finite group. Let $Y \to X$ be a $\Gamma$-Galois covering of normal integral schemes. Then there is a short exact sequence of profinite groups.
	\begin{equation*}
		1 \to \pi_1^\et(Y, \bar{\eta}) \to \pi_1^\et(X, \bar{\eta}) \to \Gamma \to 1
	\end{equation*}
	
	Moreover, there is an equivalence
	\begin{equation*}
		\mathrm{Map}_{\mathrm{Shv}(X_\et)}(Y ,BG) \simeq \mathrm{Map}_{\mathrm{Cond}\Spc}(\Gamma/\pi_1^\et(X, \bar{\eta}), BG) \simeq \mathrm{Map}_{\mathrm{Cond}\Spc}(B\pi_1^\et(Y, \bar{\eta}), BG)
	\end{equation*}
	where the first equivalence is $\Gamma$-equivariant.
	
	Taking $\pi_0$ yields a $\Gamma$-equivariant equivalence
	\begin{equation*}
		\mathrm{H}^1(Y, BG) \simeq \mathrm{Hom}_{\mathrm{Grp}(*_\proet)}(\pi_1^\et(Y, \bar{\eta}), G)/G
	\end{equation*}
	where $G$ acts by conjugation in the target. Moreover, this identification is $\Gamma$-equivariant with the following action: For $\gamma \in \Gamma$ and $f: \pi_1^\et(X, \bar{\eta}) \to G$ a representative modulo conjugation by $G$. Then pick a lift $\tilde{\gamma} \in \pi_1^\et(X, \bar{\eta})$ and precompose $f$ with the restriction of conjugation by $\tilde{\gamma}$, which restricts to $\pi_1^\et(Y, \bar{\eta})$.
\end{fact}
\begin{rem}\label{Et-9}
	In the case $G$ is abelian the formula simplifies to yield a $\Gamma$-equivariant equivalence
	\begin{equation*}
		\mathrm{H}^1(Y, BG) \simeq \mathrm{Hom}_{\mathrm{CGrp}(*_\proet)}(\pi_1^\et(Y, \bar{\eta})/\overline{[\pi_1^\et(Y, \bar{\eta}), \pi_1^\et(Y, \bar{\eta})]}, G)
	\end{equation*}
\end{rem}
\begin{lem}\label{2-2}
	 We have the following equivalences of $\Gamma_F$-modules for $i \in \NN$, $p \geq 1$ and $k \in \NN \cup \{\infty\}$.
	\begin{align*}
		\RH^0(\FF_\infty, \ZZ/l^k(i)) &\simeq \ZZ/l^k(i) & \RH^0(\FF_\infty, \ZZ_l(i)) \simeq \ZZ_l(i) \\
		\RH^p(\FF_\infty, \ZZ_l(i)) &\simeq 0
	\end{align*}
	\begin{proof}
		We first prove the corresponding results for the sheaves $\ZZ/l^k(i)$ and then pass to the limit using the $\mathrm{lim}^1$-sequence argument. By \ref{Et-7} bullet point \ref{Et-7-3} it suffices to assume $i = 0$.
		
		\textit{For $\mathrm{H}^0$}: Nothing to prove.
		
		\textit{For $\mathrm{H}^1$}: By Fact \ref{Et-32} $\mathrm{H}^1$ is equivalent to maps of profinite groups $\pi_1^\et(\FF_\infty) \to \ZZ/l^k$. But all finite discrete quotients of $\pi_1^\et(\FF_\infty)$ are not $l$-power torsion: Indeed, the absolute Galois group of any finite field $\FF$ is the profinite completion $\hat{\ZZ}$ of $\ZZ$. Moreover, for $\FF'/\FF$ a finite field extension the inclusion $\pi_1^\et(\FF') \to \pi_1^\et(\FF)$ corresponds to $\hat{\ZZ} \xrightarrow{\cdot [\FF':\FF]} \hat{\ZZ}$. Thus taking the limit yields that $\mathrm{Gal}(\bar{\FF}_\infty/\FF_\infty)$ does not contain an $l$-profinite part. Thus $\mathrm{H}^1$ vanishes.
		
		\textit{For $\mathrm{H}^p$ for $p \geq 2$}:  
		The cohomological dimension of finite fields is $1$ \cite[Ch.III, Prop.1.15]{Mil80} for torsion sheaves. Then pass to the colimit in $k$.
		
		All the $\mathrm{lim}^1$-terms clearly vanish, thus we arrive at the desired result for $\ZZ_l$-coefficients.
	\end{proof}
\end{lem}
\begin{lem}\label{2-1}
	Fix $k \in \NN \cup \{\infty\}, i \in \ZZ$. Then for $p \geq 2$:
	\begin{align*}
		\RH^0(R_\infty; \ZZ/l^k(i)) &\simeq \ZZ/l^k(i) & &
		\RH^1(R_\infty; \ZZ/l^k(i)) \simeq \Hom_{\mathrm{Cond}\mathrm{Ab}}(M_\infty, \ZZ/l^k(i)) \\
		\RH^p(R_\infty, \QQ_l/\ZZ_l(i)) &= 0
	\end{align*}
	All equivalences are maps of $\Gamma_F$-modules, where the action on $\Hom_{\mathrm{Cond}\mathrm{Ab}}(M_\infty, \ZZ/l^k(i))$ is given by conjugation.
	\begin{proof}
		Again by Fact \ref{Et-7} it suffices to treat the case $i = 0$.
		\textit{For $\mathrm{H}^0$}: Clear.
		
		\textit{For $\mathrm{H}^1$}: Apply the Fact \ref{Et-32} to the case of $R_\infty/R$. We see that the $\Gamma_F$-action on $M_\infty$ is the action described in Section \ref{A0}.
		
		\textit{For vanishing of $\mathrm{H}^2$}:
		We again use the cofiber sequence \eqref{Et-6} to obtain the short exact sequences
		\begin{equation*}
			0 \to \Pic(R_n)/l^k \to \RH^2(R_n; \ZZ/l^k) \to \text{Br}(R_n)[l^k] \to 0
		\end{equation*}
		functorial in $n \in \NN$ for $k \in \NN \cup \infty$. By Remark \ref{Et-7} we need to show that the colimit in $n, k \in \NN$ vanishes. For $\Pic(R_n)/l^k$ this is clear, since $\QQ_l/\ZZ_l \otimes \_ = \colim_k \ZZ/l^k \otimes \_$ kills finite groups of which $\mathrm{Pic}(R_n)$ is an example.
		
		So we are reduced to proving $\colim_{n, k} \mathrm{Br}(R_n)[l^k] = 0$. For that we need the following, which we carefully extract out of \cite{Mil06} in \ref{F-20}
		\begin{fact}
			There are injections $\mathrm{Br}(R_n) \to \bigoplus_{S_n} \QQ/\ZZ$ making the following diagram commute
			\[\begin{tikzcd}
				\mathrm{Br}(R_n) \arrow[r, hook] \arrow[d, "\mathrm{res}"] & \bigoplus_{v \in S_n} \QQ/\ZZ \arrow[d, "\oplus_v {[F_{n+1, v}: F_{n, v}]}"] \\
				\mathrm{Br}(R_{n+1}) \arrow[r, hook] & \bigoplus_{v \in S_{n+1}} \QQ/\ZZ
			\end{tikzcd}\]
			where $F_{n, v}$ denotes the completion of $F_n$ at the place $v$.
		\end{fact}
		In particular, since every prime $\mathfrak{l} \in S_n$ splits completely in $F_\infty/F_n$ $l = [F_{n+1}:F_n] = [F_{n+1, v}: F_{n, v}]$. Thus $\colim_{n, k} \mathrm{Br}(R_n)[l^k] = 0$ by the above fact.
		
		\textit{For vanishing of $\mathrm{H}^*$ for $* \geq 3$}: This is already true for the finite stages by Remark \ref{Et-8}.
	\end{proof}
\end{lem}
The following lemma makes precise that $\mathrm{H}^1(R_\infty, \QQ_l/\ZZ_l)$ "behaves" like the internal hom $\underline{\mathrm{map}}(M_\infty, \underline{\QQ_l/\ZZ_l})$ of $\mathcal{D}_\square(\Lambda_F)$ or, equivalently, $\mathcal{D}_\square(\ZZ_l)^{B\ZZ}$. Here we think of $M_\infty$ as being the $l$-completion of the solid abelian group $\pi_1^\et(R_\infty, \bar{\eta})^\mathrm{ab}$ concentrated in homological degree $0$. Equivalently, we may think of the solid $\Gamma_F$-module $M_\infty$ as the image of the perfect $\Lambda_F$-module $M_\infty$ under the embedding $\mathrm{Perf}(\Lambda_F) \subset \mathcal{D}_\square(\Lambda_F)$:
\begin{lem}\label{Et-27}
	$\underline{\mathrm{map}}_{\mathcal{D}_\square(\ZZ)}(M_\infty, \QQ_l/\ZZ_l)$, where $\QQ_l/\ZZ_l$ denotes the discrete object $\underline{\QQ_l/\ZZ_l}$ satisfies the following:
	\begin{enumerate}
		\item $\underline{\mathrm{map}}_{\mathcal{D}_\square(\ZZ_l)}(M_\infty, \QQ_l/\ZZ_l) \simeq \underline{\mathrm{map}_{\mathcal{D}_\square(\ZZ_l)}(M_\infty, \QQ_l/\ZZ_l)} \simeq \underline{\mathrm{map}_{\mathcal{D}_\square(\ZZ)}(M_\infty, \QQ_l/\ZZ_l)}$
		\item $\mathrm{map}_{\mathcal{D}_\square(\ZZ)}(M_\infty, \QQ_l/\ZZ_l)$ is concentrated in homological degree $0$.
	\end{enumerate}
	i.e., $\underline{\mathrm{map}}_{\mathcal{D}_\square(\ZZ_l)}(M_\infty, \QQ_l/\ZZ_l)$ is static and discrete.
	\begin{proof}
		\textit{For (1)}: The second equivalence is immediate, since $\ZZ_l$ is an idempotent algebra in $\mathcal{D}_\square(\ZZ)$ by \cite[Ex.6.4]{CS26}. The first equivalence is equivalent to $\underline{\mathrm{map}}_{\mathcal{D}_\square(\ZZ_l)}(M_\infty, \QQ_l/\ZZ_l)$ being discrete, i.e., lying in the essential image of $\nu^*$. We claim it is even true for any finite colimit over the objects $\ZZ_l \otimes_{\mathcal{D}_\square(\ZZ)} \ZZ[S]$ for $S \in *_\proet$. So we are reduced to showing that $\underline{\mathrm{map}}_{\mathcal{D}_\square(\ZZ)}(\ZZ[S], \underline{A})$ is discrete for any $S \in *_\proet$ and $A \in \mathrm{Ab}$:
		By \cite[pg.39, Formula (6.1)]{CS26} we have the formula
		\begin{equation*}
			\underline{\mathrm{map}}_{\mathcal{D}_\square(\ZZ)}(\ZZ[S]^\square, \underline{A}) \simeq R\Gamma(\_ \times S, \underline{A})
		\end{equation*}
		By Fact \ref{H-1} it suffices to see that the right-hand side sends cofiltered limits of finite sets $\{X_i\}_{i \in \mathcal{I}}$ to colimits. But writing $S = \mathrm{lim}_{j \in \mathcal{J}}$ and the fact that $\underline{A}$ is discrete we obtain:
		\begin{equation*}
			R\Gamma(S \times \lim_{i \in \mathcal{I}} X_i, \underline{A}) \simeq \colim_{(i, j) \in \mathcal{I} \times \mathcal{J}} R\Gamma(X_i \times S_j, \underline{A}) \simeq \colim_{i \in \mathcal{I}} R\Gamma(X_i \times \mathrm{lim}_j S_j, \underline{A})
		\end{equation*}
		
		\textit{For (2)}: We claim that we have the formula
		\begin{equation}\label{Et-10}
			\colim_{n \in \NN} \mathrm{map}_{\mathcal{D}_\square(\ZZ_l)}(M/(\omega_n, l^n), \underline{A}) \to \mathrm{map}_{\mathcal{D}_\square(\ZZ_l)}(M, \underline{A})
		\end{equation}
		for any $A \in \mathcal{D}(\ZZ)_{l-\text{tors}}$ and $M$ in the finite limit closure of $\Lambda$. Since $M/(\omega_n, l^n)$ and $\underline{A}$ both lie in $\nu^* = \underline{\cdot}: \mathcal{D}(\ZZ) \subset \mathcal{D}_\square(\ZZ)$ the source simplifies to
		\begin{equation*}
			\colim_{n \in \NN} \mathrm{map}_{\mathcal{D}(\ZZ)}(M/(\omega_n, l^n), \underline{A})
		\end{equation*}
		But in the case $M = M_\infty$ and $A = \QQ_l/\ZZ_l$ this expression is concentrated in homological degree $0$, since $A$ is an injective abelian group.
		\textit{For \eqref{Et-10}}: The characterization of the essential image of $\nu^*$ \ref{H-1} in the case $B = \mathbf{1}$ in terms of classicality of homology sheaves gives us
		\begin{equation*}
			\colim_{n \in \NN} \mathrm{map}_{\mathcal{D}_\square(\ZZ_l)}(M/\omega_n, \underline{A}) \to \mathrm{map}_{\mathcal{D}_\square(\ZZ_l)}(M, \underline{A})
		\end{equation*}
		Now more generally let $\mathcal{C}$ be a stable compactly generated closed symmetric monoidal category, in which the compact objects are closed under tensor product. Let $x \in \mathcal{C}^\omega$ and $y \in \mathcal{C}_{l-\text{tors}}$, i.e., $\colim(y \xrightarrow{l} y ...) = 0$. Then the map
		\begin{equation*}
			\colim_{k \in \NN} \mathrm{map}(x/l^k, y) \to \mathrm{map}(x, y)
		\end{equation*}
		is an equivalence \footnote{We have the following map of fiber sequences
			\[\begin{tikzcd}[ampersand replacement=\&]
				\colim_{k \in \NN} \mathrm{map}(x/l^k, y) \arrow[r] \arrow[d] \& \colim_{k \in \NN} \mathrm{map}(x, y) \arrow["\mathrm{id}", loop above] \arrow[r, "\colim_k l^k"] \arrow[d, "="] \& \colim_{k \in \NN} \mathrm{map}(x, y) \arrow["l", loop above] \arrow[d] \\
				\mathrm{map}(x, y) \arrow[r, phantom, "="] \& \mathrm{map}(x, y) \arrow[r] \& 0
			\end{tikzcd}\]
			Now the top right part of the diagram is $0$: By compactness of $x$ and generation of the category by compacts closed under $\otimes$ we can move the colimit past $\mathrm{map}(x, \_)$. But $\colim(y \xrightarrow{l} y \xrightarrow{l} ...) \simeq 0$, since $y$ is $l$-torsion.
		}.
	\end{proof}
\end{lem}
	\section{Relating K-theory and Iwasawa theory}
		\subsection{Ètale cohomology in terms of the Iwasawa module $M_\infty$ and the torsion of $M_\infty$}\label{C1}
			\begin{notat}
	Denote by $\iota: \mathrm{Spec}(\kappa(\mathfrak{p})) \to \mathrm{Spec}(R)$ the inclusion of the chosen residue field $\FF = \kappa(\mathfrak{p})$. As basechange along it is a left exact continuous functor it induces an adjunction
	\begin{equation*}
		\iota_*: \mathrm{Shv}(\kappa(p)_\proet, \mathcal{D}(\ZZ)) \leftrightharpoons \mathrm{Shv}(R_\proet, \mathcal{D}(\ZZ)): \iota^*
	\end{equation*}
	
	Specializing to the case $R_\infty/R$ in Corollary \ref{Et-11}, we get another adjunction
	\begin{equation*}
		f_*: \mathrm{Shv}(R_\proet, \mathcal{D}(\ZZ)) \leftrightharpoons \mathrm{Shv}(B(\Gamma_F)_\proet, \mathcal{D}(\ZZ)): f^*
	\end{equation*}
	induced by the inclusion $f: \mathrm{Pro}(R_\text{fet}) \subset R_\proet$. Denote by $\ZZ_l(i)^\mathrm{red}$, respectively, $\ZZ/l^k(i)^\mathrm{red}$ for $k \in \NN \cup \{\infty\}$ the fibers of the following two maps:
	\begin{align*}
		\ZZ/l^k(i) &\to \iota_*\iota^*\ZZ/l^k(i) \simeq \iota_*\ZZ/l^k(i) \\
		\ZZ_l(i) &\to (\iota_*\iota^*\ZZ_l(i))^\wedge_l \simeq \iota_*\ZZ_l(i)
	\end{align*}
\end{notat}
\begin{thm}\label{M-1}
	In $\mathcal{D}_\square(\Lambda_F) \subset (\mathcal{D}_\square(\ZZ_l))^{B(\ZZ \times \Delta)}$
	\begin{equation*}
		f_*\ZZ_l(i)^\mathrm{red} \simeq \underline{\mathrm{map}}(M_\infty, \ZZ_l(i))[-1]
	\end{equation*}
\end{thm}
Upon taking global sections/continuous fixed points we obtain the following result recovering \cite[Prop.4.15 and Thm.1.7]{DM98} when passing to homology groups:
\begin{cor}\label{M-6}
	We have a fiber sequence
	\begin{equation*}
		\mathrm{map}_{\mathcal{D}_\square(\Lambda_F)}(M_\infty, \ZZ_l(i))[-1] \to R\Gamma_\proet(R, \ZZ_l(i)) \to R\Gamma_\proet(\kappa(\mathfrak{p}), \ZZ_l(i))
	\end{equation*}
	Since both $M_\infty$ and $\ZZ_l(i)$ lie in the essential image of $\mathrm{Perf}(\Lambda_F) \hookrightarrow \mathcal{D}_\square(\Lambda_F)$ we even have
	\begin{equation*}
		\mathrm{map}_{\mathcal{D}_\square(\Lambda_F)}(M_\infty, \ZZ_l(i)) \simeq \mathrm{map}_{\mathcal{D}(\Lambda_F)}(M_\infty, \ZZ_l(i))
	\end{equation*}
	\begin{proof}
		Under the inclusion $\mathcal{D}_\square(\Lambda) \subset \mathrm{Shv}(B\Gamma_\proet, \mathcal{D}(\ZZ))$ evaluating at $\Gamma_F/ (C \times \Gamma^{l^n})$ corresponds to taking continuous fixed points by $C \times \Gamma^{l^n}$ for any subgroup $C$ of $\Delta$. Thus, applying Theorem \ref{M-1} yields the desired result.
	\end{proof}
\end{cor}
\begin{proof}[Proof of Theorem \ref{M-1}]
	\textit{For the analogous statement for $\QQ_l/\ZZ_l$-coefficients}:
	We have a commutative square
	\begin{equation}\label{M-11}
	\begin{tikzcd}
		\mathrm{Shv}(B(\Gamma_F)_\text{et}, \mathcal{D}(\ZZ))_{l-\text{tors}} \arrow[d, hook] \arrow[r] & \mathcal{D}(\ZZ)_{l-\text{tors}}^{B\ZZ \times \Delta} \arrow[d, hook] \\
		\mathcal{D}_\square(\Lambda_F) \arrow[r, hook] & \mathcal{D}_\square(\ZZ)^{B(\ZZ \times \Delta)}
	\end{tikzcd}
	\end{equation}
	The left vertical map is an embedding by Fact \ref{H-1}; the right vertical map is an embedding since $\mathcal{D}(\ZZ) \to \mathcal{D}_\square(\ZZ)$ is fully faithful; the bottom horizontal map is an embedding by Lemma \ref{Cond-29}. The map
	\begin{equation}\label{M-2}
		f_*\QQ_l/\ZZ_l(i) \to (f\iota)_*\QQ_l/\ZZ_l(i)
	\end{equation}
	lies in the full subcategory $\mathrm{Shv}(BG_\text{et}; \mathcal{D}(\ZZ))_{l-\text{tors}}$ by Remark \ref{Et-13}. Hence, the square \eqref{M-11} of embeddings implies that the object $f_*\QQ_l/\ZZ_l(i)^\mathrm{red}$ is simply determined by the action of a topological generator $\gamma \in \Gamma$ and the finite group $\Delta$ on its underlying object in $\mathcal{D}_\square(\ZZ)$.\newline\newline
	\noindent\fbox{%
		\parbox{\textwidth}{%
			\textbf{Main observation}: By the cohomology computation in Lemma \ref{2-1} and Lemma \ref{Et-27}, the fiber of the morphism in \eqref{M-2} coincides with the mapping object $\underline{\mathrm{map}}_{\mathcal{D}_\square(\Lambda)}(M_\infty, \QQ_l/\ZZ_l(i))[-1]$, which is underlyingly just $\underline{\mathrm{map}}_{\mathcal{D}_\square(\ZZ)}(M_\infty, \QQ_l/\ZZ_l(i))[-1]$.
			In more detail, the two lemmas and the discussion above tell us:
			\begin{itemize}
				\item Forgetting the $\Gamma_F$-action, both lie in the essential image of $\mathcal{D}(\ZZ) \subset \mathcal{D}_\square(\ZZ)$.
				\item Both are concentrated in homological degree $-1$.
				\item Both have the correct action by a topological generator $\gamma \in \Gamma$ and the finite group $\Delta$.
			\end{itemize}
		}%
	}
	\newline\newline
	\textit{To recover the desired result}: Applying $(\_)^\wedge_l[-1]$ to this fiber sequence, \eqref{M-2} turns into
	\begin{equation*}
		f_*\ZZ_l(i) \to (f\iota)_*\ZZ_l(i)
	\end{equation*}
	Moreover, the fiber identifies with $\underline{\mathrm{map}}_{\mathcal{D}_\square(\ZZ)}(M_\infty, \ZZ_l(i))[-1]$.
\end{proof}
\begin{rem}
	In \cite[Thm.3.4]{DM98}, Borel's result enters in the cohomology-group-level analogue of the above computation as follows: To solve an extension problem arising in \cite[Thm.3.8]{DM98} they need $\mathrm{H}^2_\et(R, \QQ_l/\ZZ_l(i)) = 0$ for $i \geq 2$. This is, of course implied by Borel, which implies finiteness of $\mathrm{H}^2_\et(R, \ZZ_l(i))$ and $i \geq 1$ (see \eqref{I-6} for the relation of $K$-theory and $l$-adic étale cohomology). We circumvent this by identifying $R\Gamma_\et(R_\infty, \ZZ_l(i))$ as a continuous module over the Galois group of $F_\infty/F$. Instead, the main computational input for this is $\mathrm{H}^2_\et(R_\infty, \QQ_l/\ZZ_l(i)) = 0$. As we showed, this is a simple consequence of class field theory and we need not use Borel's computation.
\end{rem}
As an application of Theorem \ref{M-1} we can also prove a formula (which is also a consequence of \cite[Lemma 8.8]{DM98}) for the characteristic polynomial of $M_\infty$ in terms of the Iwasawa module $A'_\infty$. When $k$ is taken to be even, the corollary below also recovers a theorem of Iwasawa from \cite{Iw59}, in the formulation found in \cite[Prop.2.6]{Gr01}:
\begin{cor}\label{M-5}
	Denote by $S: \Lambda \to \Lambda$ the antipode of the Hopf algebra $\Lambda$. Then
	\begin{equation*}
		\mathrm{char}((\epsilon_k M_\infty)(-1)) = \begin{cases}
			S(\mathrm{char}(\epsilon_{1-k}A'_\infty)) &\text{, if } k \not\equiv 1 + \abs{\Delta}\ZZ \\
			S(\mathrm{char}(\epsilon_0A'_\infty))(T^{\vert S_t \vert-1}) &\text{, else}
		\end{cases}
	\end{equation*}
	for any $t \in \NN$.
	\begin{proof}
		We claim that in $\mathcal{D}(\ZZ)$:
			\begin{equation}\label{M-4}
				\begin{aligned}
					\lim\limits_{n \in \NN^\op} R\Gamma_\et(R_n, \ZZ/l^n(1)^\mathrm{red}) &\simeq \lim\limits_{k, n \in \NN^\op} R\Gamma_\et(R_n, \ZZ/l^k(1)^\mathrm{red}) \simeq \lim\limits_{n \in \NN^\op} R\Gamma_\et(R_n, \ZZ_l(1)^\mathrm{red}) \\
					&\simeq \mathrm{map}_{\mathcal{D}(\Lambda_F)}(M_\infty(-1), \Lambda_F)[-1]
				\end{aligned}
			\end{equation}
			Here the limit in the middle diagram is taken along the trace maps for the finite extensions $R_{n+1}/R_n$. The $\Gamma_F$-action on the left, by acting on the tower $... R_{n+1} \supset R_n \supset ...$, corresponds to the $\Gamma_F$-action twisted by the antipode on the right.
			\begin{proof}[Proof of \eqref{M-4}]
				Conceptually, we are just computing $\Lambda_F \otimes f_*(\ZZ_l(1)^\mathrm{red})$ with point-wise solid tensor product in two ways.
				\begin{enumerate}
					\item The first equivalence follows as $\NN^\op$ is a cosifted category. $\NN$ is even filtered.
					\item The second equivalence comes from $\lim_k\ZZ/l^k(i)^\mathrm{red} \simeq \ZZ_l(i)^\mathrm{red}$ by definition of the $\mathrm{red}$-superscript.
					\item Observe that $\Lambda_F \simeq \lim_n\Lambda_F/\omega_n$ witnesses that $\Lambda_F$ is a limit over dualizable objects in $\mathcal{D}_\square(\Lambda_F)$ with respect to the pointwise monoidal structure. Thus in $\mathcal{D}_\square(\Lambda_F)$:
					\begin{equation}\label{M-14}
						\underline{\mathrm{map}}(M_\infty(-1), \Lambda_F) \simeq \lim\limits_{n \in \NN^\op} \underline{\mathrm{map}}(M_\infty(-1), \Lambda_F/\omega_n) \xleftarrow{\simeq} \lim\limits_{n \in \NN^\op} (\underline{\mathrm{map}}(M_\infty(-1), \ZZ_l) \otimes \Lambda_F/\omega_n) 
					\end{equation}
					Denote by $f_n: \mathrm{Spec}(R_{n+1}) \to \mathrm{Spec}(R_n)$ the transition maps. Recall that we can identify
					\begin{equation*}
						f^*(\ZZ[\mathrm{Gal}(F_n/F)])\simeq \ZZ[R_n]
					\end{equation*}
					naturally in $n \in \NN$; the right-hand side is natural in $n$ along the $\Gamma_F$-equivariant surjections $\mathrm{Gal}(F_{n+1}/F) \to \mathrm{Gal}(F_n/F)$, the left-hand side is natural in $n$ along $f_n$.
					
					Since $\Lambda_F/\omega_n \simeq \ZZ_l \otimes \ZZ[\mathrm{Gal}(F_n/F)]$, we have:
					\begin{equation}\label{M-7}
						\lim\limits_{n \in \NN^\op} \underline{\mathrm{map}}(M_\infty(-1), \ZZ_l) \otimes \Lambda_F/\omega_n)[-1] \simeq \lim\limits_n f_*(\ZZ_l(1)^\mathrm{red}) \otimes \Lambda_F/\omega_n \xleftarrow{\mathrm{proj}} \lim\limits_n f_*(\ZZ_l(1)^\mathrm{red} \otimes \ZZ[R_n])
					\end{equation}
					natural in $n$. By Remark \ref{T-5} we have commutative diagrams
					\[\begin{tikzcd}[column sep=huge]
						\ZZ_l(i)^\mathrm{red} \otimes \ZZ[R_{n+1}] \arrow[d, "f_n"] \arrow[r, "\mathrm{id} \otimes \mathrm{comp}_{R_{n+1}}", "\simeq"'] & \ZZ_l(i)^\mathrm{red} \otimes \ZZ[R_{n+1}]^\vee \arrow[r, "\mathrm{proj}", "\simeq"'] & (f_{n+1})_*(f_{n+1})^*(\ZZ_l(i)^\mathrm{red}) \arrow[d, "\epsilon"] \\
						\ZZ_l(i)^\mathrm{red} \otimes \ZZ[R_n] \arrow[r, "\mathrm{id} \otimes \mathrm{comp}_{R_n}", "\simeq"'] & \ZZ_l(i)^\mathrm{red} \otimes \ZZ[R_n]^\vee \arrow[r, "\mathrm{proj}", "\simeq"'] & (f_n)_*(f_n)^*(\ZZ_l(i)^\mathrm{red})
					\end{tikzcd}\]
					This shows that global sections of the last term in \eqref{M-7} is equivalent to the claimed limit of global sections along the trace maps. Up until \eqref{M-7} everything was $\Gamma_F$-equivariant when acting on the second tensor factor without twist. The twist comes in, since by Remark \ref{F-40} we have commutative diagrams:
					\[\begin{tikzcd}
						\ZZ[R_n] \arrow[r, "\mathrm{comp}_{R_n}"] \arrow[d, "\gamma"] & \ZZ[R_n]^\vee \\
						\ZZ[R_n] \arrow[r, "\mathrm{comp}_{R_n}"] & \ZZ[R_n]^\vee \arrow[u, "\gamma^\vee"]
					\end{tikzcd}\]
					for $\gamma \in \Gamma_F$.
					
					By combining \eqref{M-14} and \eqref{M-7} we arrive at the desired equivalence, also $\Gamma_F$-equivariantly.
				\end{enumerate}
			\end{proof}
			\textit{Passing to cohomology}: Taking $\pi_{-2}$ of the expression in \eqref{M-4} the right-hand side just becomes 
			\begin{equation*}
				S^*\mathrm{Ext}^1_{\Lambda_F}(M_\infty(-1), \Lambda_F)
			\end{equation*}
			as a $\Lambda_F$-module. Here $S^*$ denotes restriction of scalars along the antipode. So we arrive at the following chain of equivalences of $\Lambda_F$-modules: 
			\begin{equation}\label{M-8}
				S^*\mathrm{Ext}^1_{\Lambda_F}(M_\infty(-1), \Lambda_F) \simeq \mathrm{lim}_n \mathrm{H}^2(R_n, \ZZ/l^n(1)^\mathrm{red}) \simeq \mathrm{lim}_{n, k} \mathrm{H}^2(R_n, \ZZ/l^k(1)^\mathrm{red}) \simeq \mathrm{lim}_n \mathrm{H}^2(R_n, \ZZ_l(1)^\mathrm{red})
			\end{equation}
			Here we used the following:
			\begin{enumerate}
				\item $\pi_*R\Gamma_\mathrm{et}(R_n, \ZZ/l^n(i)^\mathrm{red})$ is degree-wise finite, hence, the $\mathrm{lim}^1$-term does not contribute. This settles the first equivalence.
				\item For the second equivalence we used, again, that $\NN^\op$ is cosifted.
				\item For the third equivalence, we again use the $\mathrm{lim}^1$-sequence applied to the global sections of $R_n$ of $\ZZ_l(1)^\mathrm{red} = \mathrm{lim}_k \ZZ/l^k(1)^\mathrm{red}$.
			\end{enumerate}
			\begin{claim}\label{M-13}
			$\mathrm{lim}_n \mathrm{H}^2(R_n, \ZZ_l(1)^\mathrm{red})$ and $\mathrm{lim}_n \mathrm{H}^2(R_n, \ZZ_l(1))$ agree
			\end{claim}
			By the long exact sequence associated to $\ZZ_l(1)^\mathrm{red} \to \ZZ_l(1) \to \iota_*\ZZ_l(1)$ the possible difference in a finite stage $n$ between the two comes from:
			\begin{enumerate}
				\item $\mathrm{H}^2(\FF_n, \ZZ_l(1))$; but this group is zero: By proétale descent (see Addendum of Lemma \ref{Cond-29}) we have a fiber sequence
				\begin{equation}\label{M-16}
					R\Gamma(\FF_n, \ZZ_l(1)) \to R\Gamma(\FF_\infty, \ZZ_l(1)) \xrightarrow{\gamma^{l^n}-1} R\Gamma(\FF_\infty, \ZZ_l(1))
				\end{equation}
				But $R\Gamma(\FF_\infty, \ZZ_l(1)) \simeq \ZZ_l(1)[0]$ is the first Tate twist of $\ZZ_l$. Thus the long exact sequence on cohomology groups of \eqref{M-16} shows that $\mathrm{H}^2(\FF_n, \ZZ_l(1)) = 0$.
				\item the cokernel of $(R_n^\times \to \FF_n^\times) \otimes \ZZ_l$: Note that both groups have the same $l$-torsion by choice of $\FF_n$: else $\FF_n$ would contain more roots of unity than $F_n$, leading to the contradiction $\FF_{n+1} = \FF_n$. But $\FF_n^\times$ is finite, hence, $\FF_n^\times \otimes \ZZ_l$ is its own $l$-torsion subgroup. Thus the cokernel is $0$.
			\end{enumerate}
			
			\textit{Conclusion}: By passing to inverse limits in $n \in \NN$ in Lemma \ref{5-2} we obtain a short exact sequence of $\Lambda_F$-modules
			\begin{equation}\label{M-9}
				0 \to A'_\infty \to \mathrm{lim}_n \mathrm{H}^2(R_n, \ZZ_l(1)) \to \mathrm{lim}_n B_n \to 0
			\end{equation}
			of $\Lambda_F$-modules. No $\mathrm{lim}^1$-term appears as the groups $A'_n$ are finite. Also by Corollary \ref{F-4} $B_n$ is a free $\ZZ_l$-module of rank $\vert S_n \vert-1$, which implies
			\begin{equation}\label{M-10}
				\mathrm{char}(\epsilon_i \lim_n B_n) = \begin{cases}
					(T)^{\vert S_t \vert-1} &\text{, if } i = 0 \\
					(1) &\text{, else}
				\end{cases}
			\end{equation}
			for any $t \in \NN$. Thus we conclude, also using $S$ as the notation for the antipode on $\Lambda_F$:
			\begin{align*}
				\mathrm{char}((\epsilon_kM_\infty)(-1)) &= \mathrm{char}(\epsilon_{k-1}(M_\infty(-1))) = \mathrm{char}(\epsilon_{k-1}\mathrm{Ext}^1_{\Lambda_F}(M_\infty(-1), \Lambda_F)) = \mathrm{char}(\epsilon_{k-1}S^*\lim_n \mathrm{H}^2(R_n, \ZZ_l(1))) \\
				&= S(\mathrm{char}(\epsilon_{1-k}\lim_n \mathrm{H}^2(R_n, \ZZ_l(1)))) = S(\mathrm{char}(\epsilon_{1-k}A'_\infty)\mathrm{char}(\epsilon_{1-k}B_t)) \\
				&= \begin{cases}
					S(\mathrm{char}(\epsilon_{1-k}A'_\infty)) &\text{, if } k \not\equiv 1 + \abs{\Delta}\ZZ \\
					S(\mathrm{char}(\epsilon_0A'_\infty))(T^{\vert S_t \vert-1}) &\text{, else}
				\end{cases}
			\end{align*}
			for any $t \in \NN$. Here we used:
			\begin{enumerate}
				\item Formulas \eqref{S-2} (compatibility of tensor products with taking $\epsilon_i$-summands) and \eqref{S-1} ($\epsilon_i$ summands of the twists $\ZZ_l(j)$) for the first equality.
				\item Remark \ref{H-5} stating that for every finitely generated $\Lambda_F$-module $N$ the underlying $\Lambda$-module satisfies:
				\begin{equation*}
					\mathrm{char}(N) = \mathrm{char}(\mathrm{Ext}^1_\Lambda(N, \Lambda)) = \mathrm{char}(\mathrm{Ext}^1_{\Lambda_F}(N, \Lambda_F))
				\end{equation*}
				This settles the second equality.
				\item \eqref{M-8} and Claim \ref{M-13} for the third equality.
				\item The facts $\epsilon_{1-k}S^*N = S^*\epsilon_{k-1}N$ and $S(\mathrm{char}{N}) = \mathrm{char}(S^*N)$ for the fourth equality.
				\item Multiplicativity of the characteristic ideal in short exact sequences of torsion $\Lambda$-modules applied to the $\epsilon_{1-k}$-summand in \eqref{M-9} for the fifth equality.
				\item Formula \eqref{M-10} and $S(T) = (T)$ for the last equality.
			\end{enumerate}
	\end{proof}
\end{cor}
\begin{rem}
	The above is the simplest proof we came up with to ensure that $\mathrm{lim}^1 R_n^\times = 0$ for the diagram along the trace maps is trivial. By formula \eqref{M-4} it suffices to show that the $\NN^\op$-indexed system
	\begin{equation}\label{M-20}
		\{\mathrm{Hom}_{\Lambda_F}(M_\infty(-1), \Lambda_F/\omega_n)\}_{n \in \NN^\op}
	\end{equation}
	has trivial $\mathrm{lim}^1$. On the other hand, it actually suffices that $\mathrm{char}((\epsilon_k M_\infty)(-1))$ and $S(\mathrm{char}(\epsilon_{1-k}A'_\infty))$ differ by a power of $\omega_m$ for some $m \in \NN$.
	
	Thus it would even be sufficient to prove that $\mathrm{lim}^1$ of the $\NN^\op$-system \eqref{M-20} is a $(\omega_m)^t$-module for some $m, t \in \NN$. Passing to $\Delta$-eigenspaces, this follows from Lemma \ref{F-25}: For any finitely generated $\Lambda$-module $N$, $\mathrm{lim}^1$ of the $\NN^\op$-indexed system
	\begin{equation*}
		\{\mathrm{Hom}_\Lambda(N, \Lambda/\omega_n)\}_{n \in \NN^\op}
	\end{equation*}
	is an $\omega_m^t$-module for some $m, t \in \NN$.
\end{rem}
\begin{rem}
	The classical proof for the odd eigenspaces of $M_\infty$ is much more straightforward \cite[Cor.3.3.14]{Sha}. The short exact used in the proof of that statement can also be recovered by passing to Pontryagin duals in \eqref{M-2} and then computing $\pi_1$. From that point-of-view one does not have full control over the difference between the characteristic ideal of $M_\infty$ and $A'_\infty$.
\end{rem}
		\subsection{The $\Lambda$-rank of $M_\infty$}\label{C2}
			We would like to generalize the following result \cite[Thm. 13.31]{Was97} to a statement on the $\Lambda$-ranks of the $\Delta$-eigenspaces $\epsilon_iM_\infty$:
\begin{prop}\label{7-4}
	In our general setting
	\begin{equation*}
		\rank_{\Lambda} M_\infty = r_2(F(\mu_l))
	\end{equation*}
\end{prop}
The goal is the following refinement:
\begin{lem}\label{7-2}
	We have
	\begin{equation*}
		\rank_{\Lambda} \epsilon_iM_\infty = \begin{cases} r_1(F) + r_2(F) &\text{, if $i$ odd} \\
		r_2(F) &\text{, if $i$ even}\end{cases}
	\end{equation*}
\end{lem}
For this refinement, we first consider the following description of the kernel of the restriction map $M_n \to A_n$ stemming from the idèle-theoretic version of class field theory (see \cite[Cor.13.6]{Was97}):
\begin{prop}\label{7-5}
	Let $n \in \NN$ and $\mathfrak{l} \in S_n$. Denote by $\CalO_{F_n, \mathfrak{l}}$ the $\mathfrak{l}$-adic completion of $\CalO_{F_n}$ and let $F_{n, \mathfrak{l}}$ be its fraction field. Then the idèle-level Artin map induces an equivalence
	\begin{equation*}
		(\prod\limits_{\mathfrak{l} \mid l \text{ in } F_n} \CalO_{F_n, \mathfrak{l}})/\widebar{\CalO_{F_n}^\times} \xrightarrow{\simeq} \ker(M_n \to A_n)
	\end{equation*}
	where the closure is taken with respect to the topology of the $\mathfrak{l}$-local units.
\end{prop}
\begin{rem}
	The map in Proposition \ref{7-5} is $\Delta$-equivariant, which is an easy consequence of the definition of the idèle-theoretic Artin map \cite[p.170 (d)]{Neu11}.
\end{rem}
\begin{lem}\label{7-9}
	Let $n \in \NN$. Denote by $U_{\mathfrak{l}} \coloneqq 1 + \mathfrak{l}\CalO_{F_n, \mathfrak{l}}$ the $\mathfrak{l}$-local $1$-units and define the global $1$-units $U_n \coloneqq \ker(\CalO_{F_n}^\times \to \prod_{\mathfrak{l}\mid l} \kappa(\mathfrak{l})^\times)$. Then the map
	\begin{equation*}
		(\prod\limits_{\mathfrak{l} \mid l} U_{\mathfrak{l}})/\widebar{U_n} \to (\prod\limits_{\mathfrak{l} \mid l} \CalO_{F_n, \mathfrak{l}}^\times)/\widebar{\CalO_{F_n}^\times}
	\end{equation*}
	becomes an equivalence upon rationalizing. Since $\prod_{\mathfrak{l} \mid l} U_{\mathfrak{l}} \subset \prod_{\mathfrak{l} \mid l} \CalO_{F_n, \mathfrak{l}}$ is clopen, $\widebar{U_n}$ is the closure of $U_n$ in either one of them.
	\begin{proof}
		Observe that $\widebar{U_n} = \widebar{\CalO_{F_n}^\times} \cap \prod_{\mathfrak{l} \mid l} U_{\mathfrak{l}}$: The inclusion $\prod_{\mathfrak{l} \mid l} U_{\mathfrak{l}} \subset \prod_{\mathfrak{l} \mid l} \CalO_{F_n, \mathfrak{l}}$ is clopen as product of inclusions of finite-index closed subgroups and therefore
		\begin{equation*}
			\widebar{\CalO_{F_n}^\times} \cap \prod_{\mathfrak{l} \mid l} U_{\mathfrak{l}} = \widebar{\CalO_{F_n}^\times \cap \prod_{\mathfrak{l} \mid l} U_{\mathfrak{l}}} = \widebar{U_n}
		\end{equation*}
		Thus the map is injective before rationalizing. But as $\prod_{\mathfrak{l} \mid l} U_{\mathfrak{l}} \subset \prod_{\mathfrak{l} \mid l} \CalO_{F_n, \mathfrak{l}}$ is of finite index, the map becomes an equivalence upon rationalizing.
	\end{proof}
\end{lem}
Motivated by Proposition \ref{7-5}, we prove the following refinement of the Dirichlet Unit Theorem:
\begin{lem}\label{7-8}
	For all $n \in \NN$ we have
	\begin{equation*}
		\dim_{\QQ_l} \epsilon_i(U_n \otimes_\ZZ \QQ_l) = \dim_{\QQ_l} \epsilon_i (\CalO_{F_n}^\times \otimes \QQ_l) \simeq \begin{cases}
			l^n(r_1(F) + r_2(F)) - 1 & \text{, if } i = 0 \\
			l^n(r_1(F) + r_2(F)) & \text{, if } i \neq 0 \text{ and } i \equiv 0 (2) \\
			l^nr_2(F) & \text{, else}
		\end{cases}
	\end{equation*}
	\begin{proof}
		Denote by $X$ the set of embeddings $F_n \to \widebar{\QQ}$ modulo complex conjugation. By the Dirichlet Unit Theorem we have a short exact sequence of $\RR[\Delta]$-modules
		\begin{equation}\label{7-10}
			0 \to \CalO_{F_n}^\times \otimes_\ZZ \RR \to \RR^X \to \RR^\text{triv} \to 0
		\end{equation}
		Choose a field $K$ sitting in a diagram
		\[\begin{tikzcd}
			\RR \arrow[r] & K \\
			\QQ \arrow[u]\arrow[r] \arrow[u] & \QQ_l \arrow[u]
		\end{tikzcd}\]
		Applying $\_ \otimes_\RR K$ to the sequence \eqref{7-10} \footnote{$\RR$ does not have enough roots of unity for applying the $\epsilon_i$.}, we obtain the short exact sequence of $K[\Delta]$-modules
		\begin{equation*}
			0 \to K \otimes_\ZZ \CalO_{F_n}^\times \to K^X \to K^\text{triv} \to 0
		\end{equation*}
		Moreover, the sequence is $\Delta$-equivariantly split via the diagonal on the right. So we need to study $X$ as a $\Delta$-set. This we do for each fiber of the map of $\Delta$-sets induced by the restriction along $F \subset F_n$
		\begin{equation*}
			X = \Hom_\QQ(F_n, \widebar{\QQ})/C_2 \xrightarrow{\vert_F} \Hom_\QQ(F, \widebar{\QQ})/C_2
		\end{equation*}
		Here $C_2$ acts from the right by complex conjugation and $\Delta$ from the left via pre-composition. Denote by $X_x$ the fiber over $C_2 \cdot x$ for $x \in \Hom_{\QQ}(F_n, \widebar{\QQ})$ and $\Hom_x(F_n, \widebar{\QQ})$ the fiber of the restriction map before taking $C_2$-orbits.
		\begin{claim}
			If $x$ is a real embedding, we have an equivalence of $\Delta$-sets
			\begin{equation*}
				X_x \simeq \coprod_{l^n} \Delta/C_2
			\end{equation*}
			\textit{Explanation}: Fix an element $y \in \Hom_x(F_\infty, \widebar{\QQ})$. As $x$ is real, restricting complex conjugation along $y$ gives rise to an element of order $2$ in $\Gal(F_\infty/F)$. Such an element is unique and is contained in $\Delta$ corresponding to $-1$ under $\Delta \hookrightarrow \ZZ_l^\times$. Thus, the notation $\Delta/C_2$ makes sense. Denote by $\tau$ this $2$-torsion element. Then, for all embeddings $y \in \Hom_x(F_n, \widebar{\QQ})$, we have $y \circ \tau = \bar{y}$.
			\begin{proof}
				By the explanation, the $C_2$-action of $\tau$ via pre-composition coincides with that of complex conjugation via post-composition. As complex conjugation acts trivially on $x$, the $C_2$-action descends to the fiber such that $X_x \simeq \Hom_x(F_n, \widebar{\QQ})/C_2$. As $\Delta \subset \Gal(F_n/F)$ acts freely on $\Hom_x(F_n, \widebar{\QQ})$, we arrive at the desired result.
			\end{proof}
		\end{claim}
		\begin{claim}
			If $x$ is a non-real embedding, we have an equivalence of $\Delta$-sets
			\begin{equation*}
				X_x \simeq \coprod_{l^n} \Delta
			\end{equation*}
			\begin{proof}
				In this case the fiber $\Delta$-equivariantly identifies with $\Hom_x(F_n, \widebar{\QQ})$, as the $C_2$-action on $\{x, \bar{x}\}$ is non-trivial. The restriction map
				\begin{equation*}
					\Hom_x(F_n, \widebar{\QQ}) \to \Hom_x(F_0, \widebar{\QQ})
				\end{equation*}
				is $\Delta$-equivariant. The right-hand side identifies with $\Delta$ as $\Delta$-set. Hence, the action on the left-hand side is also free, i.e., every orbit $\simeq \Delta$.
			\end{proof}
		\end{claim}
		Combining these two claims, we obtain
		\begin{equation*}
			K^X \simeq (\bigoplus\limits_{r_1(F)} K^{\Delta/\langle \tau \rangle} \oplus \bigoplus\limits_{r_2(F)} K^{\Delta})^{l^n}
		\end{equation*}
		An element $x \in \epsilon_iK^{\Delta}$ and $x \in \epsilon_iK^{\Delta/\langle \tau \rangle}$ must satisfy $x_{\sigma} = \omega^i(\sigma)x_1$ for $\sigma \in \Delta$. Since in the second case this is only possible if $\omega^i(\tau) = 1$, i.e., $i$ is even, we arrive at the following formulas:
		\begin{enumerate}
			\item $\epsilon_iK^{\Delta} \simeq K$ for all $i$
			\item $\epsilon_iK^{\Delta/\langle \tau \rangle} \simeq \begin{cases}
				K &\text{, if } i \equiv 0 (2) \\
				0 &\text{, if } i \equiv 1 (2)
			\end{cases}$
		\end{enumerate}
		This yields the desired result.
	\end{proof}
\end{lem}
We do the analogous computation for $\prod_{\mathfrak{l} \mid l} U_\mathfrak{l}$:
\begin{lem}\label{7-6}
	\begin{equation*}
		\dim_{\QQ_l} \epsilon_i((\prod\limits_{\mathfrak{l} \mid l} U_{\mathfrak{l}}) \otimes_{\ZZ_l} \QQ_l) = l^n[F:\QQ]
	\end{equation*}
	where we recall that $U_\mathfrak{l}$ comes with a $\ZZ_l$-module structure by taking powers\footnote{This is not true for $\CalO_{F_{n, \mathfrak{l}}}^\times$ as it contains torsion coprime to $l$, which is the reason why we proved Lemma \ref{7-9}.}.
	\begin{proof}
		Recall that $\Delta$ does not interchange the product factors, as the sequence
		\begin{equation*}
			... \to S_n \to S_{n-1} \to ...
		\end{equation*}
		is a chain of equivalences. This also shows that the Galois groups of $\mathfrak{l}$-local extensions $F_{n, \mathfrak{l}}/F_\mathfrak{l}$ and $F_n/F$ coincide.
		
		Instead of the global logarithm, we use the $\mathfrak{l}$-adic logarithm, which gives us a $\ZZ_l$-linear and $\Delta$-equivariant map from $U_{\mathfrak{l}}$ to $\mathfrak{l}$ identifying $1+\mathfrak{l}^k$ and $\mathfrak{l}^k$ for $k$ large enough \cite[Satz 5.5]{Neu06} \footnote{The equivariance uses that the local Galois groups act by continuous ring automorphisms, hence, commute with the defining power series for the $\mathfrak{l}$-adic logarithm.}. Hence, we obtain a $\Delta$-equivariant equivalence
		\begin{equation*}
			U_{\mathfrak{l}} \otimes_{\ZZ_l} \QQ_l \xrightarrow[\simeq]{\log} \CalO_{F_n, \mathfrak{l}} \otimes_{\ZZ_l} \QQ_l \simeq F_{n, \mathfrak{l}} \simeq F_\mathfrak{l}[\Gal(F_{n, \mathfrak{l}}/F_{\mathfrak{l}})] \simeq \QQ_l[\Gal(F_{n, \mathfrak{l}}/F_{\mathfrak{l}})]^{[F_\mathfrak{l}:\QQ_l]}
		\end{equation*}
		of $\Delta$-modules. Because of the splitting $\Gal(F_{n, \mathfrak{l}}/F_{\mathfrak{l}}) \simeq \Gal(F_{n, \mathfrak{l}}/F_{0, \mathfrak{l}}) \times \Delta$ the right-hand side is induced, and therefore
		\begin{equation*}
			\epsilon_i\CalO_{F_n, \mathfrak{l}}^\times \otimes_{\ZZ_l} \QQ_l \simeq \QQ_l[\Gal(F_{n, \mathfrak{l}}/F_{0, \mathfrak{l}})]^{[F_{\mathfrak{l}}: \QQ_l]} \otimes_{\QQ_l} \epsilon_i\QQ_l[\Gal(F_{0, \mathfrak{l}}/F_{\mathfrak{l}})] \simeq \QQ_l[\Gal(F_{n, \mathfrak{l}}/F_{0, \mathfrak{l}})] \otimes_{\QQ_l} \QQ_l^{[F_{\mathfrak{l}}: \QQ_l]}
		\end{equation*}
		Summing up all the $\mathfrak{l}$-components, we arrive at the desired formula, as we have
		\begin{equation*}
			\QQ_l \otimes_\QQ F_n \simeq \prod_{\mathfrak{l}} F_{n, \mathfrak{l}}
		\end{equation*}
	\end{proof}
\end{lem}
\begin{defin}
	We define the Leopoldt defect of $F_n$ at the prime $l$ to be
	\begin{equation*}
		\delta_n = \rank_\ZZ U_n - \rank_{\ZZ_l} \widebar{U_n}
	\end{equation*}
\end{defin}
The following is proven in \cite[Lem.13.30]{Was97}:
\begin{fact}[Weak Leopoldt Conjecture/Theorem]
	The $\delta_n$ are uniformly bounded in $n$.
\end{fact}
\begin{proof}[Proof of Lemma \ref{7-2}]
	By Lemma \ref{7-5}, finiteness of $A_n$ and Proposition \ref{7-9} we have an exact sequence of $\QQ_l$-modules
	\begin{equation*}
		0 \to K_n \to U_n \otimes_\ZZ \QQ_l \to (\prod_{\mathfrak{l} \mid l} U_{\mathfrak{l}}) \otimes_{\ZZ_l} \QQ_l \to M_n \otimes_{\ZZ_l} \QQ_l \to 0
	\end{equation*}
	\footnote{$U_n \otimes \ZZ_l \to \widebar{U_n}$ is surjective or equivalently $U_n \otimes \ZZ_l \to \prod_{\mathfrak{l}\mid l} U_{\mathfrak{l}}$ is closed, as a group homomorphism between profinite groups} with $\rank_{\ZZ_l} K_n = \delta_n$. Hence, combining the formulas in Lemmas \ref{7-8} and \ref{7-6} yields
	\begin{equation}\label{7-7}
		\rank_{\ZZ_l}\epsilon_i \Gal(L_n/F_n) \in [0, \delta_n] + l^n\begin{cases}
			r_2(F) + 1 & \text{, if } i = 0\\
			r_2(F) & \text{, if } i \text{ is even }\neq 0 \\
			r_1(F) + r_2(F) & \text{, if } i \text{ is odd}
		\end{cases}
	\end{equation}
	But by the introduction of \cite[Ch.13.5]{Was97} \footnote{$\mathcal{X}_\infty$ in Washington's notation is $M_\infty$ in our notation, $K_\infty$ is our $F_\infty$, $M_n$ is our $L_n$} we have
	\begin{equation*}
		M_\infty/\omega_nM_\infty \simeq \Gal(L_n/F_\infty)
	\end{equation*}
	and not $\Gal(L_n/F_n) = M_n$ on the right-hand side. Hence, the $\ZZ_l$-extension $F_\infty/F$ still contributes to the $\epsilon_0$-term in formula \eqref{7-7}. Thus
	\begin{equation*}
		\rank_{\ZZ_l}\epsilon_i M_\infty/\omega_nM_\infty \in [0, \delta_n] + l^n\begin{cases}
			r_2(F) & \text{, if } i \text{ is even} \\
			r_1(F) + r_2(F) & \text{, if } i \text{ is odd}
		\end{cases}
	\end{equation*}
	As the weak Leopoldt conjecture holds, we deduce that, by Lemma \ref{2-7}
	\begin{equation*}
		\rank_{\Lambda} \epsilon_iM_\infty = \begin{cases} r_2(F) & \text{, if } i \equiv 0 \mod 2 \\
		r_1(F) + r_2(F) & \text{, if } i \equiv 1 \mod 2
		\end{cases}
	\end{equation*}
	as desired.
\end{proof}
		\subsection{Partial recovery of Borel's computation}\label{C3}
				\subsubsection{A criterion for existence of an "algebraic" proof for Borel}
	As a reminder, Borel proved in \cite{Bor74} for $s \geq 2$:
	\begin{equation*}
			\dim_\QQ K_s(\CalO_F)_\QQ = \begin{cases}
				r_1(F) + r_2(F) &\text{, if } s \equiv 1 \mod 4 \\
				r_2(F) &\text{, if } s \equiv 3 \mod 4 \\
				0 &\text{, else}
			\end{cases}
	\end{equation*}
	Fix a topological generator $\gamma \in \Gamma$. Define $u \coloneqq c(\gamma)$ as the image of $\gamma$ under the cyclotomic character. This determines an identification $\ZZ_l[[\mathrm{Gal}(F_\infty/F_0)]] \simeq \ZZ_l[[T]]$. In this section, we deduce the following equivalent characterizations from the computations in the previous chapters:
	\begin{thm}\label{8-2}
		Let $i \geq 2$. Then the following are equivalent:
		\begin{enumerate}
			\item Borel's rank result holds in homotopical degree $s = 2i-1$.
			\item Borel's rank result holds in homotopical degree $s = 2i-2$ (equivalently, that $K_s(\mathcal{O}_F)$ is finite).
			\item $\mathrm{char}(\epsilon_{1-i}A'_\infty)$ does not have a zero at $u^{1-i}-1$ \footnote{$\mathrm{char}(\epsilon_{1-i}A'_\infty)$ is a principal ideal generated by some polynomial $P$, which we can evaluate at $u^{1-i}-1$. Since generators differ by a unit in $\ZZ_l^\times$ having zeroes at $u^{1-i}-1$ is independent of choice of generator $P$.}.
			\item $A'_\infty \otimes_{\Lambda_F} \QQ_l(1-i) \simeq (\epsilon_{1-i}A'_\infty) \otimes_\Lambda \QQ_l(1-i) = 0$.
		\end{enumerate}
	\end{thm}
	\begin{rem}\label{8-6}
		The most studied Iwasawa-module is $A_\infty$ in our notation. So naturally one might ask how $A'_\infty$ and $A_\infty$ are related:
		For each $n \in \NN$ we have a short right exact sequence
		\begin{equation*}
			\ZZ[S_n] \to A_n \xrightarrow{\text{res}_n} A'_n \to 0
		\end{equation*}
		since $\ZZ[S_n]$ is the free group of cycles lying over $l$ in $\mathcal{O}_{F_n}$. Taking limits we see that the kernel of $\text{res}: A_\infty \to A'_\infty$ is a trivial $\Gamma_F$-module, i.e., it only has non-trivial eigenspace for the trivial character and this eigenspace has characteristic ideal a power of $\gamma - 1 = T$. By multiplicativity of the characteristic ideal in short exact sequences of $\Lambda_F$-torsion modules, we obtain for $i \in \ZZ$
		\begin{equation*}
			\mathrm{char}(\epsilon_iA_\infty) = \mathrm{char}(\epsilon_iA'_\infty) \begin{cases}
				(T)^r &\text{, if } i = 0 \\
				1 &\text{, else}
			\end{cases}
		\end{equation*}
		for some $r \in \NN$. Since the polynomial $T$ does not have a zero at $u^{1-i} - 1$ for $i \neq 1$, we may thus replace $A_\infty'$ by $A_\infty$ in Theorem \ref{8-2}.
	\end{rem}
	\begin{proof}[Proof of Theorem \ref{8-2}]
		\textit{Reduction to étale cohomology computation}:
		As recalled in the introduction (see \ref{I-7}) we have a convergent spectral sequence
		\begin{equation*}
			E^2_{p, q} = \RH^{-p}(R, \ZZ_l(q/2)) \Rightarrow \pi_{p+q}(K^\et(R)\lcomp)
		\end{equation*}
		We argue that it degenerates on the $E^2$-page; note that we use homological Serre grading in the following:
		\begin{enumerate}
			\item $\mathrm{H}^0(R, \ZZ_l(i)) = 0$ for $i \geq 1$: By pro-Galois descent, we see that $\mathrm{H}^0(R, \ZZ_l(i))$ is just the $\Gamma_F$-fixed points of the $\Gamma_F$-module $\ZZ_l(i) = \mathrm{H}^0(R_\infty, \ZZ_l(i))$. But, for $i \neq 0$, there are no fixed points but $0$.
			\item $\mathrm{H}^p(R, \ZZ_l(i)) = 0$ for all $i \in \ZZ$ and $p \geq 3$: This was recalled in Fact \ref{Et-8} for $\ZZ/l^k(i)$-coefficients. By the $\mathrm{lim}^1$-sequence and the finiteness of $\ZZ/l^k(i)$-cohomology groups recalled in Fact \ref{Et-8}, we obtain the desired result.
		\end{enumerate}
		$\Rightarrow$ In homological Serre grading, the spectral sequence therefore has only two non-trivial columns in $p$-degrees $-1$ and $-2$ when one restricts to $q$-degree $1$ and upwards. Moreover, every second row is $0$. Since all non-trivial entries in the first and second quadrant are concentrated in the $p = 0$ row, there are no non-trivial differentials. Pictorially,
		\begin{center}
		\renewcommand{\arraystretch}{1.5}
		\begin{tabular}{c|c c c c c}
			& -3 & -2 & -1 & 0 & 1\\\hline
			2i & 0 & $\mathrm{H}^2(R, \ZZ_l(i))$ & $\mathrm{H}^1(R, \ZZ_l(i))$ & 0 & 0 \\
			2i-1 & 0 & 0 & 0 & 0 & 0 \\
			0 & 0 & $\mathrm{H}^2(R, \ZZ_l)$ & $\mathrm{H}^1(R, \ZZ_l)$ & $\mathrm{H}^0(R, \ZZ_l)$ & 0 \\
			-j & 0 & 0 & 0 & 0 & 0
		\end{tabular}
		\end{center}
		Rationally, $\mathrm{H}^*(\FF, \ZZ_l(i)) = 0$ if $i \neq 0$; so after rationalizing the difference between reduced and unreduced étale cohomology goes away: By proétale descent (see Addendum of Lemma \ref{Cond-29}) we have a fiber sequence
		\begin{equation}\label{M-12}
			R\Gamma(\FF_0, \ZZ_l(i)) \to R\Gamma(\FF_\infty, \ZZ_l(i)) \xrightarrow{\gamma-1} R\Gamma(\FF_\infty, \ZZ_l(i))
		\end{equation}
		Since $R\Gamma(\FF_\infty, \ZZ_l(i))_\QQ \simeq \QQ_l(i)[0]$ as $\Lambda_F$-module it suffices to see that $\gamma$ acts non-trivially on $\QQ_l(i)$ for the fiber to be rationally trivial. Since the order of $\Delta$ is invertible in $\ZZ_l$, $R\Gamma(\FF, \ZZ_l(i))_\QQ$ is a direct summand of $R\Gamma(\FF_0, \ZZ_l(i))_\QQ$, hence, trivial as well for $i \neq 0$. Thus, Corollary \ref{M-6} implies
		\begin{align*}
			\pi_{2i-1}(K^\text{et}(R)^\wedge_l)_\QQ &= \mathrm{Hom}_{\Lambda_F}(M_\infty, \QQ_l(i)) & \pi_{2i-2}(K^\text{et}(R)^\wedge_l)_\QQ &= \mathrm{Ext}^1_{\Lambda_F}(M_\infty, \QQ_l(i))
		\end{align*}
		for $i \geq 2$. Here we used that $\mathrm{Ext}^*_{\Lambda_F}(M_\infty, \_)$ commutes with filtered colimits, since $M_\infty$ is perfect. We can further simplify these formulas, since
		\begin{equation*}
			\mathrm{Ext}^*_{\Lambda_F}(M_\infty, \QQ_l(i)) \simeq \prod_j \mathrm{Ext}^*_\Lambda(\epsilon_jM_\infty, \epsilon_j\QQ_l(i)) \simeq \mathrm{Ext}^*_\Lambda(\epsilon_i M_\infty, \QQ_l(i))
		\end{equation*}
		for $* \in \NN$. The first equivalence comes from the splitting $\mathcal{D}(\Lambda_F) \xrightarrow[\prod_i\epsilon_i]{\simeq} \prod_\Delta \mathcal{D}(\Lambda)$ of categories by taking $\epsilon_i$-eigenspaces and the second equivalence comes from formula \eqref{S-1}.
		\newline\newline
		\textit{Homological algebra over $\Lambda \otimes \QQ$}:
		\begin{claim}\label{B-2}
			Let $N$ be a finitely generated $\Lambda$-torsion module. Then the following are equivalent:
			\begin{enumerate}
				\item $\mathrm{Ext}^1_\Lambda(N, \QQ_l(i)) = 0$
				\item $\mathrm{Hom}_\Lambda(N, \QQ_l(i)) = 0$
				\item $u^i-1$ is not a zero of $\mathrm{char}(N)$
			\end{enumerate}
			\begin{proof}
				By Observation \ref{H-5}, $N \otimes \QQ$ splits into a sum of elementary modules and still carries the information of the characteristic polynomial of $N$. Hence, we may assume $N = \Lambda/P$ for some polynomial $P \in \ZZ_l[T] \subset \Lambda$. But using the obvious projective resolution of $N$ in that case and $\QQ_l(i) = \Lambda_\QQ/(T - (u^i-1))$ we see that kernel and cokernel of
				\begin{equation*}
					\QQ_l \xrightarrow{P(u^i-1)} \QQ_l
				\end{equation*}
				compute $\Hom_\Lambda(N, \QQ_l(i))$ and $\mathrm{Ext}^1_\Lambda(N, \QQ_l(i))$, respectively. So either both of them are $1$-dimensional or $0$-dimensional, depending exactly on whether $P(u^i-1) = 0$ or $\neq 0$.
			\end{proof}
		\end{claim}
		\textit{Conclusion}:
		Denote by $T_i$ the $\Lambda$-torsion submodule of $\epsilon_iM_\infty$. By Observation \ref{H-5}, we have
		\begin{equation*}
			\epsilon_iM_\infty \otimes \QQ \simeq (\Lambda \otimes \QQ)^{\mathrm{rk}_\Lambda \epsilon_i M_\infty} \oplus T_i \otimes \QQ
		\end{equation*}
		Furthermore
		\begin{align}\label{C-2}
			\pi_{2i-1}(K^\text{et}(R)^\wedge_l)_\QQ &= \mathrm{Hom}_\Lambda(T_i \otimes \QQ, \QQ_l(i)) \oplus \QQ_l^{\mathrm{rk}_\Lambda(\epsilon_i M_\infty)} & \pi_{2i-2}(K^\text{et}(R)^\wedge_l)_\QQ &= \mathrm{Ext}^1_\Lambda(T_i \otimes \QQ, \QQ_l(i))
		\end{align}
		\noindent\fbox{%
			\parbox{\textwidth}{%
				Recall Borel's result:
				\begin{equation}
					\dim_\QQ \QQ \otimes K_s\CalO_F = \begin{cases}
						r_1(F) + r_2(F) &\text{, if } s \equiv 1 \mod 4 \\
						r_2(F) &\text{, if } s \equiv 3 \mod 4 \\
						0 &\text{, else}
					\end{cases}
				\end{equation}
				for $s \geq 2$. Recall what we proved in Lemma \ref{7-2}:
				\begin{equation*}
					\rank_{\Lambda} \epsilon_iM_\infty = \begin{cases} r_1(F) + r_2(F) &\text{, if $i$ is odd} \\
						r_2(F) &\text{, if $i$ even}\end{cases}
				\end{equation*}
				\textbf{Main observation}: The ranks predicted by Borel in degree $s = 2i-1$ coincide with those contributed by the $\Lambda$-rank of $\epsilon_i M_\infty$.
			}%
		}
		\newline
		
		Thus we conclude:\newline
		\textit{1 $\Leftrightarrow$ 2}: As observed in \eqref{C-2}, Borel in degree $2i-1$ equivalent to vanishing of $\mathrm{Hom}_\Lambda(T_i \otimes \QQ, \QQ_l(i))$ and Borel in degree $2i-2$ is equivalent to $\mathrm{Ext}^1_\Lambda(T_i, \QQ_l(i)) = 0$ for $i \geq 2$. But these statements are equivalent by Claim \ref{B-2}.
		\newline
		\textit{2 $\Leftrightarrow$ 3}: As observed in \eqref{C-2}, Borel in degree $2i-1$ is equivalent to $\mathrm{Hom}_\Lambda(T_i(-1), \QQ_l(i-1)) = 0$. Again by Claim \ref{B-2}, the latter is equivalent to $\mathrm{char}(T_i(-1))$ not having a zero at $u^{i-1}-1$. By Corollary \ref{M-5}
		\begin{equation*}
			\mathrm{char}(T_i(-1)) = S(\mathrm{char}(\epsilon_{1-i}A'_\infty))T^{r_i}
		\end{equation*}
		for some $r_i \in \NN$. $T$ evaluated at $u^{i-1}-1$ is non-zero for $i \geq 2$. Thus, Borel in degree $2i-1$ is equivalent to $S(\mathrm{char}(\epsilon_{1-i}A'_\infty))$ not vanishing at $u^{i-1}-1$. So it is equivalent to
		\begin{equation*}
			0 \neq S(\mathrm{char}(\epsilon_{1-i}A'_\infty))\vert_{\gamma = u^{i-1}} = \mathrm{char}(\epsilon_{1-i}A'_\infty)\vert_{\gamma = u^{1-i}} = \mathrm{char}(\epsilon_{1-i}A'_\infty)(u^{1-i}-1)
		\end{equation*}
		\textit{3 $\Leftrightarrow$ 4}: The statement is true for any torsion finitely generated $\Lambda_F \otimes \QQ$-module; this we may check on elementary modules with a single non-trivial $\Delta$-eigenspace $N = \epsilon_j N \simeq \Lambda/\mathrm{char}(N)$ for some $j$: Since $\QQ_l(1-i) = \QQ \otimes \epsilon_{1-i}\Lambda_F/(T-(u^{1-i}-1))$, 4 is equivalent to $1-i \not\equiv j \mod \abs{\Delta}$ or $\mathrm{char}(N)$ not having a zero at $u^{1-i}-1$. But that is exactly condition 3.
	\end{proof}
	\begin{rem}
		That Borel implies condition 3 in Theorem \ref{8-2} was already observed by Soulé in \cite{So87}. We are more interested in the converse.
	\end{rem}
	
	\subsubsection{Cases in which the Iwasawa-theoretic condition has an Iwasawa-theoretic proof}
	\begin{rem}\label{I-3}
		We collect some obvious consequences of Theorem \ref{8-2}:
		\begin{enumerate}
			\item Borel's result can be deduced "algebraically" in all but finitely many degrees, since characteristic ideals can only have finitely many roots and the elements $u^{1-k}-1$ are all distinct.
			\item We want to vary the odd prime $l$ in this section. Therefore, we denote by $A_{l, \infty}$ the unramified Iwasawa module for the $l$-cyclotomic extension, $A_\infty$ in our previous notation. Set
			\begin{equation*}
				\lambda_{F, \mathrm{min}} := \min_{l \text{ prime}} \lambda(A_{l, \infty}) \in \NN
			\end{equation*}
			The finite number of degrees in which the criterion fails to provide a proof for Borel is clearly bounded by $\lambda_{F, \mathrm{min}}$ from above, since Borel's result is an integral statement independent of the prime $l$.
			\item\label{I-3-1} If Borel holds for any intermediate number field $F \subset K \subset F_\infty$ in degrees $2i-1/2i-2$, it also holds for $F$ in degrees $2i-1/2i-2$: Observe that $F_\infty = K[\mu_{l^\infty}]$ by the containment assumption, and therefore the a priori two definitions of $A'_\infty$ for $F$ and $K$, respectively, coincide. Moreover, $\Lambda_K \subset \Lambda_F$ is naturally a subring by inclusion of Galois groups. Therefore, the map
			\begin{equation*}
				A'_\infty \otimes_{\Lambda_K} \QQ_l(1-i) \to A'_\infty \otimes_{\Lambda_F} \QQ_l(1-i)
			\end{equation*}
			is surjective. By the main Theorem \ref{8-2}, we therefore arrive at the desired result.
		\end{enumerate}
	\end{rem}
	Recall the following criterion for $A_\infty = 0$.
	\begin{lem}\label{I-4}
		Let $K_\infty/K$ be a $\ZZ_l$-extension of a number field $K$ such that
		\begin{enumerate}
			\item There is only one prime above $l$ in $K$ which is totally ramified.
			\item $l \nmid h_K$
		\end{enumerate}
		Then the inverse limit of $l$-Sylow subgroups of the class groups $A_\infty$ along the norm maps is trivial.
		\begin{proof}
			Identify $\ZZ_l[[\Gal(K_\infty/K)]] \simeq \ZZ_l[[T]]$, depending on a choice of generator. Denote by $A_n$ the $l$-Sylow subgroup of the class group of $K_n = K_\infty^{l^n\ZZ_l}$. Then under both hypothesis $A_n = A_\infty/\omega_n A_\infty$ by \cite[Lem.13.15]{Was97}. In the case $n = 0$, the second hypothesis implies $A_\infty = TA_\infty \subseteq (T, l)A_\infty$. Hence, by Nakayama, $A_\infty = 0$.
		\end{proof}
	\end{lem}
	\begin{ex}
		We call $l$ a regular prime if $l \nmid h_{\QQ(\zeta_l)}$. This is exactly Condition 2 in Lemma \ref{I-4} for the extension $\QQ(\mu_{l^\infty})/\QQ(\zeta_l)$. Condition 1 holds by ramification theory for cyclotomic extensions. An example of an odd regular prime number is $3$. The existence of infinitely many regular primes is not known.
	\end{ex}
	\begin{cor}
		Borel can be proven "algebraically" for $\QQ(\zeta_{l^k})$ for $l$ a regular prime and any $k \in \NN$. In particular, since there are regular primes (e.g., $3$), this gives an algebraic proof of Borel for $\QQ$.
		\begin{proof}
			In all of these cases $A_\infty = 0$ by the above discussion, hence, condition $3$ in Theorem \ref{8-2} is trivially satisfied.
		\end{proof}
	\end{cor}
	\begin{rem}[Bounds on the $\lambda$-invariant]
		In \cite[Cor.4.6]{Ang07}, the following bound on the $\lambda$-invariant of an abelian number field $F$ is proven:
		Let $N$ be the conductor of $F$, i.e., the smallest $N \in \NN$ such that $\QQ(\zeta_N) \supseteq F$. Let $n$ be maximal with $l^n \mid N$ and $d \coloneqq N/l^n$. Then for $\lambda_l(F)$ the $\lambda$-invariant of the unramified Iwasawa module of the $l$-cyclotomic extension $F_\infty^\Delta/F$\footnote{Here $F_\infty^\Delta$ denotes the fixed field of $F_\infty$ by the action of $\Delta$. So this is not the cyclotomic extension $F_\infty/F$, which we usually consider in this paper.}, satisfies
		\begin{equation*}
			\lambda_l(F) \leq 2(\frac{l-1}{2}\phi(d))^{\phi(l-1)+1}
		\end{equation*}
		This yields a bound for the number of degrees in which the algebraic argument as is might not cover Borel. This bound is unfortunately too large to prove Borel that way.
	\end{rem}
	One of the few instances in which more is known about the characteristic ideal/polynomial of $A_\infty$ is the classical Iwasawa main conjecture, proven by Wiles for totally real fields. Claiming that these statements and their proofs are "algebraic" is a bit of a stretch. Borel's proof is hard but still easier to understand for most people than the proof of the main conjecture. So the following should be taken as a thought experiment.
	
	We use the notation in \cite{Co77}:
	\begin{fact}
		Let $\mathfrak{m}$ be an ideal of $\CalO_F$. A Dirichlet character mod $\mathfrak{m}$ is a character $\chi: \text{Cl}(\mathfrak{m}) \to S^1 \subset \CC^\times$ of the ray class group $\text{Cl}(\mathfrak{m})$, where $\mathfrak{m}$ is minimal with this property. Via the Artin map of global class field theory, $\chi$ is a faithful character of the Galois group of some finite abelian extension $F_\chi$ of $F$ with conductor $\mathfrak{m}$. Associated to such a character $\chi$, there is a Hecke $L$-function defined on the right half-plane $\{s \in \CC \vert \Re(s) > 1\}$
		\begin{align*}
			L(s, \chi) = \prod\limits_{\mathfrak{p} \nmid \mathfrak{m}} \frac{1}{1-\chi(\mathfrak{p})N(\mathfrak{p})^{-s}}
		\end{align*}
		As for any type of complex $L$-function, there is an analytic continuation to $\CC\setminus\{1\}$ and a functional equation relating the values at the $L$-function and the substitution $s \rightsquigarrow 1-s$ as meromorphic functions.
		
		There is also a presentation in terms of partial zeta functions $\zeta_F(\sigma, s)$ for $\sigma \in \Gal(F_\chi/F)$ (or equivalently, $\sigma$ a class of idèles; the latter formulation can be found in Neukirch \cite[Ch.VII, Sect.6]{Neu06}):
		\begin{equation*}
			L(\chi, s) = \sum\limits_{\sigma \in \Gal(F_\chi/F)} \chi(\sigma) \zeta_F(\sigma, s) \in \mathbb{C}
		\end{equation*}
		The fact that the numbers $\zeta_F(\sigma, 1-n) \in \QQ$ for $n \geq 1$ and $F$ totally real was proven independently by Siegel and Klingen \cite{Kli62}. Choosing an identification of the roots of unity of $\widebar{\QQ}_l$ and $\CC$ we can define:
		\begin{equation*}
			L(\chi, 1-k) \coloneqq \sum\limits_{\sigma \in \Gal(F_\chi/F)} \chi(\sigma) \zeta_F(\sigma, 1-k) \in \QQ_l
		\end{equation*}
		
		Let $F$ be totally real. Then there exists an $l$-adic $L$-function $L_l:\ZZ_l \setminus \{1\} \to \CC_l$, associated to any character $\chi: \Delta \to \ZZ_l^\times$, uniquely determined by interpolating the values $L(\chi, 1-k) \in \QQ_l$ twisted by some extra term. In the case $\chi$ being an odd power of the Teichm\"uller character, this means
		\begin{equation*}
			L_l(\omega^k, 1-t) = L(\omega^{k-t}, 1-t)\underbrace{\prod\limits_{\mathfrak{l} \mid l}(1-\omega^{k-t}(\mathfrak{l})N(\mathfrak{l})^{t-1})}_{\neq 0}
		\end{equation*}
		for $t, k \in \NN$. As stated in \cite[Sect.5]{Co77}, there exists $G_i \in \ZZ_l[[T]]$ having the property
		\begin{equation*}
			L_l(\omega^{1-i}, s) = G_i(u^s-1)
		\end{equation*}
		for all $s \in \ZZ_l/\{1\}$.
	\end{fact}
	Now Iwasawa's main conjecture for $F$ a totally real field reads as follows \cite[Sect.5.1]{Co77}:
	\begin{thm}\label{C-1}
		For odd $i$ we have
		\begin{equation*}
			\mathrm{char}(\epsilon_iA_\infty) = \begin{cases}
				(T + (1-u))(G_i) &\text{, if } i = 1 \\
				(G_i) &\text{, else}
			\end{cases}
		\end{equation*}
	\end{thm}
	The main conjecture only treats odd eigenspaces of $A_\infty$. For the even eigenspaces we need the following conjecture of Greenberg; this conjecture is implied by the Kummer-Vandiver conjecture in the case $F = \QQ(\zeta_l)^+$ being the maximal totally real subfield of $\QQ(\zeta_l)$:
	\begin{conj}
		Let $F$ be totally real, and $F_\infty^{\Delta} = F_\text{cyc}/F$ be the cyclotomic $\ZZ_l$-extension of $F$ for a fixed prime $l$. Then the associated unramified Iwasawa module $\tilde{A}_\infty$ is finite.
	\end{conj}
	\begin{rem}
		Denote by $(\_)_{(l)}$ the localization at all primes $\neq l$. To apply Greenberg's conjecture, recall
		\begin{equation*}
			\bigoplus\limits_{i \text{ even}} \epsilon_i\Pic(\CalO_{F_n}[1/l])_{(l)} = (\Pic(\CalO_{F_n}[1/l])_{(l)})^{C_2} = (\Pic(\CalO_{F_n^+}[1/l])_{(l)})
		\end{equation*}
		for all $n$ \footnote{Indeed, the Hochschild-Serre spectral sequence for the sheaf $(\GG_m)_{(l)}$ (see \cite[Thm.14.9]{Mil80}) associated to the $C_2$-Galois covering $\CalO_{F_n}[1/l]/\CalO_{F_n^+}[1/l]$ has only one possibly non-trivial row consisting of the $C_2$-fixed points, since $(l, \abs{C_2}) = (l, 2) = 1$}.
		
		Taking limits, we obtain $A'_\infty = (A'_\infty)^{C_2} = \lim_n \Pic(\CalO_{F_n^+}[1/l]) \otimes_\ZZ \ZZ_l$. The latter is finite by Greenberg's conjecture.
	\end{rem}
	\begin{lem}
		Let $F$ be totally real. Assuming Greenberg's conjecture for $F$, there is an Iwasawa-theoretic proof of Borel for all intermediate fields of $F_\infty/F$.
		\begin{proof}
			By Remark \ref{I-3}, item \ref{I-3-1}, it suffices to prove Borel for $F_n$ for all $n \in \NN$. Thus it suffices to treat the case $F = F_n^+$ which implies $F_n = F_0$. In that case, $F_0/F$ is a $C_2$-extension.
			
			In this case, the characteristic ideals $\mathrm{char}(\epsilon_iA_\infty)$ for $i = 0, 1$ are exactly $\mathrm{char}(A_\infty^\pm)$. By assuming Greenberg's conjecture $\mathrm{char}(\epsilon_0A_\infty) = 1$. By the main conjecture for totally real fields (see \ref{C-1}) it remains to show that for $k \geq 2$ $\mathrm{char}(A_\infty^-)(u^{1-k}-1) \neq 0$: Observe that $(T + (1-u))\vert_{T = (u^{1-k}-1)} \neq 0$. Thus we get the chain of equivalences:
			\begin{align*}
				\mathrm{char}(A_\infty^-)(u^{1-k}-1) \neq 0 &\Leftrightarrow G_1(u^{1-k}-1) \neq 0 \Leftrightarrow L_l(1, 1-k) \neq 0 \\ 
				&\Leftrightarrow 0 \neq L(\omega^{-k}, 1-k)
			\end{align*}
			The following, proven in Lemma \ref{F-1}, finishes the proof.
			\begin{claim}
				Let $F$ be totally real, and let $L(\chi, s)$ be the generalized Dirichlet $L$-function associated to a generalized Dirichlet character $\chi$. Then
				\begin{equation*}
					L(\omega^i, 1-j) = 0 \Leftrightarrow \omega^i \text{ has sign } -(-1)^j \Leftrightarrow i + j \equiv 1 (2)
				\end{equation*}
			\end{claim}
		\end{proof}
	\end{lem}
		\appendix
		\section{Weight computation from this perspective}
		\begin{fact}
	Denote by $K[1/r]: \mathrm{Sch}_\mathrm{qcqs}^\op \to \mathrm{Sp}$ the $K$-theory presheaf pointwise localized at $r \in \NN$. By \cite{Gr92}, there exists an operation $\psi^r: K[1/r] \Rightarrow K[1/r]$ having two properties:
	\begin{enumerate}
		\item $\psi^r$ is a morphism of commutative algebras in $\mathrm{PSh}(\mathrm{Sch}_\mathrm{qcqs}; \mathrm{Sp})$
		\item There is homotopy of commutative monoids making
		\[\begin{tikzcd}
			\mathrm{Pic}(\_) \arrow[r, "(\_)^{\otimes r}"] \arrow[d] & \mathrm{Pic}(\_) \arrow[d] \\
			\Omega^\infty K[1/r] \arrow[r, "\psi^r"] & \Omega^\infty K[1/r]
		\end{tikzcd}\]
		commute. Here $\Omega^\infty: \mathrm{PSh}(\mathrm{Sch}_\mathrm{qcqs}; \mathrm{Sp}) \to \mathrm{PSh}(\mathrm{Sch}_\mathrm{qcqs}; \mathrm{An})$ denotes the precomposition by the underlying space functor $\Omega^\infty: \mathrm{Sp} \to \mathrm{An}$.
	\end{enumerate}
\end{fact}
\begin{rem}
	The Adams operations in topology are uniquely determined by the topological analogue of these two properties; this is a consequence of Snaith's theorem \cite{Sn79}. Using the motivic version of Snaith's theorem \cite{AI22} one can prove the uniqueness for algebraic K-theory as well.
\end{rem}
\begin{fact}\label{W-2}
	The category $\mathrm{Shv}(\mathrm{Sch}_\mathrm{qcqs, \et}; \mathrm{Sp})$ admits a $t$-structure whose heart is the category of abelian-group-valued sheaves \footnote{More generally, for any $\infty$-topos $\mathcal{X}$ its stabilization $=$ spectrum objects in $\mathcal{X}$ admits a $t$-structure with connective part $\mathrm{CGrp}(\mathcal{X})$ and heart $\mathrm{CGrp}(\tau_{\leq 0}\mathcal{X})$ \cite[Prop.1.3.2.7]{Lur18}.}. Denote by $\pi_q^\et(\_)$ the homotopy groups valued in $\mathrm{Shv}(\mathrm{Sch}_\mathrm{qcqs, \et}; \mathrm{Ab})$. Recall the following facts:
	\begin{enumerate}
		\item For $k \in \NN$
		\begin{equation*}
			\pi_q^\et(K^\text{et}/l^k) \simeq \begin{cases}
				\ZZ/l^k(q/2) &\text{, if } 2 \mid q \\
				0 &\text{, else}
			\end{cases}
		\end{equation*}
		A proof of this goes back to \cite[Sect.4]{Tho85}. Working through the proof, the identification happens as follows:
		The composite in $\mathrm{PSh}(\mathrm{Sch}_\mathrm{qcqs}; \mathrm{An})$
		\begin{equation}\label{W-1}
		\begin{tikzcd}
			B\mathbb{G}_m \arrow[r] & \Omega^\infty\tau_{\geq 1}^\et(K^\text{et}/l^k) \arrow[r] & B^2\pi_2^\text{et}(K^\text{et}/l^k)
		\end{tikzcd}
		\end{equation}
		becomes a map of commutative group objects, identifying the target as the cofiber by $l^k$ of the source.
		\item\label{W-2-1} Taking the limit shows that $(K^\et)^\wedge_l$ has a multiplicative filtration with associated graded given by
		\begin{equation*}
			\mathrm{gr}_q (K^\text{et})^\wedge_l \simeq \begin{cases}
				\ZZ_l(q/2) &\text{, if } 2 \mid q \\
				0 &\text{, else}
			\end{cases}
		\end{equation*}
		By \cite[Lemma 6.7]{BEM25}, the underlying $\mathbb{E}_1$-algebra is free on its degree-$2$ part over its degree-$0$ part.
		\item This is the filtration which gives rise to the étale $K$-theory spectral sequence.
	\end{enumerate}
\end{fact}
\begin{cor}\label{W-3}
	Let $k \in \NN$ be coprime to $l$. The Adams operations $\psi^k: (K^\text{et})^\wedge_l \to (K^\text{et})^\wedge_l$ restrict to the multiplicative filtration recalled in Fact \ref{W-2}. The operation on the associated graded is given by multiplication by $k^{q/2}$ on the degree-$q$ part.
	\begin{proof}
		By Item \ref{W-2-1}, it suffices to prove the claim in degree $2$:
		The diagram \eqref{W-1} is compatible with the Adams operation on the right and middle terms and taking $k$-fold power on the left. Hence, we arrive at the desired result.
	\end{proof}
\end{cor}
\textbf{Conclusion}: Recall that we have the following commutative square:
\[\begin{tikzcd}
	K(R)[1/k] \arrow[r] \arrow[d] & K(R)^\wedge_l \arrow[d] \\
	K(R)_\QQ \arrow[r] & (K(R)^\wedge_l)_\QQ
\end{tikzcd}\]
where the bottom map identifies the target as the basechange along $\QQ \to \QQ_l$ of the source.
\begin{enumerate}
	\item So we may compute the ranks of the Adams eigenspaces on $(K(R)^\wedge_l)_\QQ$. Since we are only interested in homotopy in degrees $\geq 2$, we may replace $K(R)$ by $K^\et(R)$.
	\item Since the Adams operations are compatible with the filtration we can compute the contribution of Adams eigenspaces to homotopy via the term-wise rationalized étale $K$-theory spectral sequence (cf. \ref{I-7}). By Corollary \ref{W-3}, we see that the $q$-th Adams summand exactly contributes to the $q$-th row.
	\item Recall that the étale $K$-theory spectral sequence is concentrated in rows $-1$ and $-2$ in $q$-degree $\geq 1$. As we already observed in the beginning of the proof of Theorem \ref{8-2}, the $q$-th row therefore only contributes to homotopy in degrees $2q-1/2q-2$.
\end{enumerate}
		\section{Number-Theoretical Miscellanea}\label{A2}
		\begin{lem}\label{F-2}
	$F_\infty/F$ is only ramified over $l$ and all primes in $F$ over $l$ ramify completely.
	\begin{proof}
		By \cite[Satz 10.3]{Neu06}, this is true for $F = \QQ$.
		
		Recall (see \cite[\href{https://stacks.math.columbia.edu/tag/0BU7}{Tag 0BU7}]{stacks-project}) that the restriction map
		\begin{equation*}
			\Gal(F_\infty/F) \xrightarrow[\simeq]{\text{restr}} \Gal(\QQ(\mu_{l^\infty})/\QQ(\mu_{l^\infty}) \cap F)
		\end{equation*}
		identifies the ramification and inertia groups of every $\mathfrak{l}$ in $F$ with those of $l = \mathfrak{l} \cap \QQ(\mu_{l^\infty})$. But $l$ completely ramifies in $\QQ(\mu_{l^\infty})/\QQ$ by \cite[Satz 10.3]{Neu06}.
	\end{proof}
\end{lem}
\begin{lem}\label{F-3}
	There exist infinitely many primes $\mathfrak{l}$ of $F$ that stay inert throughout $F_\infty/F$.
	\begin{proof}
		 For this we need the following fact about cyclotomic fields, which can be deduced from ramification theory for the finite extensions $\QQ(\zeta_{l^k})/\QQ$ \cite[Satz 10.3]{Neu06}.
		\begin{fact}
			No prime but $l$ ramifies. For $p \nmid l$ a prime, the inertia group $I \subset \Gal(\QQ(\mu_{l^\infty})/\QQ) \simeq \ZZ_l^\times$ of $p$ in $\QQ(\mu_{l^\infty})/\QQ$ is the closed subgroup generated by $p$.
		\end{fact}
		The restriction $\Gal(F_\infty/F) \simeq \Gal(\QQ(\mu_{l^\infty})/\QQ(\mu_{l^\infty}) \cap F)$ identifies the inertia and ramification groups of a prime $\mathfrak{p}$ in $R$ and those of $p \coloneqq \mathfrak{p} \cap \ZZ$. Thus, it remains to find a prime $\mathfrak{p}$ over a prime $p$ that is a topological generator of $\ZZ_l^\times \simeq \ZZ/(l-1) \times \ZZ_l$, i.e., $p$ that gets mapped to a generator in all finite quotients $\ZZ/(l-1) \times \ZZ_l/l^k$. By Nakayama for the ring $\ZZ_l$, it suffices to find $p$ that gets mapped to a generator of $(\ZZ_l/l^2)^\times$. But infinitely many such primes exist by Dirichlet's theorem on arithmetic progressions.
	\end{proof}
\end{lem}
\begin{lem}\label{F-22}
	Let $K$ be a number field. Denote by $I_K$ the group of divisors/fractional ideals and by $I_{K, l}$ the subgroup generated by ideals lying over $l$. Then there is a short exact sequence
	\begin{equation*}
		0 \to \mathrm{Br}(\mathcal{O}_K[1/l])[l^\infty] \to \QQ_l/\ZZ_l \otimes I_{K, l} \to \QQ_l/\ZZ_l \to 0
	\end{equation*}
	Both the Brauer group and $I_K$ have trace/pushforward and restriction maps making them both co- and contravariantly functorial in $K$. We claim that the injection in this sequence is compatible with them.
	\begin{proof}
		The short exact sequences are proven in \cite[Ch.II, Prop.2.1]{Mil06}. We quickly review the proof to convince the reader of the claimed functoriality:
		\newline
		\textit{Functoriality in the classical Brauer-Hasse-Noether (BHN) sequence}:
		For each finite place $v$ there is an identification $\mathrm{Br}(K_v) \xrightarrow[\mathrm{inv}_v]{\simeq} \QQ/\ZZ$ where $K_v$ denotes the (metric) completion of $K$ at $v$. Upon taking $l$-power torsion, for $l$ an odd prime, the BHN-Theorem states that
		\begin{equation*}
			0 \to \mathrm{Br}(K)[l^\infty] \xrightarrow{v^*} \bigoplus_v \mathrm{Br}(K_v)[l^\infty] \xrightarrow{\bigoplus_v \mathrm{inv}_v} \QQ_l/\ZZ_l \to 0
		\end{equation*}
		is a short exact sequence.
		
		Recall the following fact (see \cite[Ch.1, Sect.2]{Mil06}; there Milne proves that local fields give rise to a class formation; the statement below is part of the definition of a class formation):
		\begin{fact}
			Let $L/K$ be a field extension, and let $\tilde{v}$ be a finite place lying over $v$. Then the squares
			\[\begin{tikzcd}
				\mathrm{Br}(K_v) \arrow[d, "\mathrm{res}"] \arrow[r, "\mathrm{inv}_v"] & \QQ/\ZZ \arrow[d, "{[L_{\tilde v}:K_v]\cdot}"] & & \mathrm{Br}(L_{\tilde{v}}) \arrow[r, "\mathrm{inv}_v"] \arrow[d, "\mathrm{cores}"] & \QQ/\ZZ \arrow[d, sloped, phantom, "="] \\
				\mathrm{Br}(L_{\tilde{v}}) \arrow[r, "\mathrm{inv}_{\tilde{v}}"] & \QQ/\ZZ & & \mathrm{Br}(K_v) \arrow[r, "\mathrm{inv}_v"] & \QQ/\ZZ
			\end{tikzcd}\]
			commute.
		\end{fact}
		As a consequence, the map
		\begin{equation*}
			\bigoplus_{v \text{ place}} \mathrm{Br}(K_v) \xrightarrow[\bigoplus \mathrm{inv}_v]{\simeq} \QQ/\ZZ \otimes I_K
		\end{equation*}
		is compatible with restriction and trace/corestriction on both sides.
		\begin{fact}\label{F-22-1}
			The squares
			\[
			\begin{tikzcd}
				\mathrm{Br}(K) \arrow[d, "\mathrm{res}", shift left=1] 
				\arrow[r, "\oplus v^*"] 
				& \bigoplus_v \mathrm{Br}(K_v) \arrow[d, "\mathrm{res}", shift left=1] 
				& {} \\
				\mathrm{Br}(L) \arrow[u, shift left, "\mathrm{cores}"] \arrow[r, "\oplus {\tilde{v}}^*"] 
				& \bigoplus_v \mathrm{Br}(L \otimes_K K_v) 
				\arrow[r, phantom, "\simeq"] \arrow[u, shift left, "\mathrm{cores}"]
				& \bigoplus_v\bigoplus_{{\tilde{v}}\vert v} \mathrm{Br}(L_{\tilde{v}})
			\end{tikzcd}
			\]
			commute. The composites from top left to bottom right agree because it is just a commuting square of restriction/basechange maps. The bottom left to top right composites agree, since corestriction and restriction satisfy a basechange formula. The identification in the lower right corner comes from the identification $L \otimes_K K_v \simeq \prod_{\tilde{v}\vert v} L_{\tilde{v}}$.
		\end{fact}
		These two facts prove the desired functoriality in the classical Brauer-Hasse-Noether sequence.
		\textit{Functoriality in the restricted sequence}:
		By the proof of \cite[Prop.2.1]{Mil06} on page 164 the BHN-sequence for $\mathcal{O}_K[1/l]$ comes from restricting the BHN-sequence for $K$: We have an inclusion of short exact sequences:
		\[\begin{tikzcd}
			0 \arrow[r] & \mathrm{Br}(\mathcal{O}_K[1/l]) \arrow[d, hook] \arrow[r] & \QQ/\ZZ \otimes I_{K, l} \arrow[r, "\bigoplus_v \mathrm{inv}_v"] \arrow[d, phantom, sloped, "\subset"] & \QQ/\ZZ \arrow[d, "=", sloped, phantom] \arrow[r] & 0 \\
			0 \arrow[r] & \mathrm{Br}(K) \arrow[r, "v^*"] & \bigoplus_v \QQ/\ZZ \otimes I_K \arrow[r, "\bigoplus_v \mathrm{inv}_v"] & \QQ/\ZZ \arrow[r] & 0
		\end{tikzcd}\]
		So it remains to check that the restriction map $\mathrm{Br}(\mathcal{O}_K[1/l]) \to \mathrm{Br}(K)$ commutes with restriction and corestriction/trace in $K$. This follows for the same reason as in Fact \ref{F-22-1}.
	\end{proof}
\end{lem}
\begin{cor}\label{F-4}
	In the notation of Lemma \ref{F-22}, we have
	\begin{equation*}
		B_n \simeq \ker(I_{F_n, l} \xrightarrow{\mathrm{deg}} \ZZ_l)
	\end{equation*}
	which is compatible with trace and restriction maps. In particular, the trace induces an isomorphism $B_n \to B_{n-1}$.
	\begin{proof}
		Apply the left exact functor $\mathrm{Hom}_\text{Ab}(\QQ_l/\ZZ_l, \_)$ to the sequence in the preceding Lemma \ref{F-22}. The formula for $B_n$, including the claimed functoriality, follows from the short exact sequence:
		\begin{align*}
			\Hom_\Ab(\QQ_l/\ZZ_l, \QQ_l/\ZZ_l) &\simeq \lim_i \Hom_\Ab(\ZZ/l^i, \QQ_l/\ZZ_l[l^i]) \simeq \lim_i \Hom_\Ab(\ZZ/l^i, \ZZ/l^i) \\
			&\simeq \lim_i \ZZ/l^i \simeq \ZZ_l
		\end{align*}
		The claim about the trace $B_n \to B_{n-1}$ being an equivalence follows from the trace $I_{F_n, l} \to I_{F_{n-1}, l}$ being an equivalence; all the primes over $l$ completely ramify.
	\end{proof}
\end{cor}
\begin{lem}\label{F-1}
	Let $F$ be totally real, and $L(\chi, s)$ be the Hecke $L$-function associated to a Hecke character $\chi$. Then for $\chi$ a power of the Teichm\"uller character:
	\begin{equation*}
		L(\omega^i, 1-j) = 0 \Leftrightarrow \omega^i \text{ has sign } -(-1)^j \Leftrightarrow i + j \equiv 1 (2)
	\end{equation*}
	\begin{proof}
		We follow the notation of Neukirch \cite{Neu06} and directly specialize all formulas to the case where $F$ is totally real: First we need to introduce $L_\infty(\chi, s)$ \cite[Ch.VII]{Neu06}:
			Let $p = 0, 1$, depending on whether $\chi$ is odd or even. Set
			\begin{equation}\label{F-1-2}
				L_\infty(\chi, s) = (\pi^{-(s+p)/2}\Gamma((s+p)/2))^{r_1(F)}
			\end{equation}
			where $\Gamma$ is the usual $\Gamma$-function. Recall that $\Gamma(s)$ has poles only at $s = 0, -1, ...$. So $L_\infty(\chi, s)$ has a pole at $s = 1-k$ if and only if $2 \mid (1-k+p)$. This is the case if and only if $k$ and $p$ do not have the same parity. So if $\chi = \omega^i$ this is exactly the case if $p \equiv i (2)$. Thus, $L_\infty(\chi, s)$ having a pole at $1-j$ is equivalent to $i + j \equiv 1 (2)$.
			
			The claim we want to prove is a consequence of the functional equation of the $L$-function: Following Neukirch we first introduce the complete Hecke $L$-function without defining all non-zero terms:
			\begin{equation}\label{F-1-1}
				\Lambda(s, \chi) \coloneqq \underbrace{(\abs{d_F}N(\mathfrak{m}))^{s/2}}_{\neq 0}L_\infty(\chi, s)L(\chi, s)
			\end{equation}
			Then the functional equation \cite[Ch.VII.8]{Neu06} reads as follows
			\begin{equation*}
				\Lambda(s, \chi) = W(\chi)\Lambda(1-s, \bar{\chi})
			\end{equation*}
			with a constant $W(\chi) \in S^1$ as meromorphic functions in $s$. For $s > 1$ the left-hand side is a non-vanishing holomorphic function: By \eqref{F-1-1}, this reduces to the there occurring factors being holomorphic and non-vanishing. For the $L$-function, this follows from the converging Euler product representation of the $L$-function for these $s$. For $L_\infty(\chi, s)$ and $s > 1$, this follows from the $\Gamma$-functions appearing in \eqref{F-1-2} being non-vanishing and holomorphic for $\mathrm{Re}(s) > 1$. So the statement follows from the location of poles of $L_\infty(\chi, 1-s)$, which we covered in the first paragraph of this proof.
	\end{proof}
\end{lem}
\begin{lem}\label{F-25}
	Let $N$ be any finitely generated $\Lambda$-module. Then
	\begin{equation*}
		\mathrm{lim}^1_{n \in \NN^\op} \Hom_\Lambda(N, \Lambda/\omega_n)
	\end{equation*}
	is a $\Lambda/(\omega_m)^t$ module for some $m, t \in \NN$. This implies that its characteristic ideal divides $\omega_m^t$ for some $m, t \in \NN$.
	\begin{proof}
		$\mathrm{lim}^1_{n \in \NN^\op} \Hom_\Lambda(N, \Lambda/\omega_n)$ is again a finitely generated $\Lambda$-module as a submodule of $\mathrm{Ext}^1_\Lambda(N, \Lambda)$ witnessed by the $\mathrm{lim}^1$-sequence.
		
		\textit{Case of $N$ is sum of simple modules}: Then we can write $N = \Lambda^r \oplus \bigoplus \Lambda/(P^{s_P})$ for $P \in \Lambda$ prime and $s_P, r \in \NN$. Since $\mathrm{lim}^1$ and finite sums commute we are reduced to $N$ itself being a simple module \footnote{Of course, each of the finitely many summands might produce a different $m ,t$. But $\omega_m \mid \omega_{m+1}$ and therefore we may just pass to the maximum of the obtained $m$'s and $t$'s.}.
		\begin{itemize}
			\item $N = \Lambda$: Then the system identifies with the system $... \to \Lambda/\omega_n \to \Lambda/\omega_{n-1} \to ...$ with surjective transition maps. In this case the $\mathrm{lim}^1$-contribution is trivial.
			\item $N = \Lambda/(P^s)$ for $P \in \Lambda$ a prime coprime to all $\omega_n$: Then $\mathrm{Hom}_\Lambda(N, \Lambda/\omega_n)$ is both $P$ and $\omega_n$-torsion. But by assumption $P, \omega_n$ are coprime, which implies $\Lambda/(\omega_n, P)$ is finite. Thus $\mathrm{Hom}_\Lambda(N, \Lambda/\omega_n)$ is finitely generated over the finite ring $\Lambda/(\omega_n, P)$, hence, itself finite by Remark \ref{1-8}. But $\mathrm{lim}^1$ of a system of finite groups is trivial.
			\item $N = \Lambda/(P^s)$ with $P \mid \omega_m$ for some $m$: Then the groups $\mathrm{Hom}_\Lambda(N, \Lambda/\omega_n)$ are all $\Lambda/(\omega_m)^s$-modules. Thus $\mathrm{lim}^1$ of this $\NN^\op$-indexed system is a $\Lambda/(\omega_m)^s$-module as well.
		\end{itemize}
		
		\textit{Case $N$ contains no finite $\Lambda$-submodules \footnote{A maximal finite submodule exists, since $N$ is finitely generated and $\Lambda$ is Noetherian.}} Equivalently, there exists a short exact sequence of $\Lambda$-modules
		\begin{equation*}
			0 \to N \to N' \to F \to 0
		\end{equation*}
		where $F$ is finite and $N'$ is a sum of simple $\Lambda$-modules. The long exact sequence implies that we have an exact sequence
		\begin{equation*}
			\mathrm{Hom}_\Lambda(N', \Lambda/\omega_n) \to \mathrm{Hom}_\Lambda(N, \Lambda/\omega_n) \to \mathrm{Ext}^1(F, \Lambda/\omega_n)
		\end{equation*}
		By what we proved above, the $\Lambda$-module $\mathrm{Hom}_\Lambda(N', \Lambda/\omega_n)$ refines to a $\Lambda/(\omega_m)^t$-module for some $t, m \in \NN$ independent of $n$.
		\begin{claim}
			$\{\mathrm{Ext}^1_\Lambda(F, \Lambda/\omega_n)\}_{n \in \NN^\op}$ has a limit cofinal (the dual notion of cofinal \cite[\href{https://kerodon.net/tag/02NQ}{Tag 02NQ}]{Ker}) subsystem with transition maps being $0$.
			\begin{proof}
				By the $\mathrm{lim}^1$-sequence for computing $\mathrm{RHom}_\Lambda(F, \Lambda) = \mathrm{lim}_n\mathrm{RHom}_\Lambda(F, \Lambda/\omega_n)$ we have a surjection of abelian groups
				\begin{equation*}
					\mathrm{Ext}^1_\Lambda(F, \Lambda) \to \mathrm{lim}_n\mathrm{Ext}^1_\Lambda(F, \Lambda/\omega_n) \to 0 
				\end{equation*}
				But the left hand side term is $0$ by \cite[Lemma 7.5]{DM98}\footnote{Section 7 of the paper does not use Borel and is just about homological properties of the ring $\Lambda$.}. But $F$ is finite and therefore $\{\mathrm{Ext}^1_\Lambda(F, \Lambda/\omega_n)\}_{n \in \NN^\op}$ is a system of finitely generated $\ZZ/l^s$-modules for $s$ independent of $n$. Such abelian groups are Artinian - every descending chain of submodules stabilizes. Combining this with the fact that any $\NN^\op$-system of abelian groups with surjective transition maps has non-trivial limit, $\{\mathrm{Ext}^1_\Lambda(F, \Lambda/\omega_n)\}_{n \in \NN^\op}$ must have a limit cofinal subsystem with transition maps being $0$.
			\end{proof}
		\end{claim}
		Summarizing, we have a short exact sequence of $\NN^\op$-indexed systems
		\begin{equation*}
			0 \to \{A_n\}_{n \in \NN} \to \{\mathrm{Hom}_\Lambda(N, \Lambda/\omega_n)\}_{n \in \NN^\op} \to \{B_n\}_{n \in \NN^\op} \to 0
		\end{equation*}
		with
		\begin{itemize}
			\item $A_n$ a $\Lambda/(\omega_m)^t$-module for some $m, t\in \NN$ independent of $n$.
			\item $\{B_n\}_{n \in \NN}$ thinning out to a system with trivial transition maps.
		\end{itemize}
		Now the long exact $\mathrm{lim}^1$-sequence shows the desired result, since $\{B_n\}_n$ has both trivial limit and $\mathrm{lim}^1$.
		
		\textit{The general case}:
		Denote by $E$ the maximal finite $\Lambda$-submodule of $N$. Then its quotient $N'$ has no finite $\Lambda$-submodules. The long exact sequence of $\mathrm{Ext}$-groups gives rise to a short exact sequence
		\begin{equation*}
			0 \to \{\mathrm{Hom}_\Lambda(N', \Lambda/\omega_n)\}_n \to \{\mathrm{Hom}_\Lambda(N, \Lambda/\omega_n)\}_n \to \{B_n\}_n \to 0
		\end{equation*}
		of $\NN^\op$-indexed systems where $B_n = \mathrm{Im}(\mathrm{Hom}_\Lambda(N, \Lambda/\omega_n) \to \mathrm{Hom}_\Lambda(F, \Lambda/\omega_n))$. By the claim above $\{B_n\}_n$ also thins out to a system with trivial transition maps. Thus $\mathrm{lim}^1$ of $\{\mathrm{Hom}_\Lambda(N', \Lambda/\omega_n)\}_n$ and $\{\mathrm{Hom}_\Lambda(N, \Lambda/\omega_n)\}_n$ coincide.
	\end{proof}
\end{lem}
		\section{Bi- and Hopf algebras}
		The following construction appears in \cite[Variant 2.2.3]{Ra26}. We recall the construction to draw some consequences that are important to us. Denote by $\mathrm{coCAlg}(\_) := \mathrm{CAlg}(\_^\op)^\op$ the category of cocommutative coalgebras in the following:
\begin{constr}\label{H-3}
	Let $\mathcal{C}$ be a symmetric monoidal category and let $A \in \mathrm{coCAlg}(\mathrm{Alg}(\mathcal{C}))$ a cocommutative bialgebra. Then the category $\mathrm{Mod}_A(\mathcal{C})$ inherits a symmetric monoidal structure \cite[Variant 2.2.3]{Ra26}:
	
	\textbf{Explicit construction}: Denote by $\mathrm{Pr}^L_\text{str}$ the wide subcategory of $\mathrm{Pr}^L$ with morphisms the strongly continuous functors, i.e., those functors which have a colimit-preserving right adjoint. The usual equivalence $\mathrm{Pr}^L \simeq (\mathrm{Pr}^R)^\op$ factors through $\mathrm{Pr}^R \cap \mathrm{Pr}^L$ when restricted to $\mathrm{Pr}^L_\mathrm{str}$. Consider the symmetric monoidal functor $\Theta_\mathcal{C} = \mathrm{Mod}_{(\_)}(\mathcal{C}):\mathrm{Alg}(\mathcal{C}) \to \mathrm{Pr}^L_\text{str}$ taking a map of algebras $A \to B$ to $\_ \otimes_A B: \mathrm{Mod}_A(\mathcal{C}) \to \mathrm{Mod}_B$ (\cite[Thm.4.5.3.1]{Lur17} in the unstraightened form). This functor is symmetric monoidal, thus its opposite \cite[Rem.4.8.5.17]{Lur17} is as well. So $\mathrm{Mod}_{(\_)}(\mathcal{C})^\op: \mathrm{Alg}(\mathcal{C})^\op \to \mathrm{Pr}^R \cap \mathrm{Pr}^L \subset \mathrm{Pr}^L$ preserves commutative algebra objects and therefore takes the commutative coalgebra $A$ to a symmetric monoidal structure on $\mathrm{Mod}_A(\mathcal{C})$.
	
	From the construction one can deduce the following two core properties:
	\begin{enumerate}
		\item \textit{Restriction of scalars is symmetric monoidal}: In particular, the forgetful functor $\mathrm{Mod}_A(\mathcal{C}) \to \mathcal{C}$ is symmetric monoidal.
		\item \textit{Compatibility with symmetric monoidal basechange}: Let $f: \mathcal{C} \to \mathcal{D}$ be symmetric monoidal. We claim that the two composites of symmetric monoidal functors in the following square
		\begin{equation}\label{H-14}
		\begin{tikzcd}
			\mathrm{Alg}(\mathcal{C}) \arrow[r, "\mathrm{Mod}_{(\_)}(\mathcal{C})"] \arrow[d, "\mathrm{Alg}(f)"] & \mathrm{Mod}_\mathcal{C}(\mathrm{Alg}_{\mathbb{E}_0}(\mathrm{Pr}^L)) \arrow[d, "\_ \otimes_{\mathcal{C}} \mathcal{D}"] \\
			\mathrm{Alg}(\mathcal{D}) \arrow[r, "\mathrm{Mod}_{(\_)}(\mathcal{D})"] & \mathrm{Mod}_\mathcal{D}(\Alg_{\mathbb{E}_0}(\mathrm{Pr}^L))
		\end{tikzcd}
		\end{equation}
		are homotopic. Since the monoidal structure constructed by Raksit comes from passing to cocommutative coalgebra categories in the above, the identification $\mathrm{Mod}_A(\mathcal{C}) \otimes \mathcal{D} \to \mathrm{Mod}_{f(A)}(\mathcal{C})$ promotes to a symmetric monoidal one.
		
		\textit{Commutativity of \eqref{H-14}}: Firstly, recall how the functor $\mathrm{Mod}_{(\_)}(\mathcal{C})$ is obtained: In Higher algebra, Lurie introduces the fully faithful symmetric monoidal functor $\Theta$
	\begin{equation*}
		\Theta: \mathrm{Pr}^\mathrm{Alg} \to \Alg_{\mathbb{E}_0}(\mathrm{Pr}^\text{Mod})
	\end{equation*}
	which promotes to a map of cocartesian fibrations over $\mathrm{Alg}(\mathrm{Pr}^L)$ (\cite[Prop.4.8.5.1]{Lur17} in the special case of $\mathcal{K}$ being all small diagrams). Now $\mathrm{Mod}_{(\_)}(\mathcal{C})$ is the functor induced on fibers over $\mathcal{C}$.
	
	Given $f:\mathcal{C} \to \mathcal{D}$ a map in $\mathrm{CAlg}(\mathrm{Pr}^L)$, view it as a functor $\Delta^1 \to \mathrm{CAlg}(\mathrm{Pr}^L)$. Then pulling back $\Theta$ along this map and applying the straightening functor yields a natural transformation of functors $\Delta^1 \to \mathrm{Cat}_\infty$ or, equivalently, a square
	\[\begin{tikzcd}
		\mathrm{Alg}(\mathcal{C}) \arrow[r, "\mathrm{Mod}_{(\_)}(\mathcal{C})"] \arrow[d, "\mathrm{Alg}(f)"] & \mathrm{Mod}_\mathcal{C}(\mathrm{Alg}_{\mathbb{E}_0}(\mathrm{Pr}^L)) \arrow[d, "\_ \otimes_{\mathcal{C}} \mathcal{D}"] \\
		\mathrm{Alg}(\mathcal{D}) \arrow[r, "\mathrm{Mod}_{(\_)}(\mathcal{D})"] & \mathrm{Mod}_\mathcal{D}(\Alg_{\mathbb{E}_0}(\mathrm{Pr}^L))
	\end{tikzcd}\]
	of symmetric monoidal categories: Indeed,
	\begin{itemize}
		\item The claim about symmetric monoidality of the diagram follows from \cite[Thm.4.8.5.16]{Lur17}, where it is proven that $\Theta$ promotes to a symmetric monoidal functor making
		\begin{equation}\label{H-6}
		\begin{tikzcd}
			\mathrm{Pr}^\mathrm{Alg} \arrow[rr, "\Theta"] \arrow[dr] & & \mathrm{Pr}^\mathrm{Mod} \arrow[dl] \\
			& \mathrm{CAlg}(\mathrm{Pr}^L)
		\end{tikzcd}
		\end{equation}
		a diagram of symmetric monoidal categories such that cocartesian edges are preserved by the tensor product \cite[Thm.4.8.5.16+Def.2.1.2.13]{Lur17}. So the diagram \eqref{H-6} witnesses that the maps from top to bottom are symmetric monoidally cocartesian fibrations in the sense of \cite[Lemma 1.9 and Def.1.11]{Ram26}. By \cite[Thm.2.1]{Ram26}, which the author says is already contained in \cite{Lur17}, the straightening of such cocartesian fibrations admit a lax symmetric monoidal structure. These are the lax monoidal structures put on the $\mathrm{Alg}(\_)$ and $\mathrm{Mod}_{(\_)}(\mathcal{C})$ functor and the lax monoidal refinement of $\Theta_\mathcal{C}$.
		\item The horizontal morphisms are the correct ones by definition - $\mathrm{Mod}_{(\_)}(\mathcal{C})$ is the functor induced by $\Theta$ on fiber over $\mathcal{C}$.
		\item The vertical morphisms are also the correct ones by definition of the basechange functors of algebras and modules in $\mathrm{Pr}^L$ over $\mathcal{C}$.
	\end{itemize}
	\end{enumerate}
\end{constr}
\begin{defin}\label{Cond-15}
	Write $u: \mathrm{Mod}_A(\mathcal{C}) \to \mathcal{C}$ for the forgetful functor. Denote by $\underline{\mathrm{map}}_A(\_, \_)$ the internal mapping object for that symmetric monoidal structure, and $\underline{\mathrm{map}}(\_, \_)$ the internal mapping object of $\mathcal{C}$. There is a canonical map
	\begin{equation}\label{Cond-14}
		\underline{\mathrm{map}}(u)_A: u\underline{\mathrm{map}}_A(x, y) \to \underline{\mathrm{map}}(ux, uy)
	\end{equation}
	Indeed, by adjunction this corresponds to a map
	\begin{equation*}
		u(\underline{\mathrm{map}}_A(x, y) \otimes x) \simeq u\underline{\mathrm{map}}_A(x, y) \otimes ux \to uy
	\end{equation*}
	We choose this map to be $u$ applied to the counit of the $\otimes$-$\underline{\mathrm{map}}_A$-adjunction.
	
	We say $A$ is a Hopf algebra, if \eqref{Cond-14} is an equivalence. This condition is also called being closed symmetric monoidal.
\end{defin}
\begin{constr}
	Given a cocommutative bialgebra $A$ and a module $X$ over it. Define its shear map $s_{A, X}: A \otimes uX \to A \otimes X$ to be the $A$-module map corresponding to $\eta_\mathbf{1} \otimes \mathrm{id}_{uX}: uX \to A \otimes uX$, where $\eta$ is the unit of the $\_ \otimes A \dashv u$-adjunction.
\end{constr}
\begin{lem}\label{Cond-16}
	Let $A$ be a cocommutative bialgebra in $\mathcal{C} \in \mathrm{CAlg}(\mathrm{Pr}^L)$. Then the following are equivalent:
	\begin{enumerate}
		\item $A$ is a Hopf algebra in the sense of \ref{Cond-15}.
		\item The shear map $s_A$ is invertible.
		\item $A$, as a bialgebra in $h\mathcal{C}$, is a Hopf algebra in the usual sense (see \cite[5.2]{BLV11} for a definition).
	\end{enumerate}
	\begin{proof}
		By assumption, we have a transformation natural in $Y \in \mathrm{Mod}_A(\mathcal{C})$:
		\begin{equation}\label{H-2}
			\begin{aligned}
			\mathrm{Map}_A(A \otimes X, Y) \simeq \mathrm{Map}_A(A, \underline{\mathrm{map}}_A(X, Y)) &\simeq \mathrm{Map}(\mathbf{1}_\mathcal{C}, u\underline{\mathrm{map}}_A(X, Y)) \xrightarrow[\simeq]{\underline{\mathrm{map}}(u)_*} \mathrm{Map}(\mathbf{1}, \underline{\mathrm{map}}(uX, uY)) \\
			&\simeq \mathrm{Map}(uX, uY) \simeq \mathrm{Map}(A \otimes uX, Y)
			\end{aligned}
		\end{equation}
		We claim it coincides with restriction along the shear map. By Yoneda, it suffices to treat the case $Y = A \otimes X$: We have the commutative diagram in which we omit the associativity and commutativity witnesses for readability:\newline
		\begin{adjustbox}{center}
		\begin{tikzcd}[ampersand replacement=\&, column sep=10]
			\mathrm{Map}_A(A \otimes X, A \otimes X) \arrow[dddd, bend right=100, "u"'] \arrow[rr, no head, "\mathrm{adj}", "\simeq"'] \arrow[dr, "(1)", phantom] \& \& \mathrm{Map}_A(A, \underline{\mathrm{map}}_A(X, A \otimes X)) \arrow[ddll, "\_ \otimes X"] \arrow[dd, "\underline{\mathrm{map}}(u)_* \circ u"] \arrow[rr, "\mathrm{adj}", "\simeq"', no head] \arrow[ddrr, phantom, "(4)"] \& \& \mathrm{Map}(\mathbf{1}_\mathcal{C}, \underline{\mathrm{map}}_A(X, A \otimes X)) \arrow[dd, "\underline{\mathrm{map}}(u)_*"] \\
			\& \phantom{} \&\&\&\\
			\mathrm{Map}_A(A \otimes X, \underline{\mathrm{map}}_A(X, A \otimes X) \otimes X) \arrow[uu, "\mathrm{ev}"] \arrow[ddrr, "(\underline{\mathrm{map}}(u) \otimes X)_* \circ u"] \arrow[dd, phantom, "(2)"] \arrow[rr, "(3)", phantom] \& \& \mathrm{Map}(A, \underline{\mathrm{map}}(X, A \otimes X)) \arrow[dd, "\_ \otimes X"] \arrow[rr, "\eta^*"] \arrow[ddrr, phantom, "(5)"] \& \& \mathrm{Map}(\mathbf{1}, \underline{\mathrm{map}}(X, A \otimes X)) \arrow[dd, no head, "\text{adj}", "\simeq"'] \\
			\&\&\&\&\\
			\mathrm{Map}(A \otimes X, A \otimes X) \arrow[rrrr, bend right=10, "\eta^*"] \& \& \mathrm{Map}(A \otimes X, \underline{\mathrm{map}}(X, A \otimes X)) \arrow[ll, "\mathrm{ev}"] \& \& \mathrm{Map}(X, A \otimes X)
		\end{tikzcd}
		\end{adjustbox}
		\begin{itemize}
			\item (1) has a filler by definition of $\mathrm{ev}_A$ as the counit of the $\otimes$-$\underline{\mathrm{map}}_A$-adjunction.
			\item (2) has a filler, since by definition of $\underline{\mathrm{map}}(u)$ we have a commutative diagram
			\[\begin{tikzcd}[column sep=large]
				\underline{\mathrm{map}}_A(x, y) \otimes x \arrow[r, "\underline{\mathrm{map}}(u) \otimes \mathrm{id}_x"] \arrow[d, "\mathrm{ev}_A"] & \underline{\mathrm{map}}(x, y) \otimes x \arrow[d, "\mathrm{ev}"] \\
				y \arrow[r, phantom, "="] & y
			\end{tikzcd}\]
			\item (3) has a filler because $u$ commutes with $\_ \otimes A$ on $\mathcal{C}$ and $\mathrm{Mod}_A(\mathcal{C})$, respectively.
			\item (4) has a filler because we define $\eta$ to be the unit of the $\_ \otimes A$-$u$-adjunction or, equivalently, the unique map of bialgebras $\mathbf{1} \to A$.
			\item (5) has a filler: The left composite is the $X \otimes \_$-$\underline{\mathrm{map}}(X, \_)$-adjunction followed by $\eta^*$ whereas the other composite is $\eta^*$ followed by the same adjunction. Thus it commutes by naturality of the $\_ \otimes X$-$\underline{\mathrm{map}}(X, \_)$-adjunction.
		\end{itemize}
		Now observe that the top left to top right to bottom right composite is the composite in \eqref{H-2} until the second-to-last term in the case $Y = A \otimes A$. Obviously, the top-left to bottom-left to bottom-right composite sends $\mathrm{id}_{A \otimes A}$ to $\eta \otimes \mathrm{id}_A$, which by definition gets sent to the shear map under the last map in \eqref{H-2}.
		
		\textit{1 $\Rightarrow$ 2}: This follows from plugging in $X = A$.
		
		\textit{2 $\Rightarrow$ 1}: Assume only $s_{A, A}$ is invertible. Observe that $s_{A, A \otimes Y} \simeq s_{A, A} \otimes \mathrm{id}_Y$ for $Y \in \mathcal{C}$ is an equivalence as well. $s_{A, \cdot}: A \otimes u(\_) \to A \otimes \_$ is a natural transformation between colimit-preserving functors. Moreover, it is an equivalence on the essential image of $A \otimes \_$. Since the module category is generated by objects of the form $A \otimes \_$ under colimits, $s_{A, X}$ is an equivalence for all $X \in \mathrm{Mod}_A(\mathcal{C})$.
		
		\textit{2 $\Rightarrow$ 3}: The shear map of $A$, viewed as a bialgebra in $h\mathcal{C}$, is still invertible in the homotopy category. Now it is a classical theorem that this is a criterion for being a Hopf algebra \cite[Lemma 5.1]{BLV11} in a monoidal $1$-category, i.e., the existence of an antipode.
		
		\textit{3 $\Rightarrow$ 2}: The shear map being an equivalence can be detected in the homotopy category.
	\end{proof}
\end{lem}
\begin{cor}
	Let $f:\mathcal{C} \to \mathcal{D}$ be a morphism in $\mathrm{CAlg}(\mathrm{Pr}^L)$, and $A$ be a Hopf algebra in $\mathcal{C}$. Then so is $f(A)$.
	\begin{proof}
		$f$ induces symmetric monoidal map on homotopy categories. Such a map preserves classical Hopf algebras. By Lemma \ref{Cond-16} this proves that $f(A)$ is a Hopf algebra in $\mathcal{D}$.
	\end{proof}
\end{cor}
\begin{cor}\label{H-11}
	Endow $\mathcal{C} \in \mathrm{Pr}^L$ with the cartesian symmetric monoidal structure and let $A$ be a monoid object in $\mathcal{C}$, i.e., a cocommutative bialgebra object in $\mathcal{C}$. Then the monoidal structure on $\mathrm{Mod}_A(\mathcal{C})$ is the cartesian one.
	\begin{proof}
		Recall that, the universal property of the cartesian monoidal structure is that the forgetful functor $\mathrm{coCAlg}(\mathcal{C}) \to \mathcal{C}$ is an equivalence (see \cite[Cor.2.4.3.10]{Lur17} for the dual statement).
		
		The dual of \cite[Prop.2.2.10]{Ra26} states that, if we endow $\mathrm{Mod}_A(\mathcal{C})$ with the symmetric monoidal structure from Construction \ref{H-3}, we obtain an equivalence over $\mathcal{C}$:
		\begin{equation}
			\mathrm{Mod}_A(\mathrm{coCAlg}(\mathcal{C})) \simeq \mathrm{coCAlg}(\mathrm{Mod}_A(\mathcal{C}))
		\end{equation}
		Thus we arrive at
		\begin{equation*}
			\mathrm{coCAlg}(\mathrm{Mod}_A(\mathcal{C})) \simeq \mathrm{Mod}_A(\mathrm{coCAlg}(\mathcal{C})) \simeq \mathrm{Mod}_A(\mathcal{C})
		\end{equation*}
		For the second equivalence, we used that $\mathcal{C}$ has the cartesian monoidal structure.
	\end{proof}
\end{cor}
\begin{cor}\label{H-10}
	Let $\mathcal{C} \in \mathrm{Pr}^L$ and $A \in \mathcal{C}$ be a cocommutative Hopf algebra. Then $X \in \mathrm{Mod}_A(\mathcal{C})$ is dualizable if and only if $uX \in \mathcal{C}$ is with respect to the symmetric monoidal structure in \ref{H-3}.
	\begin{proof}[Proof sketch]
		Since $u: \mathrm{Mod}_A(\mathcal{C}) \to \mathcal{C}$ is symmetric monoidal, the "only if" part is immediate. We observe that the evaluation map $e: X \otimes \underline{\mathrm{map}}(X, \mathbf{1}) \to \mathbf{1}$ in $\mathrm{Mod}_A(\mathcal{C})$ gets sent to the evaluation map, up to the identification $u\mathrm{map}_A \simeq \mathrm{map}(u\_, u\_)$, by $u$. This provides a witness for $\tilde{X} := u\underline{\mathrm{map}}_A(X, \mathbf{1})$ being dual to $X$. Equivalently, by \cite[Lemma 4.6.1.6]{Lur17}, $e$ witnesses that $\tilde{X} \otimes \_$ is right adjoint to $X \otimes \_$. This implies that the bottom composite in the diagram of internal mapping objects
		\begin{equation}\label{H-13}
			\begin{tikzcd}
				u\underline{\mathrm{map}}_A(B, C \otimes \underline{\mathrm{map}}_A(X, \mathbf{1})) \arrow[r, "\_ \otimes X"] \arrow[d, no head, "\simeq"] & u\underline{\mathrm{map}}_A(B \otimes X, C \otimes \underline{\mathrm{map}}(X, \mathbf{1}) \otimes X) \arrow[r, "\mathrm{ev}_*"] \arrow[d, no head, "\simeq"] & \underline{\mathrm{map}}(B \otimes X, C)  \arrow[d, no head, "\simeq"] \\
				\underline{\mathrm{map}}(B, C \otimes \tilde{X}) \arrow[r, "\_ \otimes X"] & \underline{\mathrm{map}}(B \otimes X, C \otimes \tilde{X} \otimes X) \arrow[r, "e_*"] & \underline{\mathrm{map}}(B \otimes X, C)
			\end{tikzcd}
		\end{equation}
		is an equivalence. Here $\_ \otimes X$ is defined as the map corresponding under the internal hom adjunction to the natural transformation
		\begin{equation*}
			\mathrm{Map}_\mathcal{C}(\_ \otimes A, B \otimes \tilde{X}) \xrightarrow{\_ \otimes X} \mathrm{Map}_\mathcal{C}(\_ \otimes A \otimes X, B \otimes \tilde{X} \otimes X)
		\end{equation*}
		Applying $\mathrm{Map}_{\mathrm{Mod}_A(\mathcal{C})}(\mathbf{1}, \_)$ to the top composite in \eqref{H-13}, we obtain that
		\[\begin{tikzcd}
			\mathrm{Map}_{\mathrm{Mod}_A(\mathcal{C})}(B, C \otimes \tilde{X}) \arrow[r, "\_ \otimes X"] & \mathrm{Map}_{\mathrm{Mod}_A(\mathcal{C})}(B \otimes X, C \otimes \tilde{X} \otimes X) \arrow[r, "e_*"] & \mathrm{Map}_{\mathrm{Mod}_A(\mathcal{C})}(B \otimes X, C)
		\end{tikzcd}\]
		is an equivalence, witnessing that $X$ is dualizable as $A$-module, again by \cite[Lemma 4.6.1.6]{Lur17}.
	\end{proof} 
\end{cor}
\begin{rem}
	The proof actually gives us more: Given a tuple $(X, Y, \mu: X \otimes Y \to \mathbf{1})$. If $u\mu$ witnesses that $X, Y$ are dual, then $X, Y$ are dual witnessed by $\mu$.
\end{rem}
\begin{lem}\label{H-15}
	Diagram \eqref{Cond-26} in the proof of Lemma \ref{Cond-29} admits a homotopy.
\begin{proof}
	We write $\underline{\mathrm{map}}$ for the mapping spectrum and $\underline{\mathrm{map}}$ for the internal hom in the respective category. Denote by $\tilde{\phi}: \gamma_!\ZZ_l \to \ZZ_l$ the equivalence, which is adjoint to $\phi: \ZZ_l \to \gamma^*\ZZ_l$, the witness of unitality of $\gamma^*$. Denote by $\mathrm{ev}$ the counit of the inner hom adjunction. Denote by $\underline{\mathrm{map}}(\gamma^*): \gamma^*\underline{\mathrm{map}}(x, y) \to \underline{\mathrm{map}}(\gamma^*x, \gamma^*y)$ the map defined in Definition \ref{Cond-15}. Consider the following diagram:\newline
	\begin{adjustbox}{center}
		\begin{tikzcd}[column sep=small]
			\underline{\mathrm{map}}(x, y) \arrow[ddrr, phantom, "(1)"] \arrow[dd, "\gamma^*"] & & \underline{\mathrm{map}}(x, \underline{\mathrm{map}}(x, y) \otimes x) \arrow[ll, "\mathrm{ev}_*"] \arrow[dd, "\gamma^*"] \arrow[ddrr, phantom, "(2)"] & & \underline{\mathrm{map}}(\ZZ_l, \underline{\mathrm{map}}(x, y)) \arrow[ll, "\_ \otimes x"] \arrow[dd, "\gamma^*"] \arrow[rr, "\tilde{\phi}^*"] \arrow[dr, phantom, "(3)"] & & \underline{\mathrm{map}}(\gamma_!\ZZ_l, \underline{\mathrm{map}}(x, y)) \arrow[ddddll, bend left] \arrow[ddll, "\simeq", no head, "\mathrm{adj}"'] \arrow[llllll, bend right=10, "\simeq"]\\
			&&&&& \phantom{}\\
			\underline{\mathrm{map}}(\gamma^*x, \gamma^*y) & & \underline{\mathrm{map}}(\gamma^*x, \gamma^*\underline{\mathrm{map}}(x, y) \otimes \gamma^*x) \arrow[ll, "\gamma^*(\mathrm{ev})"] \arrow[dd, "\underline{\mathrm{map}}(\gamma^*)"] \arrow[ddrr, phantom, "(5)"] \arrow[dl, phantom, "(4)"] & &  \underline{\mathrm{map}}(\ZZ_l, \gamma^*\underline{\mathrm{map}}(x, y)) \arrow[ll, "\_ \otimes \gamma^*x"] \arrow[dd, "\underline{\mathrm{map}}(\gamma^*)_*"] \\
			& \phantom{} \\
			& & \underline{\mathrm{map}}(\gamma^*x, \underline{\mathrm{map}}(\gamma^*x, \gamma^*y) \otimes \gamma^*x) \arrow[uull, "\mathrm{ev}_*"] & & \underline{\mathrm{map}}(\ZZ_l, \underline{\mathrm{map}}(\gamma^*x, \gamma^*y)) \arrow[ll, "\_ \otimes \gamma^*x"] \arrow[uullll, "\simeq", bend left=45] &
		\end{tikzcd}
	\end{adjustbox}
	where we have left out the symmetric monoidality witnesses of $\gamma^*$
	\begin{enumerate}
		\item square (1) commutes by functoriality
		\item (2) has a filler by symmetric monoidality of $\gamma^*$
		\item For (3): Adding the omitted symmetric monoidality witnesses back in, we arrive at the diagram where the rightmost composite is the adjunction equivalence:
		\[\begin{tikzcd}
			\underline{\mathrm{map}}(\ZZ_l, \underline{\mathrm{map}}(x, y)) \arrow[d, "\gamma^*"]  \arrow[r, "\tilde{\phi}"] & \underline{\mathrm{map}}(\gamma_!\ZZ_l, \underline{\mathrm{map}}(x, y)) \arrow[d, "\gamma^*"] \\
			\underline{\mathrm{map}}(\gamma^*\ZZ_l, \gamma^*\underline{\mathrm{map}}(x, y)) \arrow[r, "\gamma^*(\tilde{\phi})"] \arrow[d, "\phi^*", "\simeq"'] & \underline{\mathrm{map}}(\gamma^*\gamma_!\ZZ_l, \gamma^*\underline{\mathrm{map}}(x, y)) \arrow[dl, "\eta^*"] \\
			\underline{\mathrm{map}}(\ZZ_l, \gamma^*\underline{\mathrm{map}}(x, y))				
		\end{tikzcd}\]
		The top square has a witness by functoriality; the bottom diagram has a witness by the description of adjunctions in terms of initial objects in the slice \cite[Lem.5.2.4.1]{Lur09}.
		\item (4) commutes by definition of the comparison map $\gamma^*$, since under the inner hom adjunction it corresponds to
		\begin{equation*}
			\gamma^*\underline{\mathrm{map}}(x, y) \otimes \gamma^*y \simeq \gamma^*(\underline{\mathrm{map}}(x, y) \otimes x) \xrightarrow{\gamma^*(\mathrm{ev})} \gamma^*y 
		\end{equation*}
		\item (5) has a witness, since post- and precomposition commute.
	\end{enumerate}
	Now the outer square is exactly \eqref{Cond-26}, which therefore admits a homotopy.
\end{proof}
\end{lem}
		\section{Some Sheafy Facts}\label{P}
		\subsection{Facts about the pro-étale site}
\begin{lem}\label{F-9}
	The canonical functor $\mathrm{Pro}(BG_\et) \xrightarrow{\simeq} \mathrm{Mod}_G(\mathrm{Pro}(\mathrm{Fin}))$ is an equivalence.
	\begin{proof}
		Denote by $\mathrm{ProFin}$ the full subcategory of the $1$-category of topological spaces $\mathrm{Top}$, which are a cofiltered limit of finite discrete sets. By \cite{Jo82}, these are the Stone topological spaces. $\mathrm{ProFin} \to \mathrm{Top} \to \mathrm{Set}$ is conservative, and the canonical map $\mathrm{Pro}(\mathrm{Fin}) \to \mathrm{ProFin}$ is an equivalence. Thus, to show the desired statement, we may replace $\mathrm{Mod}_G(\mathrm{Pro}(\mathrm{Fin}))$ by $\mathrm{Mod}_G(\mathrm{ProFin})$.
		\begin{claim}
			Every object in $\mathrm{Pro}(BG_\et)$ admits a description as a formal limit $\mathrm{lim}_{i \in \mathcal{I}} X_i$ with $X_i \in BG_\et$ along surjective (equivariant) transition maps.
			\begin{proof}
				Same proof as in the non-equivariant case: Let $X = \mathrm{lim}_{j \in \mathcal{J}} X_j$ be a presentation. For fixed $j$ the images of $X_{j'} \to X_j$ stabilize, giving rise to a discrete sub-$G$-set $Y_j \subset X_j$. By the same argument as in the non-equivariant case, the limit over $Y_j$ is the same as over $X_j$, by the standard limit argument.
			\end{proof}
		\end{claim}
		But now for $X = \mathrm{lim}_i X_i$ and $Y = \mathrm{lim}_j Y_j$ elements of $\mathrm{Pro}(BG_\et)$ with such a presentation the map
		\begin{equation*}
			\colim_{i \in I^\op}\mathrm{Hom}_{\mathrm{Mod}_G(*_\proet)}(X_i, Y_j) \to \mathrm{Hom}_{\mathrm{Mod}_G(*_\proet)}(X, Y_j)
		\end{equation*}
		is clearly injective as the left-hand side is filtered colimit along inclusions. It is also surjective, since it is true when dropping the equivariance condition and since the condition of a factorization $X \to X_i \to Y_j$ to be $G$-equivariant is just a condition on the second map by surjectivity of the first map, which is $G$-equivariant. Taking limits in $j$ yields the desired fully faithfulness.
		
		Bullet point \ref{F-9-2} of the following claim settles essential surjectivity:
		\begin{claim}
			Let $X \in \mathrm{Mod}_G(\mathrm{ProFin})$. Denote by $\mathcal{P}_X$ and $\mathcal{P}_X^G$ the posets of partitions and $G$-stable clopen partitions of $X$. Morphisms are given by $\mathcal{U} \to \mathcal{V}$ if $\mathcal{U}$ refines $\mathcal{V}$. There is a composition of functors
			\begin{equation*}
				\mathcal{P}_X^G \subset \mathcal{P}_X \xrightarrow{\vert \_ \vert} \mathrm{Set} \subset \mathrm{Top}
			\end{equation*}
			where the last functor sends a partition $\mathcal{U}$ to its underlying set of objects. Then:
			\begin{enumerate}
				\item\label{F-9-1} The set of objects of $\mathcal{P}_X$ is a discrete $G$-representation.
				\item\label{F-9-2} $X \xrightarrow{\simeq} \lim_{\mathcal{U} \in \mathcal{P}_X^G} \vert \mathcal{U} \vert$.
			\end{enumerate}
			\begin{proof}
				\textit{For 1}: For $n \in \NN$ we have a $G$-equivariant bijection of sets
				\begin{align*}
					\{\text{ordered clopen partitions of order }n\} &\simeq \mathrm{Map}_\mathrm{ProFin}(X, \{1, ..., n\}^\delta) \\
					\{U_i\}_{i = 1, ..., n} &\mapsto f \text{ with } f\vert_{U_i} := i \\
					\{f^{-1}(i)\}_{i = 1, ..., n} &\mapsfrom f
				\end{align*}
				where $G$ acts on the right via restriction.
				Now observe that we can endow the right-hand side with the compact-open topology. We claim that this action is continuous: We need to prove that the action map
				\begin{equation}\label{P-3}
					G \times \mathrm{Map}(X, \{1, ..., n\}^\delta) \to \mathrm{Map}(X, \{1, ..., n\}^\delta)
				\end{equation}
				is continuous. It is obtained by "adjoining over" $X$ in the continuous function
				\begin{equation*}
					G \times X \times \mathrm{Map}(X, \{1, ..., n\}^\delta) \xrightarrow{\mu \times \mathrm{id}} X \times \mathrm{Map}(X, \{1, ..., n\}^\delta) \xrightarrow{\mathrm{ev}} \{1, ..., n\}^\delta 
				\end{equation*}
				$\mathrm{ev}$ is continuous, since $X$ is locally compact by \cite[Thm.46.10]{Mun00}. Thus by \cite[Thm.46.11]{Mun00} \eqref{P-3} is continuous.
				
				But the compact-open topology on $\mathrm{Map}(X, \{1, ..., n\}^\delta)$ is discrete: Let $f: X \to \{1, ..., n\}^\delta$ be a continuous function. Then
				\begin{equation*}
					\{f\} = \bigcap\limits_{i = 1}^n\{g: X \to \{1, ..., n\}^\delta \vert g(f^{-1}(i)) \subset \{i\}\} \subseteq \mathrm{Map}(X, \{1, ..., n\}^\delta)
				\end{equation*}
				is a description of $f$ as finite intersection of generating opens of the compact-open topology. Here we used the compactness of $X$ to deduce that the $f^{-1}(i)$'s are compact. Now the quotient of a discrete topological space by a finite group is again discrete. But
				\begin{equation*}
					\mathcal{P}_X = \bigcup_{n \in \NN} \mathrm{Map}(X, \{1, ..., n\}^\delta)/\Sigma_n
				\end{equation*}
				$G$-equivariantly, since any clopen partition must be finite by compactness of $X$. Hence, $\mathcal{P}_X$ has open stabilizers.
				
				\textit{For 2}: Since limits in $\mathrm{Mod}_G(\mathrm{ProFin})$ are computed underlyingly it suffices to prove that the inclusion $\mathcal{P}_X^G \to \mathcal{P}_X$ is limit cofinal. By Joyal's version of Quillen's Theorem A \cite[\href{https://kerodon.net/tag/02NY}{Tag 02NY}]{Ker} it suffices to prove that the relative slice $(\mathcal{P}_X^G)_{/\mathcal{U}}$ is weakly contractible for all $\mathcal{U} \in \mathcal{P}_X$. We even provide a terminal object:
				The common refinement $\bigwedge\limits_{g \in G} g\mathcal{U}$ is a terminal object, if we can show that we get the same result for a common refinement over a finite indexing set:
				\begin{equation*}
					\bigwedge\limits_{g \in G} g\mathcal{U} = \bigwedge\limits_{\bar{g} \in G/\mathrm{Stab}_\mathcal{U}(G)} g\mathcal{U}
				\end{equation*}
				and $G/\mathrm{Stab}_\mathcal{U}(G)$ is finite by bullet point \ref{F-9-1}.
			\end{proof}
		\end{claim}
	\end{proof}
\end{lem}
\begin{lem}\label{F-21}
	Let $Y \to X$ be a $\Gamma$-Galois covering of schemes, with our standard assumption in Section \ref{C1} on $X,Y$, and $\Gamma$ a profinite group. There is a short exact sequence of profinite groups:
	\begin{equation}\label{F-20}
		1 \to \pi_1^\et(Y, \bar{\eta}) \to \pi_1^\et(X, \bar{\eta}) \to \Gamma \to 1
	\end{equation}
	Let $G$ be a finite group. Then there is an equivalence
	\begin{equation}\label{F-23}
		\mathrm{Map}_{\mathrm{Shv}(X_\et)}(Y ,BG) \simeq \mathrm{Map}_{\mathrm{Cond}\Spc}(\Gamma/\pi_1^\et(X, \bar{\eta}), BG) \simeq \mathrm{Map}_{\mathrm{Cond}\Spc}(B\pi_1^\et(Y, \bar{\eta}), BG)
	\end{equation}
	where the first equivalence is $\Gamma$-equivariant.
	
	Taking $\pi_0$ yields a $\Gamma$-equivariant equivalence
	\begin{equation}\label{F-24}
		\mathrm{H}^1(Y, BG) \simeq \mathrm{Hom}_{\mathrm{Grp}(\mathrm{Pro}(\mathrm{Fin}))}(\pi_1^\et(X, \bar{\eta}), G)/G
	\end{equation}
	where $G$ acts by conjugation in the source. Moreover, this identification is $\Gamma$-equivariant with the following action: Let $g \in \Gamma$, and let $f: \pi_1^\et(X, \bar{\eta}) \to G$ be a representative modulo conjugation by $G$. Pick a lift $\tilde{g} \in \pi_1^\et(X, \bar{\eta})$ and precompose $f$ with the restriction of the conjugation action by $\tilde{g}$ to the normal subgroup $\pi_1^\et(Y, \bar{\eta})$.
	\begin{proof}
		\textit{For the short exact sequence}: Since $Y/X$ is unramified they have a common universal pro-covering $\tilde{X}$ corresponding to $\pi_1^\et(X, \bar{\eta})$ under $\mathrm{Mod}_{\pi_1^\et(X, \bar{\eta})}(*_\proet) \subset X_\proet$ (see \eqref{Et-2} for the definition of this inclusion). The short exact sequence comes from the short exact sequence of Galois groups of function field extensions $K(\tilde{X})/K(Y)/K(X)$ proven in \cite[\href{https://stacks.math.columbia.edu/tag/0BMM}{Tag 0BMM}]{stacks-project}.
		
		\textit{For formula \eqref{F-23}}: Recall that $\mathrm{Pro}(X_\text{fet}) \subset X_\proet$ induces an adjunction:
		\[\begin{tikzcd}
			\mathrm{Shv}(X_\et) \arrow[r, shift left, "r_*"] & \mathrm{Mod}_{\pi_1^\et(X, \bar{\eta})}(\mathrm{Cond}\Spc) \arrow[l, shift left, "r^*"]
		\end{tikzcd}\]
		Recall the following facts about this adjunction
		\begin{enumerate}
			\item The representable $r^*\Gamma \simeq Y \in \mathrm{Shv}(X_\proet)$ by definition.
			\item The unit $BG \to r_*BG \simeq r_*r^*BG$ is an equivalence: We may check this against representables $Z \in \mathrm{Pro}(X_\text{fet})$. Then the statement is equivalent to
			\begin{equation*}
				\{\text{profinite étale locally trivial } G \text{-torsors}/Z\} \subset \{\text{proétale locally trivial }G\text{-torsors}/Z\}
			\end{equation*}
			being an equivalence. But a $G$-torsor $p:E \to Z$, which trivializes fpqc-locally on $Z$, is already finite étale by \cite[\href{https://stacks.math.columbia.edu/tag/02LA}{Tag 02LA}]{stacks-project} and \cite[\href{https://stacks.math.columbia.edu/tag/02VN}{Tag 02VN}]{stacks-project}. Since $p$ trivializes along itself and $p$ is a finite étale covering, this proves the converse inclusion.
		\end{enumerate}
		Thus we get a $\Gamma$-equivariant chain of equivalences:
		\begin{align*}
			\mathrm{Map}_{\mathrm{Shv}(X_\et)}(Y, BG) &\simeq \mathrm{Map}_{\mathrm{Shv}(X_\et)}(r^*\Gamma, BG) \simeq \mathrm{Map}_{\mathrm{Mod}_{\pi_1(X, \bar{\eta})}(\mathrm{Cond}\Spc)}(\Gamma, r_*BG) \\
			&\simeq \mathrm{Map}_{\mathrm{Mod}_{\pi_1(X, \bar{\eta})}(\mathrm{Cond}\Spc)}(\Gamma, BG) \simeq \mathrm{Map}_{\mathrm{Cond}\Spc}(\pi_1^\proet(X, \bar{\eta})\backslash \Gamma, BG)
		\end{align*}
		The last equivalence in \eqref{F-23} follows from $\pi_1^\et(X, \bar{\eta}) \backslash \Gamma \simeq B\pi_1^\et(Y, \bar{\eta})$ witnessed by \eqref{F-20}.
		
		\textit{For formula \eqref{F-24}}: Thus, it remains to prove the following statement: Let
		\begin{equation*}
			1 \to A \to B \to C \to 1
		\end{equation*}
		be a short exact sequence of profinite groups. Then we have a $C$-equivariant equivalence
		\begin{equation}\label{P-15}
			\pi_0\mathrm{Map}_{\mathrm{Cond}\Spc}(B \backslash C, BG) \simeq \mathrm{Hom}_{\mathrm{Grp}(*_\proet)}(A, G)/G
		\end{equation}
		The action on the left is given by restricting along right multiplication on $B \backslash C$. On the right, $C$ acts on $A$ by restricting the conjugation action on $B$ of any lift $\hat{c}$ for $c \in C$.
		\begin{claim}
			The above holds for $A, B, C$ discrete groups, and $\mathrm{Hom}_{\mathrm{Grp}(\mathrm{Set})}(A, G)/G$ on the right-hand side in \eqref{P-15}.
			\begin{proof}[Proof sketch]
				Recall the following facts:
				\begin{enumerate}
					\item $B \backslash C$ and $BG$ lie in the essential image of $\mathrm{An} \subset \mathrm{Cond}\Spc$. So we may work in $\mathrm{An}$.
					\item\label{F-21-1} $1$-truncated anima are modelled by the $(2, 1)$-category of groupoids. In particular,
					\begin{equation*}
						h(\mathrm{An}_{\leq 1}) \simeq h(\mathrm{Gpd})
					\end{equation*}
					, i.e., the homotopy category of $1$-truncated anima and the homotopy category of groupoids coincides.
					\item The homotopy quotient $B \backslash C$ is modelled by the action groupoid $B \backslash \backslash C$. Its objects are the elements of $C$ and the mapping sets are $\mathrm{Map}(c, c') = \{b \in B \vert bc = c'\}$ \footnote{This follows by combining the facts:
						\begin{enumerate}
							\item The classical Grothendieck construction extends Luries straightening-unstraightening equivalence \cite[Introduction to 3.2]{Lur09}.
							\item The total space of the Grothendieck construction is exactly $B \backslash \backslash C$.
							\item The total space of the unstraightening of a functor with source a anima computes its colimit \cite[Cor.3.3.4.3]{Lur09}.
						\end{enumerate}
					}. The residual $C$-action on the quotient is given by right multiplication on objects. For $c \in C$, denote by $r_c: B \backslash \backslash C \to B \backslash \backslash C$ this functor.
					\item[$\Rightarrow$] $BA$ is modelled by the groupoid $A \backslash \backslash *$. By bullet point \ref{F-21-1}, homotopy classes $[BA, BG]$ are thus given by functors of groupoids $f: A\backslash\backslash * \to G \backslash \backslash *$ modulo homotopy, or, equivalently, group homomorphisms $A \to G$ modulo conjugation by $g \in G$ on the target. This gives rise to the desired equivalence of underlying sets.
					\item\label{F-21-2} Every $c \in C$ gives rise to a homotopy equivalence $c: A \backslash \backslash * \xrightarrow{\simeq} B \backslash \backslash C$ given by the functor mapping the point to $c$ and which induces the equivalence $A \simeq \{b \in B \vert bc = c\} \subset B$ on mapping sets. Every lift $\hat{c}$ along $B \to C$ of $c$ gives rise to a homotopy $\hat{c}: 1 \Rightarrow c: A \backslash \backslash * \xrightarrow{\simeq} B \backslash \backslash C$.
				\end{enumerate}
				Denote by $\mathrm{conj}_{\hat{c}}: A \to A$ the inner automorphism induced by conjugation by $\hat{c}$ on $A \subset B$. Now we have a strictly commuting square
				\[\begin{tikzcd}
					A \backslash \backslash * \arrow[d, "\mathrm{conj}_{\hat{c}}\backslash \backslash *"] \arrow[r, "1"] & B \backslash \backslash C \arrow[d, "l_c"] \\
					A \backslash \backslash * \arrow[r, "c"] & B \backslash \backslash C
				\end{tikzcd}\]
				of groupoids. The bottom map is homotopic to the top vertical map as described in \ref{F-21-2}. Hence, the action on homotopy classes of groupoids $[BA, BG] \simeq \mathrm{Hom}_{\mathrm{Grp}}(A, G)/G$ is via restriction along conjugation by $\hat{c}$.
			\end{proof}
		\end{claim}
		Now we reduce to the general case as follows: $A, B, C$ are representable profinite groups in $\mathrm{Cond}\Spc$. They admit a map from $A^\delta, B^\delta, C^\delta$, the constant sheaves associated to the underlying discrete groups, via the counit of the $\underline{\cdot} \dashv \_(*)$-adjunction. This comparison gives rise to a $C$-equivariant map of fiber sequences
		\[\begin{tikzcd}
			\mathrm{Map}_{\mathrm{Cond}\Spc_*}(B^\delta \backslash C^\delta, BG) \arrow[r] & \mathrm{Map}_{\mathrm{Cond}\Spc}(B^\delta \backslash C^\delta, BG) \arrow[r] & BG \arrow[d, sloped, phantom, "="] \\
			\mathrm{Map}_{\mathrm{Cond}\Spc_*}(B \backslash C, BG) \arrow[u] \arrow[r] & \mathrm{Map}_{\mathrm{Cond}\Spc}(B \backslash C, BG) \arrow[u, "\epsilon^*"] \arrow[r] & BG
		\end{tikzcd}\]
		Thus the long exact sequence on homotopy groups yields a square
		\[\begin{tikzcd}
			\mathrm{Hom}_{\mathrm{Grp}(\mathrm{Set})}(A, G) \arrow[r, no head, "\simeq"] & \pi_0\mathrm{Map}_{\mathrm{Cond}\Spc_*}(B^\delta \backslash C^\delta, BG) \arrow[r] & \pi_0\mathrm{Map}_{\mathrm{Cond}\Spc}(B^\delta \backslash C^\delta, BG) \\
			\mathrm{Hom}_{\mathrm{Grp}(*_\proet)}(A, G) \arrow[r, no head, "\simeq"] \arrow[u, hook, sloped] & \pi_0\mathrm{Map}_{\mathrm{Cond}\Spc_*}(B \backslash C, BG) \arrow[r] \arrow[u] & \pi_0\mathrm{Map}_{\mathrm{Cond}\Spc}(B \backslash C, BG) \arrow[u]
		\end{tikzcd}\]
		where:
		\begin{enumerate}
			\item The horizontal equivalences on the left come from the recognition principle for group objects in any $\infty$-topos $\mathcal{X}$: Pointed connected objects are equivalent to group objects via
			\[\begin{tikzcd}
				\mathcal{X}_{*, \geq 1} \arrow[r, shift left, "\Omega"] & \mathrm{Grp}(\mathcal{X}) \arrow[l, shift left, "B"]
			\end{tikzcd}\]
			\footnote{By characterization in \cite[Sect.6.1]{Lur09} a topos the \v{C}ech nerve construction induces an equivalence $\mathrm{Eff}(\mathcal{X}) \xrightarrow{\simeq} \mathrm{Gpd}(\mathcal{X})$ between the full subcategory of $\mathrm{Fun}(\Delta^1, \mathcal{X})$ on effective epimorphisms and the full subcategory of $\mathrm{Fun}(\Delta^\op, \mathcal{X})$ of groupoids. Taking vertical fibers over the point in the commutative diagram
				\[\begin{tikzcd}[ampersand replacement=\&]
					\mathrm{Gpd}(\mathcal{X}) \arrow[d, "\vert_{[0]}"] \arrow[r, "\simeq"] \& \mathrm{Eff}(\mathcal{X}) \arrow[d, "\text{source}"] \\
					\mathcal{X} \arrow[r, phantom, "="] \& \mathcal{X}
				\end{tikzcd}\]
				yields the desired equivalence.} Moreover, $A^\delta, B^\delta, C^\delta$ live in the essential image of $\mathrm{Set} \subset \mathrm{An} \subset \mathrm{Cond}\Spc$.
			\item The leftmost vertical map is injective, since $\mathrm{*_\proet} \to \mathrm{Set}$ is faithful.
			\item The right vertical morphisms are surjective. In both cases they witness that the target is the quotient by the conjugation action by $G$.
		\end{enumerate}
		So we conclude that the right vertical map is injective. Since $\pi_0\epsilon^*$ is $C$-equivariant this identifies the action as the correct one.
	\end{proof}
\end{lem}
\subsection{Canonicity of the finite étale trace}
The approach in this subsection is very much inspired by the much more complicated situation in \cite{EHKSY21}.
\begin{fact}
	Let $\mathcal{C}$ be a category. Let $E, F \subset \mathcal{C}$ be wide subcategories, closed under composition, containing all equivalences, and such that for $e: z \to w$ and $f: y \to w$ the pullback
	\[\begin{tikzcd}
		x \arrow[r, "\bar{e}"] \arrow[d, "\bar{f}"] & y \arrow[d, "f"] \\
		z \arrow[r, "e"] & w
	\end{tikzcd}\]
	in $\mathcal{C}$ exists and $\bar{e} \in E$ and $\bar{f} \in F$. Under these assumptions, denote by $\mathrm{Span}(\mathcal{C}, E, F)$ the $(2, 1)$-category of spans as for example constructed in \cite{LLP26} (and usually attributed to Barwick). Objects are the same as those of $\mathcal{C}$, and morphisms are given by spans $x \xleftarrow{e} y \xrightarrow{f} z$. The conditions on $E, F$ ensure that the composition with another morphism $z \xleftarrow{e'} w \xrightarrow{f'} v$ defined as the outer span in
	\[\begin{tikzcd}[column sep=small, row sep=small]
		& & y \times_z w \arrow[dr] \arrow[dl] & & \\
		& y \arrow[dr] \arrow[dl] & & w \arrow[dr] \arrow[dl] & \\
		x & & z & & v
	\end{tikzcd}\]
	is well-defined. In the diagram above the square is a pullback.
	
	Moreover, there are left exact functors $E^\op \to \mathrm{Span}(\mathcal{C}, E, F) \xleftarrow{\gamma} F$ by inserting only in the backward or forward direction.
\end{fact}
\begin{defin}
	Let $\mathcal{D} \in \mathrm{CAlg}(\mathrm{Pr}^L)$. For $X$ a qcqs scheme, denote 
	\begin{equation*}
		\mathrm{PSh}^\mathrm{\fet}(X_\proet; \mathcal{D}) := \mathrm{PSh}_\Sigma(\mathrm{Span}(X_\proet, X_\proet, \fet); \mathcal{D})
	\end{equation*}
	where the subscript $\Sigma$ denotes product-preserving presheaves.
	
	Left Kan extension along the opposite of $\gamma: X_\proet \to \mathrm{Span}(X_\proet, X_\proet, \mathrm{finite } \et)$ yields an adjunction:
	\begin{equation*}
		\gamma^*: \mathrm{PSh}(X_\proet; \mathcal{D}) \leftrightharpoons \mathrm{PSh}^{\fet}(X_\proet; \mathcal{D}): \gamma_* 
	\end{equation*}
\end{defin}
\begin{defin}
	We call $F \in \mathrm{PSh}^{\fet}(X_\proet; \mathcal{D})$ $\proet$-local if $\gamma_*F$ is a hypersheaf. Denote by $\mathrm{Shv}^\mathrm{\fet}(X_\proet; \mathcal{D})$ the full subcategory on local objects. By general principle, we obtain an adjunction:
	\begin{equation*}
		\gamma^*: \mathrm{Shv}(X_\proet; \mathcal{D}) \leftrightharpoons \mathrm{Shv}^{\fet}(X_\proet; \mathcal{D}): \gamma_* 
	\end{equation*}
\end{defin}
\begin{lem}\label{P-13}
	$\gamma^* \dashv \gamma_*$ is an equivalence for $\mathcal{D} = \mathrm{CMon}(\mathrm{An})$. Thus, this is also true for arbitrary $\mathcal{D}$, which are module over $\mathrm{CMon}(\mathrm{An})$ in $\mathrm{Pr}^L$.
	\begin{proof}
		We claim that $\mathrm{Span}(X_\proet, X_\proet, \fet)$ is already semi-additive. Therefore, the product-preserving functors on its opposite category are semi-additive as well: $\emptyset$ is clearly initial and terminal, as for $Y \in X_\proet$ the span $Y \leftarrow \emptyset \to \emptyset$ is the only span mapping to $\emptyset$. So it remains to check that for $Y, Z \in X_\proet$ the spans $Y \sqcup Z \leftarrow Y \to Y$, $Y \sqcup Z \leftarrow Z \to Z$ witness that $Y \sqcup Z$ is the product in the span category, which is an easy check.
		
		Now the core observation is the following: Let $Y \in X_\proet$ and let $\mathcal{O}_x$ be the spectrum of a strictly henselian local ring. Then
		\begin{equation}\label{P-10}
			\mathrm{Map}_{X_\proet}(\mathcal{O}_{x}, Y) \xrightarrow{\gamma^*}\mathrm{Map}_{\mathrm{Span}(X_\proet, X_\proet, \fet)}(\mathcal{O}_x, Y) 
		\end{equation}
		witnesses that the target is the free commutative monoid in $\mathrm{An}$ of its source: Recall that $(\mathcal{O}_x)_{\fet}$ is equivalent to the category of finite sets (see \cite[\href{https://stacks.math.columbia.edu/tag/04GK}{Tag 04GK}]{stacks-project}). One easily checks that
		\[\begin{tikzcd}[row sep=tiny]
			& \mathcal{O}_x^{\sqcup S} \arrow[ddl, "\nabla"] \arrow[ddr, "\coprod_S f_s"] & & & \\
			& & & \mapsto & \sum\limits_S f_s \in \coprod_{S \in \mathrm{Fin}} \mathrm{Map}_{X_\proet}(\mathcal{O}_x, Y)^{\times S}\\
			\mathcal{O}_x & & Y
		\end{tikzcd}\]
		provides an inverse for the map induced by \eqref{P-10}
		\footnote{Here $(\_)^{\times S}$ stands for the groupoid with vertices the $\vert S \vert$-fold product with paths between objects that are related by a permutation. This is exactly $((\_)^{\times \vert S \vert})_{h\Sigma_{\vert S \vert}}$.}.
		
		To finish the proof, we need bullet points \ref{P-13-1} and \ref{P-13-3} in the following:
		\begin{enumerate}
			\item\label{P-13-1} Both versions of $\gamma_*$ are conservative.
			\item\label{P-13-2} $\gamma_*\gamma^*: \mathrm{PSh}(X_\proet; \mathrm{An}) \to \mathrm{PSh}(X_\proet; \mathrm{An})$ sends covering sieves generated by a single morphism to $\proet$-equivalences.
			\item\label{P-13-3} Both versions of $\gamma_*$ preserve all sifted colimits.
		\end{enumerate}
		\begin{proof}
			\textit{For 1}: $X_\proet \to \mathrm{Span}(X_\proet, X_\proet, \fet)$ is essentially surjective, hence, restriction along its opposite induces a conservative functor on functor categories.
			
			\textit{For 2}: This follows from the computation of mapping spaces out of henselian local rings \eqref{P-10}.
			
			\textit{For 3}: $\mathrm{PSh}_\Sigma \subset \mathrm{PSh}$ preserves sifted colimits and restriction preserves all colimits. By bullet point \ref{P-13-2}, we have a commutative diagram
			\[\begin{tikzcd}
				\mathrm{PSh}^{\fet}(X_\proet; \mathrm{An}); \mathrm{An}) \arrow[r] \arrow[d, "L_\proet"] & \mathrm{CMon}(\mathrm{PSh}(X_\proet; \mathrm{An})) \arrow[d, "L_\proet"] \\
				\mathrm{Shv}^{\fet}(X_\proet); \mathrm{An}) \arrow[r] & \mathrm{CMon}(\mathrm{Shv}(X_\proet; \mathrm{An}))
			\end{tikzcd}\]
			The top map preserves products and sifted colimits, and therefore all colimits. Since the vertical functors are localizations, the bottom map is colimit-preserving as well.
		\end{proof}
		By the triangle identities and conservativity of $\gamma_*$, it suffices to prove that the unit $\mathrm{id} \to \gamma_*\gamma^*$ is an equivalence. By bullet point \ref{P-13-3}, $\gamma_*\gamma^*$ becomes colimit-preserving when one replaces $\mathrm{Shv}(X_\proet; \mathrm{An})$ by $\mathrm{CMon}(\mathrm{Shv}(X_\proet; \mathrm{An}))$. So it suffices to check that the unit is an equivalence on representables, which is true, again, by \eqref{P-10}.
	\end{proof}
\end{lem}
\begin{obs}\label{P-5}
	Let $f: Y \to X$ be an étale map. Then there is an adjunction
	\begin{equation}\label{P-4}
	\begin{tikzcd}
		\mathrm{Span}(X_\proet, X_\proet, \fet) \arrow[r, "f^*", shift left=2] \arrow[r, phantom,  "{\rotatebox{-90}{$\dashv$}}"] & \mathrm{Span}(Y_\proet, Y_\proet, \fet) \arrow[l, shift left=2, "f_\#"]
	\end{tikzcd}
	\end{equation}
	extending the corresponding adjunction between $X_\proet$ and $Y_\proet$.
	In case $f$ is finite, the following natural transformations witness $f_\# \dashv f^*$: Unit and counit are given by
	\[\begin{tikzcd}
		& W \times_X Y \arrow[dl, "\mathrm{pr}_W"] \arrow[dr, "="] & & & & Z \arrow[dl, "{(\mathrm{id}, r)}"] \arrow[dr, "="] & \\
		W & & W \times_X Y & & Z \times_X Y & & Z
	\end{tikzcd}\]
	for $Z \xrightarrow{r} Y \in Y_\proet$ and $W \in X_\proet$. Here we used, that by étaleness of $f$ the diagonal $Z \xrightarrow{(\mathrm{id}, r)} Z \times_X Y$ is clopen, hence finite étale.
	\begin{proof}
		Keeping notation as in the description, the triangle identities amount to commutativity of the following diagrams of spans:
		\[\begin{tikzcd}
			& & Z \arrow[dl, "{(\mathrm{id}, r)}"] \arrow[ddll, "=", bend right] \arrow[ddrr, "=", bend left] \arrow[dr, "="] & & \\
			& Z \times_X Y \arrow[dl, "Z \times_X f"] \arrow[dr, "="] & & Z \arrow[dl, "{(\mathrm{id}, r)}"] \arrow[dr, "="] & \\
			W \times_X Y & & W \times_X Y \times_X Y & & W \times_X Y
		\end{tikzcd}\]
		and
		\[\begin{tikzcd}
			& & W \times_X Y \arrow[dl, "W \times \Delta_Y"] \arrow[ddll, "=", bend right] \arrow[ddrr, "=", bend left] \arrow[dr, "="] & & \\
			& W \times_X Y \times_X Y \arrow[dl, "W \times_X Y \times f"] \arrow[dr, "="] & & W \times_X Y \arrow[dl, "W \times \Delta_Y"] \arrow[dr, "="] & \\
			W \times_X Y & & W \times_X Y \times_X Y & & W \times_X Y
		\end{tikzcd}\]
	\end{proof}
\end{obs}
\begin{obs}\label{P-14}
	Note that $Y$ lives in the image of $\mathrm{Span}(X_{\fet}, X_{\fet}, X_{\fet}) \subset \mathrm{Span}(X_\proet, X_\proet, \fet)$, and every object is dualizable in the first category witnessed by
	\[\begin{tikzcd}
		& Y \arrow[dl] \arrow[dr] & & & & Y \arrow[dl] \arrow[dr] & \\
		X & & Y \times_X Y & & Y \times_X Y & & X
	\end{tikzcd}\]
	Observe that the latter is exactly the dualizing datum used in the construction of \eqref{Et-12} upon applying $\ZZ[\_]$: It remains to identify the image of the span $Y \times_X Y \xleftarrow{\Delta} Y \to Y$. More generally, we identify the image of a clopen inclusion in the forward direction $U \sqcup V \leftarrow U \to U$ for any $U, V \in X_\proet$: Since the latter span actually witnesses semi-additivity of $\mathrm{PSh}_\Sigma(X_\proet; \mathrm{An})$ for $U, V$, we know that it must be sent to the composite
	\begin{equation*}
		\ZZ[U \sqcup V] \xleftarrow{\simeq} \ZZ[U] \oplus \ZZ[V] \xrightarrow{\mathrm{pr}_U} \ZZ[U]
	\end{equation*}
	This is exactly the map we used in \eqref{Et-12} to define the dualizing datum.
\end{obs}
\begin{cor}\label{P-6}
	The trace transformation defined in \eqref{Et-12} witnesses $f_* \dashv f^*$.
	\begin{proof}[Proof sketch]
		The adjunction $f^* \dashv f_*$ comes from restricting along the adjunction $f^* \dashv f_\#$ of span-categories in \eqref{P-4}. Similarly, the adjunction $f_\# \dashv f^*$ on span-categories specified in Observation \ref{P-5} induces an adjunction $f_* \dashv f^*$ upon restriction, witnessed by some counit $\bar{\epsilon}$. It remains to prove the following two statements:
		\begin{enumerate}
			\item The adjunction $f_* \dashv f^*$ described above is compatible with the projection formula, by which we mean that
			\begin{equation}\label{P-12}
			\begin{tikzcd}
				\_ \otimes f_*f^*\mathbf{1} \arrow[r, "\mathrm{id} \otimes \bar{\epsilon}_\mathbf{1}"] \arrow[d, "\mathrm{proj}"] & \mathrm{id} \arrow[d, no head, "="] \\
				f_*f^* \arrow[r, "\bar{\epsilon}"] & \mathrm{id} 
			\end{tikzcd}
			\end{equation}
			admits a homotopy. As a reminder, the projection formula is witnessed by
			\[\begin{tikzcd}
				\_ \otimes f_*\mathbf{1} \xrightarrow{\tilde{\eta}_f} f_*f^*(\_ \otimes f_*\mathbf{1}) \simeq f_*(f^*(\_) \otimes f^*f_*\mathbf{1}) \xrightarrow{f_*(\mathrm{id} \otimes \tilde{\epsilon}_f)} f_*f^*
			\end{tikzcd}\]
			\item On the unit it coincides with \eqref{Et-12}.
		\end{enumerate}
		\textit{For 1}: Under the equivalence
		\begin{equation}\label{P-11}
			\underline{\mathrm{map}}(f_\#f^*\mathbf{1}, \_) \simeq f_*f^*(\_)
		\end{equation}
		$\bar{\epsilon}$ corresponds to restriction along the unit $\mathrm{id} \to f_\#f^*$ on span categories. Upon applying \eqref{P-11} to the diagram which we want to show commutes, it corresponds to
		\[\begin{tikzcd}
			\_ \otimes \underline{\mathrm{map}}(f_\#f^*\mathbf{1}, \_) \arrow[d] \arrow[r, "\mathrm{id} \otimes \bar{\eta}^*"] & \mathrm{id} \arrow[d, no head, "="] \\
			\underline{\mathrm{map}}(f_\#f^*\mathbf{1}, \_) \arrow[r, "\bar{\eta}^*"] & \mathrm{id}
		\end{tikzcd}\]
		Here the left vertical map is adjoint to $\mathrm{id} \otimes \mathrm{ev}$. But this clearly commutes. Moreover, the left vertical map is an equivalence, since $f_\#f^*\mathbf{1}$ is dualizable.
		
		\textit{For 2}: The two composites we want to identify are
		\begin{align*}
			f_*f^*\mathbf{1} &\simeq \underline{\mathrm{map}}(f_\#f^*\mathbf{1}, \mathbf{1}) \xleftarrow{\mathrm{comp}_Y} f_\#f^*\mathbf{1} \xrightarrow{f} \mathbf{1} \\
			f_*f^*\mathbf{1} &\simeq \underline{\mathrm{map}}(f_\#f^*\mathbf{1}, \mathbf{1}) \xrightarrow{\bar{\eta}^*} \mathbf{1}
		\end{align*}
		Now observe that the above maps are determined by the canonical (see Observation \ref{P-14}) dualizing datum for $\ZZ[Y]$ induced by the geometric evaluation map. Hence, they must coincide.
	\end{proof}
\end{cor}
\begin{rem}\label{T-5}
	Observe that by Corollary \ref{P-6}, there must be a canonical natural transformation $\delta_f: f_\# \to f_*$, which recovers $\mathrm{comp}_Y$ on the unit. This is usually called ambidexterity. The construction of $\epsilon_f: f_*f^* \to \mathrm{id}$  via $\mathrm{comp}_Y$ shows that it is the only $\mathrm{Shv}(X_\proet; \mathcal{D}(\ZZ))$-linear transformation of the strongly linear functors $f_*f^*$ recovering $\mathrm{comp}_Y^{-1} \circ f$ on the unit. Thus $\delta_f$ is uniquely determined by $\delta_{f, f^*}$ being $\mathrm{Shv}(X_\proet; \mathcal{D}(\ZZ))$-linear and recovering $\mathrm{comp}_Y$ on the unit \footnote{By linearity we mean as functors between homotopy categories.}.

	Given $Z \xrightarrow{g} Y \xrightarrow{f} X$ composable finite étale morphisms, we therefore get $\delta_{gf} \simeq \delta_{f, g_\#} \circ f_*\delta_g \simeq f_\#\delta_g \circ \delta_{f, g_*}$:
	
	As a corollary, we get a commutative diagram
	\[\begin{tikzcd}
		(fg)_\#(fg)^* \arrow[r, "\delta_{f, g_\#}"] \arrow[rr, bend left, "\delta_{fg}"] \arrow[d, "\dot{\epsilon}_{fg}"] & f_*g_\#(fg)^* \arrow[r, "f_*\epsilon_g"] \arrow[d, "f_*\dot{\epsilon}_{g, f^*}"] & (fg)_*(fg)^* \arrow[d, "f_*\epsilon_{g, f^*}"] \\
		f_\# \arrow[r, "\mathrm{comp}_Y"] \arrow[rr, bend right, "\mathrm{comp}_Y"] & f_* \arrow[r, no head, "="] & f_*
	\end{tikzcd}\]
	where $\dot{\epsilon}$ denotes the counit of the $f_\# \dashv f^*$-adjunction. The left-hand square commutes by naturality. The right square commutes, since, by the above uniqueness statement of the norm map, $\epsilon_f \simeq \dot{\epsilon}_f \circ \delta_{f, f^*}$.
\end{rem}
\begin{rem}\label{F-40}
	Since the pairing $\mathrm{gev}_Y: \ZZ[Y]^{\otimes 2} \to \ZZ$ is natural in $Y/X$ finite étale, for $g: Z \to Y$ a finite étale map we obtain a commutative diagram:
	\[\begin{tikzcd}
		\ZZ[Y] \arrow[r, "\mathrm{comp}_Y"] & \ZZ[Y]^\vee \arrow[d, "g^\vee"] \\
		\ZZ[Z] \arrow[u, "g"] \arrow[r, "\mathrm{comp}_Z"] & \ZZ[Z]^\vee
	\end{tikzcd}\]
\end{rem}

\subsection{Categories of sites and $\infty$-topoi}

Let $\mathcal{C}$ denote an arbitrary category with finite limits throughout this subsection. Recall the following definitions \cite[Def.6.2.2.1]{Lur09}:
\begin{defin}
	Let $\mathcal{D} \subset \mathcal{C}$ be a full subcategory. $\mathcal{D}$ is called a sieve if, for $d \in \mathcal{D}$ and $c \to d$ a morphism, $c \in \mathcal{D}$.
\end{defin}
\begin{defin}
	A Grothendieck topology on $\mathcal{C}$ is a set of so-called covering sieves $J_c$ of $\mathcal{C}_{/c}$ for all $c \in \mathcal{C}$, satisfying the following axioms:
	\begin{enumerate}
		\item (Identity): $\mathcal{C}_{/c} \in J_c$
		\item (Basechange): For $\mathcal{C}^{(0)}_{/c} \in J_c$ and $h: d \to c$ a morphism, the sieve generated by the essential image of $f^*: \mathcal{C}^{(0)}_{/c} \subset \mathcal{C}_{/c} \to \mathcal{C}_{/d}$, which we denote by $h^*\mathcal{C}_{/c}^{(0)}$, is a covering sieve.
		\item (Localness): For any sieves $\mathcal{C}_{/c}^{(0)}, \mathcal{C}^{(1)}_{/c} \subset \mathcal{C}_{/c}$ such that the first is a covering sieve and $h^*\mathcal{C}^{(1)}_{/c}$ is a covering sieve for every $h \in \mathcal{C}^{(0)}_{/c}$, $\mathcal{C}_{/c}^{(1)}$ is a covering sieve.
	\end{enumerate}
\end{defin}
\begin{defin}\label{F-18}
	Define $\mathrm{Cat}_\mathrm{lex} \subset \mathrm{Cat}$ to be the category of small categories with finite limits and left exact functors between them.
	
	A (continuous) functor between categories with a Grothendieck topology $f: (\mathcal{C}, J_{\mathcal{C}, \_}) \to (\mathcal{D}, J_{\mathcal{D}, \_})$ is a left exact functor $f: \mathcal{C} \to \mathcal{D}$ such that, for every covering sieve $\mathcal{C}_{/c}^{(0)} \in J_{\mathcal{C}, c}$, the sieve generated by the essential image $f(\mathcal{C}_{/c}^{(0)})$ is a covering sieve.
	
	Denote by $\mathrm{Cat}_\mathrm{lex}^\mathrm{Groth} \subset \mathrm{Cat}_\mathrm{lex}$ the category of small categories with finite limits and a Grothendieck topology, and continuous functors between them.
\end{defin}
\begin{ex}\label{F-19}
	There is also the notion of a Grothendieck topology in terms of covering families on a $1$-category $\mathcal{C}$, as defined in \cite[III.2]{MM92}. The covering family notion gives rise to a Grothendieck topology as described above as follows:
	To each covering family $\mathcal{U} = \{U_i \to X\}$ we associate the full subcategory $\mathcal{C}^\mathcal{U}_{/X} \subset \mathcal{C}_{/X}$ of morphisms that have a factorization through one of the $U_i$. This turns out to satisfy the axioms of a Grothendieck topology as defined in Definition \ref{F-18}.
	
	In particular, a functor $f: \mathcal{C} \to \mathcal{D}$ that preserves covering families induces a continuous morphism of sites.
\end{ex}
\begin{defin}\label{P-7}
	Every topos $\mathcal{X}$ admits the structure of a category with Grothendieck topology by \cite[Def.6.2.4.1+Prop.6.2.4.2]{Lur09} whose covering sieves $\mathcal{C}_{/c}^{(0)}$ are generated by a single effective epimorphism $c' \to c$ \cite[Def. after Cor.6.2.3.5]{Lur09}. By the characterization of effective epimorphisms via \v{C}ech nerves \cite[Cor.6.2.3.5]{Lur09} they are preserved by left adjoints in geometric morphisms. Thus this construction assembles into a functor $\mathrm{LTop} \to \mathrm{Cat}_\mathrm{lex}^\mathrm{Groth}$.
\end{defin}
\begin{lem}\label{F-17}
	The forgetful functor $\mathrm{LTop} \to \mathrm{Cat}_\mathrm{lex}^\mathrm{Groth}$ described in Definition \ref{P-7} has a left adjoint given pointwise by $\mathcal{C} \mapsto \mathrm{Shv}^\mathrm{un}(\mathcal{C})$.
	\begin{proof}
		By \cite[Prop.6.2.3.20]{Lur09} the full subcategory $\mathrm{Fun}_*(\mathrm{Shv}^\text{un}(\mathcal{C}), \mathcal{X}) \subset \mathrm{Fun}(\mathrm{Shv}^\text{un}(\mathcal{C}), \mathcal{X})$ on geometric morphisms identifies via restriction along Yoneda with the full subcategory on functors $f:\mathcal{C} \to \mathcal{X}$, which satisfy:
		\begin{enumerate}
			\item $f$ is left exact.
			\item For every collection $\{c_\alpha \to c\} \subset \mathcal{C}_{/c}$ which generate a covering sieve,
			\begin{equation*}
				\coprod_\alpha f(c_\alpha) \to f(c)
			\end{equation*}
			is an effective epimorphism.
		\end{enumerate}
		These are exactly the continuous functors $\mathcal{C} \to \mathcal{X}$ where the target has the canonical topology: Indeed, this boils down to proving that $h: x \to y$ in $\mathcal{X}$ is an effective epimorphism if it generates a covering sieve in the canonical topology generated by a single effective epimorphism $e: w \to y$; the converse is true by definition. But the sieve generated by a single morphism is exactly the full subcategory of morphisms admitting a factorization through that morphism. Thus, there has to exist a factorization $e: w \to z \xrightarrow{h} x$. By \cite[Cor.6.2.3.12]{Lur09}, effective epimorphisms satisfy right cancellation, so $h$ must be an effective epimorphism as well.
	\end{proof}
\end{lem}
We prove that the category $\mathrm{Cat}_\mathrm{lex}^\mathrm{Groth}$ has finite coproducts, which are computed underlyingly as products.
\begin{lem}\label{F-16}
	Let $(\mathcal{C}, J_{\mathcal{C}, \_}), (\mathcal{D}, J_{\mathcal{D}, \_}) \in \mathrm{Cat}_\mathrm{lex}^\mathrm{Groth}$.
	
	Then the smallest Grothendieck topology generated by the collections of sieves
	\begin{equation}\label{F-15}
		\mathcal{C}_{/c}^{(0)} \times \{d\} \cup \{c\} \times \mathcal{D}_{/d}^{(0)} \in \mathcal{C \times \mathcal{D}}_{/c \times d}
	\end{equation}
	for all $\mathcal{C}_{/c}^{(0)} \in J_{\mathcal{C}, c}$ and $\mathcal{D}_{/d}^{(0)} \in J_{\mathcal{D}, d}$ makes the functors $\mathcal{C} \xrightarrow{\mathrm{id} \times *} \mathcal{C} \times \mathcal{D} \xleftarrow{* \times \mathrm{id}} \mathcal{D}$ continuous morphisms of sites, which witness that $\mathcal{C} \times \mathcal{D}$ is their coproduct in $\mathrm{Cat}_\mathrm{lex}^\mathrm{Groth}$.
	\begin{proof}[Proof sketch]
		It remains to prove two things:
		\begin{enumerate}
			\item There exists a minimal topology containing the required generating set described in \eqref{F-15}.
			\item The coproduct in $\mathrm{Cat}_\mathrm{lex}$ is the product of categories.
		\end{enumerate}
		\textit{For 1}: We just need to show that Grothendieck topologies on a fixed category $\mathcal{F} \in \mathrm{Cat}_\mathrm{lex}$ are closed under intersection, and that there is a maximal "discrete" Grothendieck topology: The latter is given by taking all sieves in $\mathcal{F}_{/d}$ for all $d \in \mathcal{F}$. This clearly satisfies all conditions of a Grothendieck topology. Given Grothendieck topologies $\{(\mathcal{F}_i, J_{i, \_})\}_{i \in I}$ it is obvious that $(\mathcal{F}, \cap_{i \in I} J_{i, \_})$ is again a Grothendieck topology. Let us check the pullback condition: Let $\mathcal{F}_{/c}^{(0)} \in \cap_{i \in I} J_{i, c}$. Then by the pullback condition of all $J_{i, \_}$ the sieve $g^*\mathcal{F}_{/c}^{(0)}$ lies in $J_{i, d}$ for all morphisms $g: d \to c$ in $\mathcal{F}_{/d}$ and $i \in I$; so it also lies in the intersection.
		
		\textit{For 2}: For $\mathcal{C}, \mathcal{D} \in \mathrm{Cat}_\mathrm{lex}$ denote by $\mathrm{Fun}_\mathrm{lex}(\mathcal{C}, \mathcal{D}) \subset \mathrm{Fun}(\mathcal{C}, \mathcal{D})$ the full subcategory of left exact functors. We just observe that we have an explicit inverse of the functor $\mathrm{Fun}_\mathrm{lex}(\mathcal{C} \times \mathcal{D}, \mathcal{F}) \to \mathrm{Fun}_\mathrm{lex}(\mathcal{C}, \mathcal{F}) \times \mathrm{Fun}_\mathrm{lex}(\mathcal{D}, \mathcal{F})$:
		\begin{equation*}
			\mathrm{Fun}_\mathrm{lex}(\mathcal{C}, \mathcal{F}) \times \mathrm{Fun}_\mathrm{lex}(\mathcal{D}, \mathcal{F}) \xrightarrow{\times} \mathrm{Fun}_\mathrm{lex}(\mathcal{C} \times \mathcal{D}, \mathcal{F} \times \mathcal{F}) \xrightarrow{(\_ \times_{\mathcal{F}} \_)_*} \mathrm{Fun}_\mathrm{lex}(\mathcal{C} \times \mathcal{D}, \mathcal{F})
		\end{equation*}
		- taking the product of functors followed by postcomposition with the left exact functor $\mathcal{F} \times \mathcal{F} \xrightarrow{\times} \mathcal{F}$.
	\end{proof}
\end{lem}
\begin{fact}
	Hypercompletion of topoi assembles into the reflector onto the full subcategory $\mathrm{LTop}^\wedge \subset \mathrm{LTop}$ of hypercomplete $\infty$-topoi \cite[Cor. 6.5.2.13]{Lur09}.
\end{fact}
\begin{fact}
	The non-full subcategory $\mathrm{LTop} \subset \mathrm{Pr}^L$ is closed under the Lurie tensor product and therefore promotes to a symmetric monoidal category \cite[Ex.4.8.1.19]{Lur17}. Also by \cite[Ex.4.8.1.19]{Lur17}, this monoidal structure coincides with the cocartesian monoidal structure on $\mathrm{LTop}$. The cocartesian symmetric monoidal structures clearly make the left adjoint $\mathrm{Shv}^\text{un}(\_)$ symmetric monoidal. On the source the coproduct is just the product of categories underlyingly, as proved in Lemma \ref{F-16}. Since hypercompletion is a left adjoint and therefore symmetric monoidal with respect to the cocartesian structure $(\_)^\wedge: \mathrm{LTop} \to \mathrm{LTop}^\wedge \subset \mathrm{LTop}$ is lax symmetric monoidal.
\end{fact}
\begin{constr}
	$\mathrm{Mod}_G(*_\proet)=BG_\proet$ is a $*_\proet$-module in categories: On objects it is given by the following construction: Let $X \in BG_\proet = \mathrm{Mod}_G(*_\proet)$ and $Y \in *_\proet$. Then the cartesian product $X \times Y$ inherits a $G$-module via acting on the left-hand side factor.
	
	We claim this promotes to a module structure in $\mathrm{Cat}^\mathrm{Groth}_\mathrm{lex}$ with respect to the cartesian monoidal structure: We need to show that the multiplication map $BG_\proet \times *_\proet \to BG_\proet$ is continuous:
	\begin{enumerate}
		\item It preserves covering families/is continuous: But coverings in $BG_\proet$ are detected by the underlying functor $\mathrm{Mod}_G(*_\proet) \to *_\proet$ under which the multiplication map is just given by taking the product.
		\item It is left exact: Finite limits are created and detected by the underlying functor. As underlyingly, the multiplication is just given by the product functor it is clearly left exact, since limits commute with arbitrary products.
	\end{enumerate}
\end{constr}
\begin{constr}
	Applying the symmetric monoidal composite
	\begin{equation*}
		\mathrm{Cat}^\mathrm{Groth}_\mathrm{lex} \to \mathrm{LTop} \to \mathrm{LTop}^\wedge
	\end{equation*}
	to the $*_\proet$-module $BG_\proet$ we obtain that $\mathrm{Shv}(BG_\proet)$ is a $\mathrm{Shv}(*_\proet)$-module. Moreover, since $u: \mathrm{Mod}_G(*_\proet) \to *_\proet$ is clearly $*_\proet$-linear, this makes $u^*: \mathrm{Shv}(BG_\proet) \to \mathrm{Shv}(*_\proet)$ into a $\mathrm{Shv}(*_\proet)$-linear functor.
\end{constr}
\begin{cor}\label{Cond-27}
	$u^*$ is conservative and corepresented by $G$.
	\begin{proof}
		\textit{For conservativity}: By the adjunction $\mathrm{ind}_\mathbf{1}^G \dashv u$ we have a factorization of the geometric morphism $\prod_{X \in \Set}((G \times X)^\wedge)^*$
		\begin{equation*}
			\mathrm{PSh}(BG_\proet) \xrightarrow{\mathrm{lke}(u^\op) = \mathrm{res}(\mathrm{ind}_\mathbf{1}^G)} \mathrm{PSh}(*_\proet) \xrightarrow{\prod_{X \in \Set} (X^\wedge)^*} \prod_{X \in \Set} \mathrm{An}
		\end{equation*}
		This also passes to the hypercompletion of sheaf categories, by the universal property of hypercompletion. As already proven, the composite and the morphism on the right are conservative when passing to hypercomplete sheaf categories and therefore the leftmost morphism is as well. But upon passing to the hypercompletions of sheaf categories, the left-hand side functor becomes $u^*$, proving the claim.
		
		\textit{For corepresentability}: Denote by $L_\proet$ hypersheafification. Observe that on presheaf-categories $\underline{\mathrm{Map}}(G, \_) = \mathrm{lke}(u^\op)$. Now $\mathrm{lke}(u^\op)$ descends to hypercomplete sheaves, whereas the right-hand side preserves sheaves and hypercomplete objects. Thus there is an equivalence
		\begin{equation*}
			L_\proet\underline{\mathrm{Map}}(G, \_) \simeq \underline{\mathrm{Map}}(G, L_\proet\_): \mathrm{PSh}(BG_\proet) \to \mathrm{PSh}(*_\proet)
		\end{equation*}
		Restricting to $\mathrm{Shv}(*_\proet)$ on the source, we obtain that $u^*$ is corepresented by $G$, which is discrete, hence, hypercomplete.
	\end{proof}
\end{cor}
\begin{lem}\label{P-2}
	Let $G$ be a profinite group viewed as a group object in $\mathrm{Cond}\Spc$. Then the functor $BG_\proet \subset \Mod_G(\mathrm{Cond}\Set) \subset \Mod_G(\mathrm{Cond}\Spc)$ induces an equivalence
	\begin{equation}\label{Cond-5}
		\mathrm{Shv}(BG_\proet) \to \Mod_{G}(\mathrm{Cond}\Spc)
	\end{equation}
	
	Thus applying $\_ \otimes_{\mathrm{Pr}^L} \mathcal{D}$ for any $\mathcal{D} \in \mathrm{CAlg}(\mathrm{Pr}^L)$, yields a symmetric monoidal equivalence:
	\begin{align*}
		\mathrm{Shv}(BG_\proet; \mathcal{D}) \to \Mod_{f(G)}(\mathrm{Cond}\mathcal{D})
	\end{align*}
	where $f$ denotes the canonical symmetric monoidal functor $\mathrm{Cond}\Spc \to \mathrm{Cond}\mathcal{D}$.
	Here we define the left-hand side as in \cite[Def. 1.3.1.4]{Lur18} and endow the right-hand side with the monoidal structure from Fact \ref{Cond-20}.
	\begin{proof}
		Set $\mathcal{C} = \mathrm{Cond}\Spc$. We just check the conditions of relative Schwede-Shipley \cite[Prop. 4.8.5.8]{Lur17}. This proposition identifies the essential image of the fully faithful functor $\mathrm{Mod}_{\_}(\mathcal{C}): \mathrm{Alg}(\mathcal{C}) \to \mathrm{Mod}_{\mathcal{C}}(\mathrm{Pr}^L)_{\mathcal{C}/}$ defined in \cite[Cor.4.8.5.13]{Lur17}. The only condition for an object $\mathcal{M} \in \mathrm{Mod}_{\mathcal{C}}(\mathrm{Pr}^L)_{\mathcal{C}/}$ to lie in the essential image is that the right adjoint to the pointing $\mathcal{C} \to \mathcal{M}$ is conservative and colimit-preserving. In our case, we choose the pointing by $\mathrm{Cond}\Spc$ to be $G$ (not the tensor unit for the cartesian monoidal structure).
		\begin{enumerate}
			\item $\mathrm{Shv}(BG_\proet)$ is, in particular, presentably tensored over $\mathrm{Shv}(*_\proet)$ thus admits $\mathrm{Shv}(*_\proet)$-mapping objects $\underline{\mathrm{Map}}(X, Y) \in \mathrm{Shv}(*_\proet)$ for $X, Y \in \mathrm{Shv}(BG_\proet)$. So it is an object in the target of the $\mathrm{Mod}_{(\_)}(\mathcal{C})$-functor.
			\item By Corollary \ref{Cond-27} $u^* \simeq \underline{\mathrm{Map}}(G, \_)$ is conservative and colimit-preserving.
			\item The endomorphism algebra as an object in $\mathrm{Cond}\Spc$ computes to be
			\begin{equation*}
				\underline{\mathrm{Map}}(G, G) \simeq \mathrm{Map}_{\mathrm{Mod}_G(*_\proet)}(G \times \_, G) \simeq G
			\end{equation*}
			by subcanonicity of the proétale topology \cite[Lemma 4.2.4]{BS14} and the fact that representables are discrete and therefore hypercomplete objects. Here we identify $G$ with its image under the Yoneda embedding.
		\end{enumerate}
		Thus we conclude with Schwede-Shipley, that the canonical $\mathrm{Shv}(*_\proet)$-linear $\mathrm{Mod}_G(\mathrm{Shv}(*_\proet)) \to \mathrm{Shv}(BG_\proet)$ map mapping $G$ to $G$ is an equivalence.
		
		\textit{Proof of the addendum}: By \cite[Prop.4.8.1.17]{Lur17}, $\mathrm{Shv}(BG_\proet) \otimes_{\mathrm{Pr}^L} \mathcal{C}$ coincides with the definition of $\mathcal{C}$-valued sheaves in \cite[Def.1.3.1.4]{Lur18} for any $\mathcal{C} \in \mathrm{Pr}^L$ as symmetric monoidal categories.
		
		The equivalence $\Mod_{f(G)}(\mathrm{Cond}\mathcal{D}) \simeq \mathrm{Mod}_G(\mathrm{Cond}\Spc) \otimes_{\mathrm{Pr}^L} \mathcal{D}$ is a special case of \cite[Thm.4.8.4.6]{Lur17}. That the symmetric monoidal structure is the stated one follows from it being base changed from the cartesian symmetric monoidal equivalence (see \ref{H-3} for the Hopf monoidal structure being stable under basechange; see \ref{H-11} that it coincides with the cartesian monoidal structure, when the underlying category is cartesian monoidal).
	\end{proof}
\end{lem}
	\bibliographystyle{alpha}
	\bibliography{bibliography}
\end{document}